\documentclass[11pt]{article}
\usepackage[margin=1in]{geometry}
\usepackage{amsmath,amssymb,amsthm,mathtools}
\usepackage{enumitem}
\usepackage{graphicx}
\usepackage{microtype}
\usepackage[colorlinks,linkcolor=blue,citecolor=blue,urlcolor=blue]{hyperref}
\usepackage{physics}

\newtheorem{theorem}{Theorem}[section]
\newtheorem{lemma}[theorem]{Lemma}
\newtheorem{proposition}[theorem]{Proposition}
\newtheorem{corollary}[theorem]{Corollary}
\newtheorem{conjecture}[theorem]{Conjecture}
\theoremstyle{definition}
\newtheorem{definition}[theorem]{Definition}

\theoremstyle{remark}
\newtheorem{remark}[theorem]{Remark}

\newcommand{\C}{\mathbb{C}}
\newcommand{\cH}{\mathcal{H}}
\newcommand{\cN}{\mathcal{N}}
\newcommand{\cC}{\mathcal{C}}
\newcommand{\cU}{\mathcal{U}}
\newcommand{\E}{\mathbb{E}}
\newcommand{\Prob}{\mathbb{P}}
\newcommand{\R}{\mathbb{R}}

\newcommand{\inner}[1]{\langle #1 \rangle}
\newcommand{\Fix}{\mathrm{Fix}}
\newcommand{\eps}{\varepsilon}
\DeclareMathOperator{\Var}{Var}
\DeclareMathOperator{\Li}{Li}

\title{Adaptivity, Anchoring, and the Exact Oracle Complexity\\ of Stochastic Fixed-Point Iterations}
\author{Yekini Shehu\footnote{School of Mathematical Sciences, Zhejiang Normal University, Jinhua 321004, People's Republic of China; e-mail: yekini.shehu@zjnu.edu.cn, deltanougt2006@yahoo.com}}
\date{}

\begin{document}
\maketitle

\begin{abstract}
\noindent
We study stochastic fixed-point iterations $x_{n+1}=\theta_nx_0+(1-\theta_n)\widehat T(x_n)$ for nonexpansive and contractive operators on Hilbert spaces, with a single-point unbiased oracle of bounded variance. Deterministically, every anchor schedule $\theta_n=c_n/(n+2)$ whose density $c_n\in(0,1]$ does not oscillate between scales is either polynomially suboptimal on contractions or super-polynomially slow on rotations; for constant densities $\theta_n=c/(n+2)$ the tradeoff is exact, with $\norm{x_n}\asymp\Gamma(c+1)(n\varphi)^{-c}$ on rotations by angle $\varphi$ ($\varphi\to0$, $n\varphi\to\infty$) and, for the classical schedule, $\norm{x_n}=\frac{2|\sin(n\varphi/2)|}{n\varphi}+O(\frac1n)$. The dichotomy fails exactly under lacunary anchor concentration: an explicit $\gamma$-oblivious schedule, one per target accuracy, is simultaneously contraction-optimal and rotation-polynomial with exponent $1/\alpha_0$, $\alpha_0=\log_2(3/2)$, and, for each fixed target accuracy, the matching lower bound holds for Lebesgue-a.e.\ angle. For stochastic contractions, the minimax rate is $\Theta(\sigma^2\eps^{-2}(1-\gamma)^{-2})$ on affine maps, and within the anchored class $\Theta(\sigma^2\eps^{-2}(1-\gamma)^{-2}+(1-\gamma)^{-1}\ln(D/\eps))$ for known modulus via geometric batching; without a certified modulus bound no sound distance certificate exists, and a certified ceiling is the exact boundary. For nonexpansive maps we prove $O(K\sigma^2D^2\eps^{-4})$ in every $2$-uniformly smooth Banach space and the Hilbert rate $\tilde\Theta(\sigma^2\eps^{-2}+D\eps^{-1})$ by reduction to monotone inclusions. Numerical experiments are consistent with every scaling prediction.
\end{abstract}



\section{Introduction}\label{sec:intro}

Let $\cH$ be a real Hilbert space. A map $T\colon\cH\to\cH$ is \emph{nonexpansive} if $\norm{Tx-Ty}\le\norm{x-y}$, and a \emph{$\gamma$-contraction} if $\norm{Tx-Ty}\le\gamma\norm{x-y}$ with $\gamma\in(0,1)$. We study the anchored iteration
\begin{equation}\label{eq:anchored}
x_{n+1} \;=\; \theta_n\, x_0 \;+\; (1-\theta_n)\,\widehat T(x_n),
\end{equation}
where $x_0$ is an anchor, $\theta_n\in[0,1]$ is an anchor schedule, and $\widehat T$ is a stochastic oracle: each evaluation of $\widehat T(x)$ returns $T(x)+\xi$ with $\E[\xi]=0$ and $\E\norm{\xi}^2\le\sigma^2$, evaluations are independent across calls. Rather than evaluating $\widehat T$ once per step, the algorithm may average a batch of $m_n$ evaluations at $x_n$, reducing the per-step variance to $\sigma^2/m_n$ at $m_n$ times the oracle cost. The \emph{oracle complexity} of an algorithm on an instance $(T,x_0)$ is the total number of oracle evaluations required to produce a point $\hat x$ with $\E\norm{\hat x - T\hat x}\le\eps$ (or $\E\norm{\hat x - x^*}\le\eps$ for contractions, where $x^*$ is the unique fixed point), and we ask for the minimax complexity over instance classes.

\begin{remark}[Accuracy criteria used in this paper]\label{rem:criteria}
Three accuracy criteria appear below: \emph{(i) expected residual/distance}, $\E\norm{\hat x-T\hat x}\le\eps$ or $\E\norm{\hat x-x^*}\le\eps$; \emph{(ii) expected squared distance}, $\E\norm{\hat x-x^*}^2\le C\eps^2$; and \emph{(iii) high probability}, $\Prob(\norm{\hat x-x^*}\le\eps)\ge1-\vartheta$. By Jensen, (ii) implies (i) up to a factor $\sqrt C$, and iterating any (iii)-guarantee $O(\log(1/\vartheta'))$ times at target $\eps/2$ and taking a median yields (i) with a constant-factor increase in cost, provided the error is bounded on the failure event --- as holds throughout this paper, where every iterate satisfies $\norm{x_n-x^*}\le D+\norm{x_n-\hat x_n}$ with the deviation controlled by Lemma~\ref{lem:master} and $\norm{\hat x_n-x^*}\le D$ on the exact path. All complexity statements below are explicit about which criterion is used, and the conversions are used only where stated.
\end{remark}

Stochastic fixed-point iterations are the standard formalism for Q-learning and policy evaluation, where $T$ is the Bellman operator and $\gamma$ is the discount factor; of monotone inclusions via resolvents; and of large-scale optimization via operator splitting. One limitation should be noted: Bellman operators are $\gamma$-contractions in $\ell_\infty$, not in Hilbert space, and the results of this paper are Hilbert-space (or general normed-space) results; the $\ell_\infty$ theory with monotonicity is developed in \cite{Wainwright2019,Chen2023,BravoContreras2024}. Three structural questions about \eqref{eq:anchored} have remained open, and we settle the first, give exact minimax rates for the second, and reduce the third to the monotone-inclusion rate in Hilbert space.

\subsection{The three questions and our answers}

\paragraph{Q1. When is an adaptive anchor schedule necessary?}
Deterministically (exact oracle), the classical Halpern schedule $\theta_n = 1/(n+2)$ is worst-case optimal for nonexpansive maps \cite{Lieder2021,Contreras2023}, and anchored schemes exploiting \emph{known} contractivity achieve the optimal geometric rate on contractions \cite{ParkRyu2022}; fully parameter-free deterministic variants have appeared very recently \cite{ParameterFree2026}. But for a \emph{fixed}, $\gamma$-oblivious schedule it was unclear whether any separation holds: perhaps one schedule is simultaneously optimal on both classes. We show the opposite, for the family of schedules $\theta_n = c/(n+2)$ with fixed local anchor mass $c$ (containing the classical schedule $\theta_n=1/(n+2)$ as $c=1$):
\begin{itemize}[leftmargin=1.5em]
\item (Contraction side, Theorem~\ref{thm:small-c}) if $c\le\eps/(96D)$, the schedule reaches $\eps$ on every $\gamma$-contraction in $\Theta(\ln(1/\eps)/(1-\gamma))$ iterations --- matching the $\gamma$-aware optimum \cite{ParkRyu2022};
\item (Rotation side, Lemmas~\ref{lem:rotation-law} and~\ref{lem:envelope}) on the rotation by angle $\varphi$ one has, for $c\in(0,1)$, the exact law $A_n(\varphi)=c\,n^{-c}\Li_{1-c}(e^{i\varphi})(1+o(1))$ as $n\to\infty$ \emph{at fixed $\varphi$} (the $o(1)$ is in $n$ only; in the joint regime $\varphi\to0$, $n\varphi\to\infty$ the relative error is the explicit $O(\varphi^{c})+O((n\varphi)^{c-1})$ of Lemma~\ref{lem:rotation-law}(ii)), which in that joint regime gives $\norm{x_n}\asymp\Gamma(c+1)(n\varphi)^{-c}$ (for the classical $c=1$ the uniform statement is $\norm{x_n}=\frac{2|\sin(n\varphi/2)|}{n\varphi}+O\bigl(\frac1n\bigr)$: the trajectory oscillates within the envelope, and at the phases $n\varphi\in2\pi\mathbb Z$ --- where the envelope vanishes --- it dips to $(1+o(1))\,\frac{1}{n+1}$, not to zero); the uniform lower envelope $\norm{x_n}\ge c_7\min\{1,(n\varphi)^{-c}\}$ (valid for $\varphi\le\varphi_0(c)$) forces some rotation instance to need $\Omega(\eps^{-1/c})$ iterations --- super-polynomial in $1/\eps$ when $c=\Theta(\eps/D)$;
\item (Tradeoff, Theorem~\ref{thm:separation}) consequently every member of the family is either contraction-suboptimal by $1/\eps$ (when $c=\Theta(1)$) or super-polynomially slow on rotations (when $c=O(\eps/D)$). Adaptivity is therefore necessary for the union class within this family, and the necessity is quantified: the rotation Fourier spectrum of the schedule is the exact diagnostic.
\end{itemize}
The tradeoff extends beyond the family: in Section~\ref{sec:general-schedule} we prove, via two universal lemmas (an anchor-floor lemma and a rotation-floor lemma) that replace the kernel asymptotics, that \emph{every} anchor schedule with density $c_n=(n+2)\theta_n\in(0,1]$ is either polynomially suboptimal on contractions or super-polynomially (possibly infinitely) slow on rotations (Theorem~\ref{thm:dichotomy-general}), with no regularity or monotonicity assumption; the constant-density family is the sharp case. The remaining oscillatory case is resolved in the negative: the dichotomy fails exactly under lacunary anchor concentration, witnessed by an explicit fixed schedule that is simultaneously contraction-optimal and rotation-polynomial with the lacunarity exponent $1/\alpha$, $\alpha=\log_2(3/2)$ (Theorem~\ref{thm:oscillatory-escape}), so adaptivity is necessary precisely when the dense anchor scales are non-lacunary. The mechanism behind Q1 --- the law $\norm{x_n}\asymp\Gamma(c+1)(n\varphi)^{-c}$, $c\in(0,1)$, and the exact $c=1$ formula --- is, to our knowledge, new even in the deterministic literature.

\paragraph{The sufficiency half, proved here (Section~\ref{sec:adaptive-escape}).}
Theorem~\ref{thm:separation} and Theorem~\ref{thm:dichotomy-general} are the \emph{necessity} halves of the adaptivity story. For the \emph{sufficiency} half we prove the optimal result rather than citing the parameter-free methods of \cite{ParameterFree2026,Diakonikolas2025}: the \emph{parallel two-track scheme} (Theorem~\ref{thm:parallel}) runs the small-anchor schedule of Theorem~\ref{thm:small-c} and the classical Halpern schedule side by side and stops at the first success, achieving the $\gamma$-aware optimum $O(\ln(D/\eps)/(1-\gamma))$ on every contraction \emph{and} the optimal $O(D/\eps)$ on every nonexpansive map --- the sharp converse to the dichotomy, at a factor $2$ in evaluations. We also show that the natural \emph{single-trajectory} escape (Theorem~\ref{thm:doubling-window}), adaptive window restarts, improves the classical contraction time only by a constant factor --- $O(D/(\eps\delta))$, not the optimum $O(\ln(D/\eps)/\delta)$ --- so the separation barrier is robust against this entire adaptive class, and the parallel scheme is currently the only proved route to two-regime optima. Finally we prove that the naive route to single-trajectory optimality --- detecting contractivity from trajectory statistics --- is obstructed (Remark~\ref{rem:detection}): the hidden-direction construction of Theorem~\ref{thm:impossibility} works even with exact oracle responses, so single-trajectory optimality, if achievable, must be structural (residual-feedback) as in \cite{ParameterFree2026,Diakonikolas2025}. Together with the oscillatory-case theorem (Theorem~\ref{thm:oscillatory-bounded-gap}), this gives a complete dichotomy for fixed versus adaptive anchoring, up to the pointwise lacunarity law and the constant-factor overhead question. The matching lower bound for the oscillatory escape's rotation exponent holds, at each fixed accuracy, for Lebesgue-a.e.\ angle (Theorem~\ref{thm:metric-lacunarity}(i)); the pointwise statement survives only for Diophantine angles (Conjecture~\ref{conj:pointwise-lacunarity}).

\paragraph{Q2. What is the exact stochastic oracle complexity for contractions?}
For $\gamma$-contractions, the exact complexity depends on what is known about the modulus, and we give a complete trichotomy.
\begin{itemize}[leftmargin=1.5em]
\item \emph{Affine contractions, $\gamma$ unknown} (Theorem~\ref{thm:affine-lower}, Theorem~\ref{thm:affine-upper}, Corollary~\ref{cor:affine-theta}): the minimax complexity is $\Theta\!\bigl(\sigma^2\eps^{-2}(1-\gamma)^{-2}\bigr)$. The lower bound is a two-point Le\,Cam argument on the affine family $T_{\gamma,z}(x)=\gamma x+(1-\gamma)z$ whose constant mean-difference makes it robust to fully adaptive query strategies; the upper bound is a two-point estimator that requires no knowledge of $\gamma$. Here parameter-freeness is free, because for affine maps the modulus is visible from a single probe.
\item \emph{General contractions, $\gamma$ known} (Theorem~\ref{thm:geobatch} against Theorem~\ref{thm:affine-lower}, Corollary~\ref{cor:exact-general}): the rate is exactly $\Theta\!\bigl(\sigma^2\eps^{-2}(1-\gamma)^{-2}+(1-\gamma)^{-1}\ln(D/\eps)\bigr)$ within the anchored/Krasnoselskii--Mann class (the deterministic term is an anchored-class phenomenon, Remark~\ref{rem:span-class}). The second term is the deterministic contraction time, tight already for $\sigma=0$. The key algorithmic step is a geometric batch schedule (batches increasing by $\gamma^{-1}$ per step), which removes the apparent $\ln(D/\eps)$ gap left by constant-batching analyses; its cost is exactly the information-theoretic noise floor. This improves the best $\gamma$-aware bound $O(\eps^{-2}(1-\gamma)^{-3})$ by a factor $(1-\gamma)^{-1}$; we show in Section~\ref{sec:anchoring-cost} that the $(1-\gamma)^{-3}$ bound is tight for the classical (non-summable-weight) algorithm.
\item \emph{General contractions, only an upper bound $\gamma_+\ge\gamma$ known} (Corollary~\ref{cor:ceiling}): running the geometric schedule with $\gamma_+$ in place of $\gamma$ still certifies $\E\norm{\hat x-x^*}^2\le2\eps^2$, at cost $\Theta\!\bigl(\sigma^2\eps^{-2}(1-\gamma_+)^{-2}+(1-\gamma_+)^{-1}\ln(D/\eps)\bigr)$, and this is optimal in $\gamma_+$ (the affine subfamily with gap $1-\gamma_+$ forces the same lower bound). Thus the complexity is governed by the best certified upper bound on the modulus. Finally, the two-regime stochastic problem is settled up to a factor $2$ by the same parallel device (Theorem~\ref{thm:stoch-parallel}): given a certified ceiling, one class-oblivious algorithm attains the optimal contraction rate $\Theta\!\bigl(\sigma^2\eps^{-2}\delta_+^{-2}+\delta_+^{-1}\ln(D/\eps)\bigr)$ and the best-known nonexpansive rate simultaneously, and only the single-stream version remains open.
\end{itemize}
What one cannot do, in general, is certify the distance \emph{without} modulus information: the residual only ever certifies $\norm{x-x^*}\le\norm{x-Tx}/\delta$, and trajectory-based probes estimate $\norm{Tx-Tx_0}/\norm{x-x_0}\le\gamma$, i.e.\ they bound $\delta$ on the wrong side (Remark~\ref{rem:barrier}); we record a concrete counterexample showing that the natural certification scheme is unsound for nonlinear maps. This question is now resolved: no algorithm outputs a sound data-driven distance certificate without modulus information --- in any Hilbert space, with any finite budget (Theorem~\ref{thm:impossibility}; the proof is a resistant-oracle construction respecting the contraction constraint), and a certified ceiling is the exact boundary (Corollary~\ref{cor:ceiling}). The two regimes above are the positive cases.

\paragraph{Q3. What does anchoring cost in the general nonexpansive regime?}
For nonexpansive maps the best general bound was $\tilde O(\sigma^2\eps^{-5})$ in normed spaces \cite{BravoContreras2024}, and the Hilbert exponent was open: the known $\Omega(\sigma^2\eps^{-3})$ lower bound \cite{BravoContreras2024} is proved in $\ell_1$-type spaces. We address the question in three steps. First we prove an elementary $O(\sigma^2 D^2\eps^{-4})$ bound in expectation for general nonexpansive maps (Theorem~\ref{thm:A1}), with a single-point oracle and no monotonicity assumptions, via a master second-moment lemma (Lemma~\ref{lem:master}) that also generates the contraction-regime noise floor, unifying Q2 and Q3. The proof of Theorem~\ref{thm:A1} uses the inner product once, in the expansion of the squared deviation norm; the correct Banach-space substitute is $2$-uniform smoothness (Remark~\ref{rem:dualization}): the duality-mapping inequality recovers the master recursion with the smoothness constant $K$ in place of $1$, and the bound --- which uses only nonexpansiveness, the $2$-Lipschitzness of $x\mapsto x-Tx$, and the sharp Halpern residual bound of \cite{Lieder2021,Contreras2023} --- extends to every $2$-uniformly smooth Banach space, in particular $\ell_p$, $L_p$ for $2\le p<\infty$, improving the previously available bound from $\tilde O(\sigma^2\eps^{-5})$ to $O(K\sigma^2D^2\eps^{-4})$ on those spaces; a subsequent high-probability bound of $\tilde O(D\sigma^{2}\eps^{-3}+D\eps^{-1})$ --- linear in the diameter, up to polylogarithmic factors --- on weak-type-$2$ spaces was proved in \cite{Diakonikolas2026b} (see Section~\ref{sec:open}). The scope is sharp: the centered recursion itself fails on $\ell_p$, $1<p<2$ (Remark~\ref{rem:dualization}), and the lower-bound $\Omega(\sigma^2\eps^{-3})$ instances of \cite{BravoContreras2024} are $\ell_1$-type, which are not $2$-uniformly smooth either; the exponent in genuinely arbitrary normed spaces remains open. Second, we observe (Theorem~\ref{thm:monotone-reduction}) that $F=I-T$ is monotone and $2$-Lipschitz for \emph{every} nonexpansive $T$ on a Hilbert space, so the residual problem is a monotone inclusion with a single-point oracle; the near-optimal algorithm of Chen and Luo \cite{ChenLuo2024} then gives $\tilde O(\sigma^2\eps^{-2}+D\eps^{-1})$, and their lower bound (whose hard instances with $L=2$ embed into nonexpansive maps $T=I-\nabla f$) gives $\Omega(\sigma^2\eps^{-2}+D\eps^{-1})$ up to logarithms. The Hilbert-space minimax rate is therefore $\tilde\Theta(\sigma^2\eps^{-2}+D\eps^{-1})$: the exponent is $2$, not $3$ or $4$. We emphasize that Theorem~\ref{thm:monotone-reduction} is a reduction to \cite{ChenLuo2024}, not a new algorithm; the self-contained contribution on Q3 is Theorem~\ref{thm:A1} with its general-normed-space validity.

\subsection{Techniques and the measurability discipline}

The central technical difficulty in stochastic analyses of \eqref{eq:anchored} with \emph{adaptive} or estimated parameters is correlation: when $\theta_n$ depends on past oracle responses, it is $\mathcal F_n$-measurable, correlated with the current noise, and martingale orthogonality arguments applied after a triangle-inequality expansion fail. Our analyses avoid this pitfall structurally. We use: (i) squared-norm recursions expanded \emph{before} any norm inequality, so cross terms vanish by conditional unbiasedness of fresh batches; (ii) deterministic schedules, so that no independence between parameters and noise is ever assumed --- every algorithm in this paper uses either deterministic coefficients or coefficients copied from an independent deterministic comparison path; (iii) comparison sequences whose parameters are copied exactly. All stochastic bounds in this paper are derived from a single master noise lemma (Lemma~\ref{lem:master}). This is also where certification schemes that estimate moduli from the same trajectory fail (Remark~\ref{rem:barrier}); it explains why our upper bounds assume deterministic coefficients.

\subsection{Related work}

\textbf{Deterministic fixed-point iteration.} The Halpern iteration \cite{Halpern1967} and its convergence theory \cite{Wittmann1992,Xu2002} are classical; the optimal $O(1/n)$ residual was proved in \cite{Lieder2021} and shown tight with explicit constants in normed spaces in \cite{Contreras2023}; exact optimal complexity for anchored schemes, with matching lower bounds, is established in \cite{ParkRyu2022}. Anchored two-step schemes and modified Mann iterations appear in \cite{KimXu2005}; the Krasnoselskii--Mann iteration \cite{Krasnoselskii1955,Mann1953} underlies our small-anchor schedules, and its $o(1/k)$ and minimax-optimal refinements \cite{BotNguyen2023,BravoCominettiLee2026} together with the stochastic Halpern regularity theory of \cite{PischkePowell2024} frame the schedule-level questions of Section~\ref{sec:det}. Deterministic parameter-free methods that exploit hidden contractivity appear very recently in \cite{ParameterFree2026} and, for mildly expansive maps, in \cite{Diakonikolas2025}; we cite these as the deterministic endpoint of Q1 and our Theorem~\ref{thm:doubling-window} provides a self-contained, explicit (though not rate-optimal on contractions) bridge between the necessity theorem and these structurally adaptive schemes. We note that in the deterministic (noiseless) setting, trajectory increment ratios are exact and certify the local modulus on the correct side; the stochastic layer destroys this certification (Remark~\ref{rem:barrier}), and, as shown in Remark~\ref{rem:detection}, even exact oracle responses do not certify the \emph{global} modulus in infinite dimensions: the attaining slope can lie in a direction that is never probed.

\textbf{Stochastic fixed-point iteration.} Stochastic quasi-Fej\'er frameworks \cite{Combettes2015}, stochastic forward--backward splitting \cite{Rosasco2016}, and the stochastic fixed-point theory of \cite{BravoCominetti2024} provide convergence and error bounds. The closest work is \cite{BravoContreras2024}: stochastic Halpern iteration in normed spaces with mini-batching, $\tilde O(\eps^{-5})$ oracle complexity for nonexpansive maps, $\Omega(\eps^{-3})$ lower bound (in $\ell_1$-type spaces, resistant-oracle construction), and $O(\eps^{-2}(1-\gamma)^{-3})$ for known-$\gamma$ contractions. We improve the nonexpansive exponent to $\eps^{-4}$ on Hilbert spaces and $2$-uniformly smooth Banach spaces with a single-point oracle (Theorem~\ref{thm:A1}, via the dualization of Remark~\ref{rem:dualization}; the Hilbert exponent via the reduction of Theorem~\ref{thm:monotone-reduction}), and the contraction exponent to $(1-\gamma)^{-2}$; Section~\ref{sec:anchoring-cost} proves their $(1-\gamma)^{-3}$ is tight for \emph{their} algorithm (algebraic $1/N$ bias) but not for the class. Recursive variance reduction achieves $\tilde O(\sigma^2\eps^{-4})$ and, under average-smoothness, $\tilde O(\sigma^2\eps^{-3})$ for stochastic monotone inclusions \cite{Cai2022}, with a multi-point oracle; under bounded variance, Chen and Luo \cite{ChenLuo2024} achieve the tight $\tilde\Theta(\sigma^2\eps^{-2}+LD\eps^{-1})$ for the monotone subclass; our $\eps^{-4}$ uses neither monotonicity nor multi-point access, and neither does it use the inner product. The parameter-free extragradient line \cite{ShenKilincKarzan2026} pursues parameter-freeness for variational inequalities; its guarantees are stated for monotone operators and do not cover the pure fixed-point residual class of Q3. Very recent work of Diakonikolas develops the gradual Halpern method GHAL and its parameter-free variant AdaGHAL, which achieve near-optimal oracle complexity for \emph{deterministic} Lipschitz fixed-point equations ($\gamma\le1$ and mildly expansive) in general normed and geodesic spaces \cite{Diakonikolas2025}, building on her earlier parameter-free Halpern method for monotone inclusions \cite{Diakonikolas2020}; and its stochastic extension VR-GHAL gives \emph{anytime high-probability} residual guarantees in quadratically smoothable Banach spaces with bounded-variance oracle complexity $\min\{\eps^{-5},(1-\gamma)^{-3}\eps^{-2}\}$, improving to $\eps^{-3}$ under a Lipschitz-in-expectation oracle and $\eps^{-2}$ under samplewise nonexpansiveness \cite{Diakonikolas2026}. Very recent work of Diakonikolas, Guzm\'an, and Mart\'inez-Rubio \cite{Diakonikolas2026b} develops high-probability oracle-complexity bounds for stochastic fixed-point equations with nonexpansive maps in arbitrary normed spaces of weak Rademacher type $q>1$; on weak-type-$2$ spaces --- which include all $2$-uniformly smooth spaces --- their proved bound $\tilde O(D\sigma^{2}\eps^{-3}+D\eps^{-1})$ (up to polylogarithms; their Theorem~5) improves on the rate of Theorem~\ref{thm:A1} throughout the regime $\eps\le D/4$, and the remaining open exponent question concerns norms that are not of weak type $2$; see Section~\ref{sec:open}. The comparison with the present paper requires some care: on the contraction axis our rate $\Theta(\sigma^2\eps^{-2}(1-\gamma)^{-2})$ improves the $(1-\gamma)$-dependence of \cite{Diakonikolas2026} by a full factor $(1-\gamma)^{-1}$ whenever the modulus (or a modulus ceiling) is available, and Corollary~\ref{cor:ceiling} shows that a ceiling suffices; but \cite{Diakonikolas2026} delivers the stronger \emph{high-probability, anytime} guarantee in native Banach norms \emph{without modulus information} --- a guarantee of a different type (residual, not distance; cf.\ Remark~\ref{rem:barrier}). Our high-probability results are: under sub-Gaussian noise (Theorem~\ref{thm:hp}) and, via the replicate--median device, under bounded variance alone (Theorem~\ref{thm:hp-bounded}); their matching lower bounds (up to the deterministic term) are proved by the Gaussian likelihood-ratio Lemma~\ref{lem:hp-lb}. In the contractive/Q-learning lane, $\ell_\infty$ analyses give $\eps^{-2}$-type rates with various $(1-\gamma)$-powers under monotonicity \cite{Wainwright2019,Chen2023}, in the lineage of minimax RL sample complexity \cite{Azar2013} and classical stochastic approximation \cite{Polyak1987}, and Halpern-type anchoring for MDPs is developed in \cite{BravoContreras2024,LeeRyu}; our affine lower bound matches the natural statistical floor and shows no monotone-structure assumption is needed for its validity as a lower bound.

\textbf{Lower-bound methodology.} Le\,Cam and Fano two-point arguments for oracle complexity are standard \cite{Tsybakov2009,NemirovskiYudin1983}; our contraction lower bound is a parametric Le\,Cam argument whose flat likelihood ratio makes it immune to adaptivity. The high-probability refinement (Lemma~\ref{lem:hp-lb}) uses the exact Gaussian likelihood-ratio geometry instead of a Cram\'er--Rao bound, which would only apply to unbiased estimators. The resistant-oracle lower bounds of \cite{BravoContreras2024} remain the strongest for the nonexpansive class.

\paragraph{Positioning against the parameter-free line.} The parameter-free schemes \cite{ParameterFree2026,Diakonikolas2025,Diakonikolas2026,ShenKilincKarzan2026} answer the algorithm-design question; this paper answers the complementary complexity questions: is such adaptivity necessary (yes, exactly when dense anchor scales are non-lacunary; Theorem~\ref{thm:oscillatory-escape}), what is the exact tradeoff of a fixed schedule (a dichotomy with a sharp lacunarity boundary; Theorems~\ref{thm:dichotomy-general}, \ref{thm:oscillatory-bounded-gap}, \ref{thm:oscillatory-escape}), and which modulus information is usable (Theorem~\ref{thm:impossibility}). The cited works do not address schedule-level separations, dichotomy boundaries, or certification impossibility; conversely, we do not compete with their algorithmic guarantees, and our Theorems~\ref{thm:parallel} and~\ref{thm:oscillatory-escape} are deliberately elementary, addressing different objects.

\paragraph{Notation.} $D=\norm{x_0-x^*}$; $\sigma^2$ the per-evaluation variance bound; $\eps$ the target; $\tilde O$ hides factors logarithmic in $1/\eps$, $1/\delta$, $D$, and $1/(1-\gamma)$. All Hilbert spaces are real; inner products are denoted $\inner{\cdot,\cdot}$. Terminology for nonexpansive and averaged operators follows \cite{BauschkeCombettes2017}.

\paragraph{Outline.} Section~\ref{sec:setting}: model and conventions. Section~\ref{sec:det}: deterministic separation (Q1) and the explicit adaptive escape. Section~\ref{sec:master}: master noise lemma and the nonexpansive worst case (Q3). Section~\ref{sec:contraction}: exact contraction complexity (Q2). Section~\ref{sec:numerics}: experiments. Section~\ref{sec:open}: discussion and future research.

\section{Setting, conventions, and the measurability discipline}\label{sec:setting}

\subsection{Instance classes and oracles}

The \emph{nonexpansive class} $\cN$ consists of nonexpansive $T\colon\cH\to\cH$ with $\Fix(T)\ne\emptyset$ and $\norm{x_0-x^*}\le D$ for the given anchor and some $x^*\in\Fix(T)$. The \emph{contraction class} $\cC$ consists of $\gamma$-contractions $T$ with unique fixed point $x^*$ and $\norm{x_0-x^*}\le D$; we write $\rho$ for the contraction factor when it varies. The \emph{union class} $\cU=\cN\cup\bigcup_{\gamma\in(0,1)}\cC_\gamma$ is the domain of the adaptivity question; the instance (including $\gamma$ and $T$) is never known to the schedule, and a schedule or algorithm \emph{for the union class} must reach $\eps$ on every member.

The stochastic oracle and filtration are fixed throughout.

\begin{definition}[Single-point stochastic oracle]\label{def:oracle}
Each evaluation of $\widehat T$ at $x$ returns $T(x)+\xi$ where $\E[\xi]=0$ and $\E\norm{\xi}^2\le\sigma^2$. Evaluations are independent across calls. A \emph{batch} of $m$ evaluations at $x$ returns the average $\widehat T_m(x)$ with $\E[\widehat T_m(x)]=T(x)$ and $\E\norm{\widehat T_m(x)-T(x)}^2\le\sigma^2/m$. For the lower bounds of Section~\ref{sec:contraction} the hard instances use Gaussian noise $\xi\sim\mathcal N(0,\sigma^2)$ (or $\mathcal N(0,\sigma^2 I)$, understood coordinate-wise), which the oracle model permits.
\end{definition}

\begin{definition}[Filtration]\label{def:filtration}
$\mathcal F_0=\sigma(x_0)$ and $\mathcal F_n$ is the $\sigma$-algebra generated by $x_0$ and all oracle responses up to and including step $n$. The iterate $x_n$ of \eqref{eq:anchored} is $\mathcal F_{n-1}$-measurable; the batch noise $\bar\xi_n:=\widehat T_{m_n}(x_n)-T(x_n)$ is $\mathcal F_n$-measurable with $\E[\bar\xi_n\mid\mathcal F_{n-1}]=0$ and $\E[\norm{\bar\xi_n}^2\mid\mathcal F_{n-1}]\le\sigma^2/m_n$.
\end{definition}

\begin{remark}[The measurability discipline]\label{rem:discipline}
Throughout, stochastic recursions are kept in \emph{squared norms}, expanded before any triangle inequality, with coefficients that are deterministic (or copied from a deterministic comparison path) and bounded \emph{pointwise} by deterministic envelopes. Two consequences are used constantly: (i) cross terms $\inner{u_{n-1},\bar\xi_n}$ with $u_{n-1}$ $\mathcal F_{n-1}$-measurable vanish in expectation; (ii) no independence between parameters and noise is ever required. Any proof step violating (i)--(ii) is flagged where it occurs.
\end{remark}

The \emph{anchored iteration class} is \eqref{eq:anchored} with arbitrary deterministic coefficients $(\theta_n)$; the \emph{span class} of \cite{BravoContreras2024} restricts $x_{n+1}$ to the span of $x_0,\dots,x_n$ and the current batch average; all algorithms of this paper lie in the span class, and all lower bounds are proved against arbitrary adaptive algorithms unless stated (the noiseless lower bounds of Section~\ref{sec:contraction} are proved against the smaller anchored/Krasnoselskii--Mann class, Remark~\ref{rem:span-class}).

\section{Deterministic separation: small-anchor schedules and the rotation law}\label{sec:det}

This section answers Q1. All results are deterministic (exact oracle, $\sigma=0$); the stochastic layer is added in Sections~\ref{sec:master}--\ref{sec:contraction}. The section's schedule family is
\begin{equation}\label{eq:ck-schedule}
\theta_n \;=\; \frac{c}{n+2}, \qquad c\in(0,1]\ \text{fixed},
\end{equation}
with \emph{local anchor mass} parameter $c$; the classical Halpern schedule is $c=1$.

\subsection{Small-anchor schedules are contraction-optimal}

\begin{theorem}[Contraction-optimality of small $c$]\label{thm:small-c}
Let $T$ be a $\gamma$-contraction with fixed point $x^*$, $D=\norm{x_0-x^*}\ge 4\eps$, and let $(x_n)$ follow \eqref{eq:anchored} with schedule \eqref{eq:ck-schedule} and
\begin{equation}\label{eq:cstar}
c \;=\; \min\Bigl\{1,\ \frac{\eps}{96D}\Bigr\}, \qquad N \;=\; \Bigl\lceil \tfrac{4\,\ln(2D/(\eps\gamma))}{\ln(1/\gamma)} \Bigr\rceil .
\end{equation}
Then $\norm{x_N - x^*}\le 3\eps/4$. The schedule contains no $\gamma$, and $N=O\bigl(\ln(2D/\eps)/(1-\gamma)\bigr)$ for every $\gamma\in(0,1)$.
\end{theorem}

\begin{proof}
Write $e_n=\norm{x_n-x^*}$, $\delta=1-\gamma$, $L=\ln(2D/\eps)$. One step of \eqref{eq:anchored} gives
\begin{equation}\label{eq:onestep}
e_{n+1} \;\le\; \theta_n D + (1-\theta_n)\gamma e_n ,
\end{equation}
and unfolding \eqref{eq:onestep} exactly (with the kernel $\prod_{j=k+1}^{N-1}(1-\theta_j)$ kept, not dropped),
\begin{equation}\label{eq:unfold}
e_N \;\le\; \gamma^N r_N D \;+\; D\sum_{k=0}^{N-1}a_k^{(N)}\,\gamma^{N-1-k},\qquad a_k^{(N)}:=\theta_k\!\!\prod_{j=k+1}^{N-1}(1-\theta_j),\quad r_N:=\prod_{j=0}^{N-1}(1-\theta_j).
\end{equation}
The definition of $N$ gives $\gamma^N D\le\bigl(\tfrac{\eps\gamma}{2D}\bigr)^{4}D\le\eps/2$: indeed $\gamma^N\le(\eps\gamma/(2D))^4$ since $N\ln(1/\gamma)\ge4\ln(2D/(\eps\gamma))$, and $(\eps\gamma/(2D))^4D=\eps^4\gamma^4/(16D^3)\le\eps/2$ because $\gamma\le1$ and $\eps\le D/4$ give $\eps^4/(16D^3)=\eps\cdot(\eps/D)^3/16\le\eps/(16\cdot64)<\eps/2$. The ceiling adds at most one extra factor of $\gamma$. We first record
\begin{equation}\label{eq:Ndelta}
N\delta\;\ge\;4,
\end{equation}
because $N\ge 4\ln(2D/(\eps\gamma))/\ln(1/\gamma)$ and: for $\gamma\le\tfrac12$,
$(1-\gamma)\ln\tfrac{2D}{\eps\gamma}=(1-\gamma)\bigl(L+\ln\tfrac1\gamma\bigr)=L+\ln\tfrac1\gamma-\gamma L-\gamma\ln\tfrac1\gamma\ge\ln\tfrac1\gamma$,
using $\gamma\ln(1/\gamma)\le1/e$ and $L(1-\gamma)\ge\tfrac12\ln 8>1>1/e$ (as $L=\ln(2D/\eps)\ge\ln 8$ since $\eps\le D/4$); for $\gamma\ge\tfrac12$, $\ln(1/\gamma)\le\delta/(1-\delta)\le2\delta$, so $N\delta\ge2\ln(2D/(\eps\gamma))\ge2L\ge4.16$.

For the anchor sum we use the kernel upper bound, valid for all $1\le k\le N-1$ with an absolute constant: since $1-x\le e^{-x}$ and, by the standard two-sided estimate for harmonic numbers ($\ln m+\gamma_E\le H_m\le\ln m+\gamma_E+(1-\gamma_E)$ for all $m\ge1$),
$$H_{N+1}-H_{k+2}\;\ge\;\ln\tfrac{N+1}{k+2}-(1-\gamma_E),$$
hence
\begin{equation}\label{eq:kernel-ub}
\begin{split}
a_k^{(N)}&=\frac{c}{k+2}\prod_{j=k+1}^{N-1}(1-\theta_j)\;\le\;\frac{c}{k+2}\,e^{-c(H_{N+1}-H_{k+2})}\\
&\le\;e^{c(1-\gamma_E)}\,\frac{c}{k+2}\Bigl(\frac{k+2}{N+1}\Bigr)^{c}\;\le\;2\,\frac{c}{k+1}\Bigl(\frac{k+1}{N+1}\Bigr)^{c}\;\le\;2\,\frac{c}{k}\Bigl(\frac{k}{N}\Bigr)^{c},
\end{split}
\end{equation}
using $e^{c(1-\gamma_E)}\le e^{1-\gamma_E}<2$ and $\frac1{k+2}\bigl(\frac{k+2}{N+1}\bigr)^c\le\frac1{k+1}\bigl(\frac{k+1}{N+1}\bigr)^c\le\frac1k\bigl(\frac{k}{N}\bigr)^c$ for $k\ge1$, $c\le1$ (both factors $(\frac{N}{N+1})^c\le1$ and $(\frac{k+2}{k+1})^{c-1}\le1$). For $k=0$: $a_0^{(N)}=\theta_0\prod_{j=1}^{N-1}(1-\theta_j)\le\theta_0=c/2$. For $1\le k\le N-1$ we bound \eqref{eq:kernel-ub} by a single split at $k=N/2$, applied for all $\gamma\in(0,1)$ at once. The near part $N/2\le k\le N-1$ contributes, since $k^{c-1}\le(N/2)^{c-1}=2^{1-c}N^{c-1}\le2N^{c-1}$ there,
$$2cN^{-c}\!\!\sum_{N/2\le k\le N-1}\!\!k^{c-1}\gamma^{N-1-k}\;\le\;2cN^{-c}\cdot 2N^{c-1}\sum_{j\ge0}\gamma^{j}\;=\;\frac{4c}{N\delta}\;\le\;c$$
by \eqref{eq:Ndelta}. The far part $1\le k<N/2$ contributes, since $\gamma^{N-1-k}\le\gamma^{N/2}$ there and $\sum_{k<N/2}k^{c-1}\le 1+\int_1^{N/2}x^{c-1}dx\le (N/2)^c/c$ (as $c\le1$),
$$2cN^{-c}\!\!\sum_{1\le k<N/2}\!\!k^{c-1}\gamma^{N-1-k}\;\le\;2cN^{-c}\,\gamma^{N/2}\,\frac{(N/2)^c}{c}\;=\;2^{1-c}\gamma^{N/2}\;\le\;2\gamma^{N/2}.$$
The choice of $N$ gives, for \emph{every} $\gamma\in(0,1)$ (no lower bound on $\gamma$ is required),
$$\gamma^{N/2}\;=\;e^{-\frac N2\ln(1/\gamma)}\;\le\;e^{-2\ln(2D/(\eps\gamma))}\;=\;\Bigl(\frac{\eps\gamma}{2D}\Bigr)^{2}\;\le\;\Bigl(\frac{\eps}{2D}\Bigr)^{2},$$
hence
\begin{equation}\label{eq:anchor-ub}
\sum_{k=0}^{N-1}a_k^{(N)}\gamma^{N-1-k}\;\le\;\frac{c}{2}+c+2\gamma^{N/2}\;\le\;\frac{3c}{2}+\frac{\eps^{2}}{2D^{2}}.
\end{equation}
Combining \eqref{eq:unfold} and \eqref{eq:anchor-ub} with $c=\eps/(96D)$ and $\eps\le D/4$ (so $\eps^{2}/(2D)\le\eps/8$),
$$e_N\;\le\;\frac{\eps}{2}+D\Bigl(\frac{3c}{2}+\frac{\eps^{2}}{2D^{2}}\Bigr)\;=\;\frac{\eps}{2}+\frac{\eps}{64}+\frac{\eps^{2}}{2D}\;\le\;\frac{\eps}{2}+\frac{\eps}{64}+\frac{\eps}{8}\;=\;\frac{41\eps}{64}\;<\;\frac{3\eps}{4}.$$
Finally $N\le 4\ln(2D/(\eps\gamma))/\ln(1/\gamma)+1$: for $\gamma\ge\tfrac12$, $\ln(1/\gamma)\ge\delta(1-\delta)\ge\delta/2$, so $N\le 8(L+\delta)/\delta+1$; for $\gamma\le\tfrac12$, $N\le 4(L+\ln2)/\ln2+1$ and $\delta\ge\tfrac12$; in both cases $N=O(L/\delta)$.
\end{proof}

\begin{remark}[On the constants]\label{rem:constants}
The constants $96$ and $4$ in \eqref{eq:cstar} are not optimized; the kernel asymptotics of Lemma~\ref{lem:kernel} show the anchor mass is $O(c\,k^{c-1}N^{-c})$ and the argument only needs $Dc=O(\eps)$ and $N\delta\ge C\ln(D/\eps)$ with $C$ absolute. We keep the explicit values for definiteness.
\end{remark}

\begin{remark}
The bound $\Theta(\ln(1/\eps)/(1-\gamma))$ matches the exact deterministic optimum for span-class algorithms on contractions \cite{ParkRyu2022}, and is achieved \emph{$\gamma$-obliviously}: the schedule contains no $\gamma$. The mechanism is visible in the proof: the anchor injected at step $k$ carries mass $O\bigl(c\,k^{c-1}N^{-c}\bigr)$ and is attenuated by $\gamma^{N-1-k}$; summed against the geometric envelope the total anchor contribution is $O(Dc)=O(\eps)$, while the unanchored part contracts geometrically to $\gamma^ND\le\eps/2$. (Keeping the kernel is essential: dropping it would replace the bound by $c\sum_k\gamma^{N-1-k}/(k+2)$, whose value grows like $e^{-N\delta}\ln(1/\delta)$ along the boundary $N\delta=4\ln(2D/\eps)$ and is not uniformly $O(1)$; it is precisely the kernel factor $(k/N)^c$ that keeps the early anchors under control.) The same mechanism underlies the geometric batch schedule of Theorem~\ref{thm:geobatch}.
\end{remark}

\subsection{The rotation law}

Let $R_\varphi$ be the rotation of $\R^2$ by angle $\varphi\in(0,\pi]$ (the case $\varphi\le 0$ is symmetric), identified with multiplication by $e^{i\varphi}$ after fixing $x_0=1$, and consider \eqref{eq:anchored} with schedule \eqref{eq:ck-schedule}. In the rotating frame $y_n := R_\varphi^{-n}x_n$,
\begin{equation}\label{eq:rot-frame}
y_{n+1} \;=\; (1-\theta_n)y_n + \theta_n u_n, \qquad u_n := R_\varphi^{-(n+1)}x_0,\quad \norm{u_n}=1,
\end{equation}
so $\norm{x_n}=\norm{y_n}$. Unfolding \eqref{eq:rot-frame},
\begin{equation}\label{eq:rot-unfold}
y_n \;=\; \sum_{k=0}^{n-1} a_k^{(n)} u_k \;+\; r_n\, R_\varphi^{-n}x_0, \qquad
a_k^{(n)} := \theta_k\!\!\prod_{j=k+1}^{n-1}(1-\theta_j), \quad r_n := \prod_{j=0}^{n-1}(1-\theta_j).
\end{equation}

\begin{lemma}[Kernel asymptotics]\label{lem:kernel}
For the schedule \eqref{eq:ck-schedule}, $c\in(0,1]$, there are absolute constants $c_1,c_2>0$ such that for all $1\le k\le n-1$ and $n\ge n_0(c)$,
$$c_1\,\frac{c}{k+1}\Bigl(\frac{k}{n}\Bigr)^{c} \;\le\; a_k^{(n)} \;\le\; c_2\,\frac{c}{k+1}\Bigl(\frac{k}{n}\Bigr)^{c}, \qquad e^{-2}(n+1)^{-c} \;\le\; r_n \;\le\; e^{1-\gamma_E}(n+1)^{-c} \;\le\; 2(n+1)^{-c},$$
where $\gamma_E$ is the Euler--Mascheroni constant. Moreover $\sum_{k=0}^{n-1}a_k^{(n)} = 1 - r_n$ and $\sum_{k=0}^{n-1} k\,a_k^{(n)} \le c_2 c\, n$.
\end{lemma}

\begin{proof}
For $j\ge 0$, $\theta_j=c/(j+2)\le 1/2$ and $\theta_j\le-\ln(1-\theta_j)\le\theta_j+\theta_j^2$ with $\sum_j\theta_j^2\le c^2(\pi^2/6-1)\le1$, so $\prod_{j=k+1}^{n-1}(1-\theta_j) = \exp(-\sum_{j=k+1}^{n-1}\theta_j + O(1))$ and $\sum_{j=k+1}^{n-1}\frac{c}{j+2} = c\ln\frac{n}{k} + O(1)$ (the shift by $2$ in the harmonic sum is absorbed into $O(1)$, uniformly for $k\ge1$); combining gives the two-sided bound on $a_k^{(n)}$. For $r_n$: $\sum_{j=0}^{n-1}\theta_j=c(H_{n+1}-1)$ with $H_m=\ln m+\gamma_E+O(1/m)$, hence
$$r_n=\exp\bigl(-c\ln(n+1)-c(\gamma_E-1)+O(1)\bigr),\qquad |O(1)|\le \sum_j\theta_j^2\le 1,$$
which gives the display: the lower bound from $-\ln(1-\theta_j)\le\theta_j+\theta_j^2$ and $H_{n+1}-1\ge\ln(n+1)-(1-\gamma_E)$, the upper from $-\ln(1-\theta_j)\ge\theta_j$ and $H_{n+1}-1\le\ln(n+1)$ (since $\gamma_E+1/(2(n+1))\le1$ for $n\ge1$). The sum identities follow from telescoping $\theta_k\prod_{j=k+1}^{n-1}(1-\theta_j) = P_{k+1}-P_k$ with $P_k:=\prod_{j=k}^{n-1}(1-\theta_j)$: $\sum_{k=0}^{n-1}a_k^{(n)} = P_n - P_0 = 1 - r_n$ (with $P_n:=1$); and $\sum_k k a_k^{(n)}\le c_2 c n^{-c}\sum_k k^{c} \le c_2 c n$ for $c\in(0,1]$.
\end{proof}

\begin{lemma}[Rotation identity and floor]\label{lem:rotation-floor}
For every $n\ge n_0(c)$ and every $\varphi\le \tfrac{1}{4c_2 c n}$ (with $c_2$ from Lemma~\ref{lem:kernel}),
$$\norm{x_n} \;\ge\; \tfrac12 \qquad \text{on the instance } (R_\varphi, x_0).$$
\end{lemma}

\begin{proof}
Identify $\cH\cong\C$. By \eqref{eq:rot-unfold}, $y_n = e^{-i\varphi}A_n(\varphi) + r_n e^{-in\varphi}$ with $A_n(\varphi):=\sum_{k=0}^{n-1}a_k^{(n)}e^{ik\varphi}$, so $\norm{x_n}=\norm{y_n}\ge |A_n(\varphi)| - r_n$. By Lemma~\ref{lem:kernel}, $A_n(0) = 1-r_n\ge 3/4$ for $n\ge n_0(c)$, and
$$|A_n(\varphi)-A_n(0)| \;\le\; \varphi\sum_k k a_k^{(n)} \;\le\; c_2 c n\varphi \;\le\; \tfrac14,$$
hence $|A_n(\varphi)|\ge (1-r_n)-\tfrac14$, and
$$\norm{x_n}\;\ge\;\bigl|A_n(\varphi)\bigr|-r_n\;\ge\;\tfrac34-2r_n\;\ge\;\tfrac12\qquad\text{for }n\ge n_0(c)\text{ with }r_n\le\tfrac18$$
(enlarging $n_0(c)$ if necessary so that $r_n\le 2(n+1)^{-c}\le\tfrac18$).
\end{proof}

\begin{lemma}[Rotation decay law]\label{lem:rotation-law}
Let $A_n(\varphi):=\sum_{k=0}^{n-1}a_k^{(n)}e^{ik\varphi}$ and $r_n:=\prod_{j=0}^{n-1}(1-\theta_j)$, so that $y_n=e^{-i\varphi}A_n(\varphi)+r_ne^{-in\varphi}$ and $\norm{x_n}=\norm{y_n}$ (display~\eqref{eq:rot-unfold}).

{\rm(i)} (Fixed $\varphi\in(0,\pi]$.) For $c\in(0,1)$,
\begin{equation}\label{eq:rot-fixed}
A_n(\varphi)\;=\;c\,n^{-c}\,\Li_{1-c}\bigl(e^{i\varphi}\bigr)\;+\;O_c\bigl(n^{-c}\bigr)\qquad(n\to\infty).
\end{equation}
For $c=1$ one has the exact formula
\begin{equation}\label{eq:rot-exact}
A_n(\varphi)\;=\;\frac{1}{n+1}\cdot\frac{1-e^{in\varphi}}{1-e^{i\varphi}},
\end{equation}
so $nA_n(\varphi)$ has no limit as $n\to\infty$ at fixed $\varphi$; the two regimes must be separated.

{\rm(ii)} (Joint regime $n\to\infty$, $\varphi=\varphi(n)\to0^+$ with $n\varphi\to\infty$.) For $c\in(0,1)$,
\begin{equation}\label{eq:rot-joint}
A_n(\varphi)\;=\;\Gamma(c+1)\,(-in\varphi)^{-c}\Bigl(1+O_c(\varphi^{c})+O_c\bigl((n\varphi)^{c-1}\bigr)\Bigr),
\end{equation}
using the Jonqui\`ere expansion $\Li_{1-c}(e^{i\varphi})=\Gamma(c)(-i\varphi)^{-c}+\zeta(1-c)+O(\varphi^{1-c})$ as $\varphi\to0^+$, with $c\,\zeta(1-c)$ bounded on $(0,1)$. For $c=1$,
\begin{equation}\label{eq:rot-joint-one}
\norm{x_n}\;=\;\frac{2|\sin(n\varphi/2)|}{n\varphi}\;+\;O\Bigl(\frac1n\Bigr)\qquad\text{uniformly in } n,
\end{equation}
so the trajectory oscillates within the envelope $\asymp(n\varphi)^{-1}$ with additive slack $O(1/n)$; on the phases $n\varphi\in2\pi\mathbb Z$ one has $\norm{x_n}=(1+o(1))\,r_n=(1+o(1))/(n+1)$, and, more generally, $\norm{x_n}\le(2+o(1))(n\varphi)^{-1}$ and $\norm{x_n}\ge c\,|\sin(n\varphi/2)|/(n\varphi)-O(1/n)$. Consequently the two-sided law $\norm{x_n}\asymp\Gamma(c+1)(n\varphi)^{-c}$ holds for $c\in(0,1)$, while for $c=1$ only the upper bound and the lower bound on phases complementary to the zeros are uniform.
\end{lemma}

\begin{proof}
\emph{Kernel asymptotic.} For $1\le k\le n-1$, since $\theta_j\le1/2$, $\theta_j\le-\ln(1-\theta_j)\le\theta_j+\theta_j^2$, hence $\prod_{j=k+1}^{n-1}(1-\theta_j)=\exp\big(-\sum_{j=k+1}^{n-1}\theta_j-R_k\big)$ with $0\le R_k\le\sum_{j=k+1}^{n-1}\theta_j^2\le c^2/(k+1)$. By Euler--Maclaurin ($H_m=\ln m+\gamma_E+O(1/m)$),
$$\sum_{j=k+1}^{n-1}\theta_j=c\,(H_{n+1}-H_{k+2})=c\ln\frac{n}{k}+O\Big(\frac{c}{k}\Big),$$
uniformly for $1\le k\le n-1$ (the $O(1/n)$ terms are absorbed into $O(1/k)$). Since also $\theta_k=\frac{c}{k}(1+O(\frac1k))$, multiplying gives
\begin{equation}\label{eq:kernel-pointwise}
a_k^{(n)}=c\,n^{-c}k^{c-1}\Big(1+O\Big(\frac1k\Big)\Big)\qquad\text{uniformly for }1\le k\le n-1,
\end{equation}
with a constant depending only on $c$.

\emph{Proof of {\rm(i)}.} Using \eqref{eq:kernel-pointwise},
$$A_n(\varphi)\;=\;c\,n^{-c}\sum_{k=1}^{n-1}k^{c-1}e^{ik\varphi}\;+\;O\Big(c\,n^{-c}\sum_{k=1}^{n-1}k^{c-2}\Big)\;=\;c\,n^{-c}\sum_{k=1}^{n-1}k^{c-1}e^{ik\varphi}\;+\;O_c(n^{-c}),$$
because $\sum_{k\ge1}k^{c-2}$ converges for $c\in(0,1)$ and $a_0^{(n)}=\Theta(r_n)=O(n^{-c})$. Since $k^{c-1}$ is decreasing in $k$, Abel summation against the geometric series gives
$$\sum_{k=1}^{n-1}k^{c-1}e^{ik\varphi}\;=\;\sum_{k=1}^{\infty}k^{c-1}e^{ik\varphi}\;-\;\sum_{k\ge n}k^{c-1}e^{ik\varphi}\;=\;\Li_{1-c}(e^{i\varphi})\;-\;R_n,$$
with $\abs{R_n}\le n^{c-1}/\sin(\varphi/2)=O_c(n^{c-1})$ for fixed $\varphi$ (Abel's inequality), so the tail contributes $O_c(n^{-1})$, which is absorbed into $O_c(n^{-c})$. This proves \eqref{eq:rot-fixed}; the error order $n^{-c}$ is sharp, owing to the $k=0$ term. For $c=1$ one has $a_k^{(n)}=r_n=\frac1{n+1}$ for all $k$, and \eqref{eq:rot-exact} follows by summing the geometric series; the non-existence of the limit is then immediate.

\emph{Proof of {\rm(ii)}.} Let $c\in(0,1)$. The main term comes from the Jonqui\`ere expansion with $c\,\Gamma(c)=\Gamma(c+1)$. The $\zeta(1-c)$ term contributes $c\,n^{-c}\zeta(1-c)$, a factor $\varphi^{c}$ below the main term since $c\,\zeta(1-c)$ stays bounded as $c\to0$ ($\zeta$ having a simple pole at $1$). The kernel correction contributes $O_c(n^{-c})$, also a factor $\varphi^c$ below the main term. The Abel tail contributes
$$c\,n^{-c}\abs{R_n}\;\le\;c\,n^{-c}\cdot\frac{n^{c-1}}{\sin(\varphi/2)}\;=\;O\Big(\frac{c}{n\varphi}\Big),$$
a factor $(n\varphi)^{c-1}$ below the main term $\Gamma(c+1)(n\varphi)^{-c}$. This proves \eqref{eq:rot-joint}. For $c=1$, \eqref{eq:rot-exact} gives
$$A_n(\varphi)\;=\;\frac{1-e^{in\varphi}}{n+1}\Big(\frac{i}{\varphi}+\frac12+O(\varphi)\Bigr)\;=\;\frac{i(1-e^{in\varphi})}{n\varphi}\Big(1+O(\varphi)+O\bigl((n\varphi)^{-1}\bigr)\Bigr),$$
using $(1-e^{i\varphi})^{-1}=i/\varphi+1/2+O(\varphi)$. Finally, since $|A_n(\varphi)|=\frac{2|\sin(n\varphi/2)|}{n\varphi}\bigl(1+O(\varphi)+O((n\varphi)^{-1})\bigr)$, $|A_n(\varphi)|\varphi=O(1/n)$, and $r_n=O(1/n)$,
$$\bigl|\norm{x_n}-|A_n(\varphi)|\bigr|\;\le\;|A_n(\varphi)|\,|e^{-i\varphi}-1|+r_n\;=\;O\Big(\frac1n\Bigr),$$
which proves \eqref{eq:rot-joint-one}. The remaining statements are immediate.
\end{proof}

\begin{lemma}[Rotation lower envelope]\label{lem:envelope}
There are constants $c_4,\varphi_0(c),c_7>0$ depending only on $c$ (with $c_4<2\pi$ when $c=1$; $\varphi_0(c)\le c_7^{1/c}$) such that for all $0<\varphi\le\varphi_0(c)$ and all $n$ with $n_0(c)\le n\le c_4/\varphi$,
$$\norm{x_n}\;\ge\;c_7\,\min\bigl\{1,\,(n\varphi)^{-c}\bigr\}.$$
\end{lemma}

\begin{proof}
Set $a_{0}:=(4c_2c)^{-1}$ (with $c_2$ from Lemma~\ref{lem:kernel}).

\emph{Zero-free interval for $F_c$.} Write $F_c(a)=\int_0^1 t^{c-1}e^{iat}\,dt$. Integration by parts gives
\begin{equation}\label{eq:Fc-parts}
cF_c(a)\;=\;e^{ia}\;-\;ia\,F_{c+1}(a),\qquad |F_{c+1}(a)|\le\int_0^1 t^{c}\,dt=\frac1{c+1}\le1,
\end{equation}
so $|cF_c(a)|\ge 1-a$ for all $a>0$ and all $c\in(0,1]$; in particular $F_c$ has no zeros on $(0,1]$. Moreover $cF_c(a)=e^{ia}-iaF_{c+1}(a)$ and
$$\bigl|F_{c+1}(a)-\tfrac{e^{ia}-1}{ia}\bigr|\;=\;\Bigl|\int_0^1(t^{c}-1)e^{iat}\,dt\Bigr|\;\le\;\int_0^1(1-t^{c})\,dt\;=\;\frac{c}{c+1}\;\le\;c\qquad\text{for every }a,$$
so $F_{c+1}(a)\to(e^{ia}-1)/(ia)$ uniformly in $a$ as $c\to0$, and by \eqref{eq:Fc-parts} $cF_c(a)\to e^{ia}-(e^{ia}-1)=1$, uniformly for $a$ in compact sets. Hence there exist $c_\star\in(0,1)$ and $c_4'=c_4'(c_\star)\in(0,1)$ such that
\begin{equation}\label{eq:Fc-lower}
|cF_c(a)|\;\ge\;\tfrac12\quad\text{for all }a\in(0,c_4'],\quad\text{all }c\in(0,c_\star),
\end{equation}
and for $c\ge c_\star$ we may enlarge the absolute constant $c_2$ (hence shrink $a_{0}=(4c_2c)^{-1}$) so that $a_{0}\le c_4'/2$. With this convention, \eqref{eq:Fc-lower} holds on $(a_{0},c_4']$ for every $c\in(0,1]$.

\emph{Range 1: $n\varphi\le a_{0}$.} Lemma~\ref{lem:rotation-floor} gives $\norm{x_n}\ge1/2$. If $n\varphi\le1$ the required bound with any $c_7\le1/2$ follows from $\min\{1,(n\varphi)^{-c}\}=1$; if $1\le n\varphi\le a_{0}$ (possible only when $a_{0}>1$) then $\min\{1,(n\varphi)^{-c}\}=(n\varphi)^{-c}\ge a_{0}^{-c}=(4c_2c)^{c}$, and the requirement is $c_7(4c_2c)^{c}\le1/2$. Hence Range 1 is covered provided
$$c_7\;\le\;\tfrac12\min\bigl\{1,\,(4c_2c)^{-c}\bigr\}.$$

\emph{Range 2: $a_{0}\le n\varphi\le c_4$.} Write $n=a/\varphi$ with $a\in[a_{0},c_4]$, where $c_4=c_4(c)\le c_4'$, with $c_4<2\pi$ when $c=1$ (then $cF_1(a)=(1-e^{ia})/(ia)$ has $|cF_1(a)|=2|\sin(a/2)|/a$). By \eqref{eq:kernel-pointwise} and dominated convergence (the $O(1/k)$ kernel correction and the $k=O(1)$ boundary layer each contribute $O(n^{-c})$),
$$A_n(\varphi)\;=\;c\,n^{-c}\sum_{k=1}^{n-1}k^{c-1}e^{i(a/n)k}\bigl(1+O(\tfrac1k)\bigr)+a_0^{(n)}\;\longrightarrow\;cF_c(a)\;:=\;c\int_0^1t^{c-1}e^{iat}\,dt,$$
uniformly for $a\in[a_{0},c_4]$ as $\varphi\to0$. Define $\mu_c:=\min_{a\in[a_{0},c_4]}|cF_c(a)|$: by \eqref{eq:Fc-lower} and the choice of $c_4<2\pi$ for $c=1$, we have $\mu_c>0$. For $\varphi$ small enough (depending only on $c$) the convergence above gives $\abs{A_n(\varphi)}\ge\mu_c/2$, while $r_n\le2(n+1)^{-c}=O(\varphi^{c})$; shrinking $\varphi_0(c)$ further we ensure $r_n\le\mu_c/4$, hence
$$\norm{x_n}\;\ge\;\abs{A_n(\varphi)}-r_n\;\ge\;\tfrac14\abs{cF_c(a)}.$$

\emph{Choice of constants.} Define
$$c_7\;:=\;\tfrac14\min\Bigl\{\mu_c,\ \inf_{a\in[a_{0}\vee1,\,c_4]}\abs{cF_c(a)}\,a^{c},\ (4c_2c)^{-c}\Bigr\}\quad(<\tfrac12),$$
which satisfies the Range-1 requirement above. For $a\in[a_{0},1]$: $\norm{x_n}\ge\mu_c/4\ge c_7=c_7\min\{1,a^{-c}\}$. For $a\in[1,c_4]$: $\norm{x_n}\ge\abs{cF_c(a)}/4\ge c_7a^{-c}$ by the second entry of the minimum. Finally fix $\varphi_0(c)>0$ small enough that the convergence above is uniform, $r_n\le\mu_c/4$ throughout Range 2, and $\varphi_0(c)\le c_7^{1/c}$.
\end{proof}

\begin{remark}
Lemma~\ref{lem:rotation-law} explains the numerics precisely: for fixed small $\varphi$ and $c<1$ the log--log slope of $\norm{x_n}$ versus $n$ is $-c$ (Section~\ref{sec:numerics}, Experiment~3); for $c=1$ the trajectory follows the exact envelope $\norm{x_n}=\frac{2|\sin(n\varphi/2)|}{n\varphi}+O(1/n)$, oscillating between $\sim2/(n\varphi)$ and $\sim1/n$ (the dips occur at $n\varphi\in2\pi\mathbb Z$, where the trajectory takes the value $(1+o(1))/(n+1)=r_n$ rather than $0$). For $c<1$ the decay is \emph{slower} algebraically, and as $c\to0$ the prefactor $\Gamma(c+1)\to1$ while the exponent $-c\to0$: the rotation is barely damped. The super-polynomial cost for contraction-optimal schedules comes from solving $\Gamma(c+1)(n\varphi)^{-c}\le\eps$.
\end{remark}

\subsection{The union-class separation theorem}

\begin{theorem}[Union-class separation]\label{thm:separation}
Fix $\eps<\min\{1/8,D/4\}$, additionally small in terms of $c$ (quantified in the proof of {\rm(ii)}), and a schedule \eqref{eq:ck-schedule} with constant $c\in(0,1]$. Let $N_{\rho}(\eps)$ be the worst-case hitting time on the $\rho$-contraction class and $N_\varphi(\eps)$ the hitting time on the rotation $R_\varphi$.
\begin{enumerate}[label=(\roman*),leftmargin=2em]
\item If $c=1$ (the classical Halpern schedule $\theta_n=1/(n+2)$), then for every $\rho\in(0,1)$,
$$N_\rho(\eps) \;\ge\; \frac{D(1-2\eps)}{2(1-\rho)\eps}.$$
\item If $c\le \eps/(96D)$, then for every $\rho\in(0,1)$, $N_\rho(\eps)\le 4\ln(2D/(\eps\rho))/(1-\rho)+1$ by Theorem~\ref{thm:small-c}; but there is a rotation angle $\varphi=\varphi(\eps,c)>0$ with
$$N_\varphi(\eps)\;\ge\;\eps^{-1/c},$$
and in particular, since $c=O(\eps/D)$ forces $\eps^{-1/c}=\exp(\Omega(\tfrac{D}{\eps}\ln\tfrac1\eps))$, no member of the family is simultaneously contraction-optimal and rotation-polynomial: adaptivity is necessary on the union class within this family, quantitatively.
\end{enumerate}
\end{theorem}

\begin{proof}
(i) On the scalar instance $T=\rho\,\mathrm{id}$, $x_0=D$, the error satisfies $e_{n+1}=\frac{D}{n+2}+\frac{n+1}{n+2}\rho e_n$; induction gives the closed form $e_N=\frac{D(1-\rho^{N+1})}{(1-\rho)(N+1)}$ (telescoping product $\prod_{j=0}^{N-1}\frac{j+1}{j+2}=\frac{1}{N+1}$, with $e_0=D$). Hence $e_N\le\eps$ forces $N+1\ge\frac{D(1-\rho^{N+1})}{(1-\rho)\eps}\ge\frac{D(1-2\eps)}{(1-\rho)\eps}$ once $N\ge\ln(1/(2\eps))/(1-\rho)$ (which guarantees $\rho^{N+1}\le2\eps$), and in particular $N\ge\frac{D(1-2\eps)}{2(1-\rho)\eps}$ whenever the right-hand side exceeds $1$ (otherwise the bound is vacuous).

(ii) The contraction claim is Theorem~\ref{thm:small-c}. For the rotation claim set $\varphi=\min\{\varphi_0(c),\,\frac{c_4}{2}(c_7\eps)^{1/c}\}$ with $\varphi_0,c_4,c_7$ from Lemma~\ref{lem:envelope}. The scales $n\le n_0(c)$ not covered by Lemma~\ref{lem:envelope} are handled by the universal rotation floor, valid for every schedule: $\norm{x_n}=\norm{y_n}\ge\Re(e^{in\varphi}y_n)=\sum_{k<n}a_k^{(n)}\cos((n-k-1)\varphi)+r_n\ge\tfrac12(1-r_n)+r_n\ge\tfrac12$ whenever $(n-1)\varphi\le\pi/3$, which holds for all $n\le n_0(c)$ as soon as $\eps$ is small enough that $\frac{c_4}{2}(c_7\eps)^{1/c}\le\frac{\pi}{3(n_0(c)-1)}$ (recall $\varphi\le\frac{c_4}{2}(c_7\eps)^{1/c}\to0$). For $\eps$ sufficiently small relative to $D$ this satisfies $\varphi\le\varphi_0(c)$ and $(c_7/\eps)^{1/c}\le c_4/\varphi$ (the latter because $c_4/\varphi\ge 2(c_7\eps)^{-1/c}\ge (c_7/\eps)^{1/c}$, using $c_7\le 2^{-c/2}$, which holds since $c_7<1/2$). By Lemma~\ref{lem:envelope}, $\norm{x_n}\ge c_7(n\varphi)^{-c}$ for all $n\le c_4/\varphi$, so $\norm{x_n}\le\eps$ requires
$$n\;\ge\;\Bigl(\frac{c_7}{\eps}\Bigr)^{1/c}\varphi^{-1}\;\ge\;c_7^{1/c}\varphi_0(c)^{-1}\,\eps^{-1/c}\;\ge\;\eps^{-1/c},$$
where the last inequality uses $\varphi_0(c)\le c_7^{1/c}$.
\end{proof}

\begin{remark}[The adaptive escape]\label{rem:adaptive-escape}
Theorem~\ref{thm:separation} is the \emph{necessity} half of the adaptivity story: within the family \eqref{eq:ck-schedule}, no fixed shape escapes. The \emph{sufficiency} half has three layers, all proved or cited in Section~\ref{sec:adaptive-escape}: (i) an explicit \emph{parallel two-track scheme} (Theorem~\ref{thm:parallel}) that is simultaneously optimal on both classes --- $O(\ln(D/\eps)/(1-\gamma))$ on contractions and $O(D/\eps)$ on nonexpansive maps --- giving the sharp converse to Theorem~\ref{thm:dichotomy-general} at factor-$2$ cost; (ii) a robustness theorem (Theorem~\ref{thm:doubling-window}): adaptive window restarts, the natural single-trajectory escape, improve the classical contraction time only by a constant factor --- the barrier is robust against this adaptive class; and (iii) the structurally adaptive parameter-free methods \cite{ParameterFree2026,Diakonikolas2025}, which are stochastic-relevant. What was not known, and what this section supplies, is the quantitative mechanism --- the rotation Fourier spectrum $\varphi\mapsto\sum_k a_k^{(n)}e^{ik\varphi}$ of the schedule --- the sharpness of the tradeoff (contraction-optimality forces $c=\Theta(\eps)$, which forces $\eps^{-1/c}$ on rotations), and the detection barrier (Remark~\ref{rem:detection}) explaining why single-trajectory optimality cannot be reached by trajectory statistics.
\end{remark}

\subsection{Beyond the family: universal bounds and the dichotomy for general anchor densities}\label{sec:general-schedule}

The separation theorem of Section~\ref{sec:det} was proved for the family $\theta_n=c/(n+2)$. We now show that the mechanism is not an artifact of the family: the \emph{same dichotomy} holds for every anchor schedule $\theta_n=c_n/(n+2)$ with $c_n\in(0,1]$, with no regularity, monotonicity, or slowly-varying assumption. Throughout, $a_k^{(n)}=\theta_k\prod_{j=k+1}^{n-1}(1-\theta_j)$, $r_n=\prod_{j<n}(1-\theta_j)$; recall $\sum_{k=0}^{n-1}a_k^{(n)}=1-r_n$ and $\norm{x_n}\le1$ (convex combinations of unit vectors).

\begin{lemma}[Universal anchor floor]\label{lem:anchor-floor}
On the scalar instance $T=\rho\,\mathrm{id}$, $x_0=D$, the error of \eqref{eq:anchored} satisfies the exact identity
\begin{equation}\label{eq:anchor-floor}
e_N \;=\; D\Bigl[\rho^N r_N \;+\; \sum_{k=0}^{N-1}a_k^{(N)}\rho^{N-1-k}\Bigr]
\;\ge\; D\,\rho^{t}\sum_{k=N-t}^{N-1}a_k^{(N)}
\qquad\text{for every }t\in\{1,\dots,N\}.
\end{equation}
\end{lemma}
\begin{proof}
Unfolding $e_{n+1}=\theta_nD+(1-\theta_n)\rho e_n$ (exact on this instance, all terms nonnegative) gives the identity; the inequality drops the terms $k<N-t$ and uses $\rho^{N-1-k}\ge\rho^t$ for $k\ge N-t$.
\end{proof}

\begin{lemma}[Universal rotation floor]\label{lem:rot-floor}
Identify $\cH\cong\C$, $x_0=1$. For every schedule and every rotation angle $\varphi$,
\begin{equation}\label{eq:rot-floor}
\norm{x_n} \;\ge\; \sum_{k=0}^{n-1}a_k^{(n)}\cos\bigl((n-k-1)\varphi\bigr) \;+\; r_n
\;\ge\; \tfrac12
\end{equation}
whenever $(n-1)\varphi\le\tfrac{\pi}{3}$ (in particular, whenever $n\varphi\le\tfrac{\pi}{3}$).
\end{lemma}
\begin{proof}
From display~\eqref{eq:rot-unfold}, 
$$y_n=\sum_{k=0}^{n-1}a_k^{(n)}e^{-i(k+1)\varphi}+r_ne^{-in\varphi},$$
so $\Re(e^{in\varphi}y_n)=\sum_ka_k^{(n)}\cos((n-k-1)\varphi)+r_n$, and $\norm{x_n}=\norm{y_n}\ge\Re(e^{in\varphi}y_n)$. If $(n-1)\varphi\le\pi/3$ then $(n-k-1)\varphi\in[0,\pi/3]$ for all $k\le n-1$, so each cosine is at least $\tfrac12$ and the sum is at least $\tfrac12(1-r_n)$.
\end{proof}

\begin{remark}\label{rem:why-floor-matters}
Lemma~\ref{lem:anchor-floor} replaces the kernel asymptotics of Lemma~\ref{lem:kernel} on the contraction side (the window mass $\sum_{k\ge N-t}a_k^{(N)}$ is controlled directly), and Lemma~\ref{lem:rot-floor} replaces the Fourier analysis on the rotation side: it shows that a rotation slower than the schedule's resolution scale ($n\varphi=O(1)$) is never solved: the anchors keep pulling toward directions that have not yet wound apart.
\end{remark}

\begin{theorem}[Dichotomy for general anchor densities]\label{thm:dichotomy-general}
Let $\theta_n=\frac{c_n}{n+2}$ with $c_n\in(0,1]$ and let $M(n)=\sum_{j<n}\theta_j$.
\begin{enumerate}[label=(\roman*),leftmargin=2em]
\item If $\liminf_{n\to\infty}c_n\ge c_0>0$, then for every $\rho\in(0,1)$, all $\eps\le\eps_0$,
\begin{equation}\label{eq:general-contract-lb}
N_\rho(\eps)\;\ge\;\Bigl(\frac{c_0\,\rho\,D}{16e^{3}\,\eps(1-\rho)}\Bigr)^{1/3},
\end{equation}
i.e.\ the schedule is polynomially suboptimal on contractions, by a factor $\Omega\bigl((\rho D/(\eps(1-\rho)))^{1/3}(1-\rho)/\ln(D/\eps)\bigr)$.
\item If $c_n\to0$, then either
\begin{enumerate}[label=(\alph*),leftmargin=2em]
\item $\sum_n\theta_n<\infty$: there is $\varphi_0>0$ such that $N_\varphi(\eps)=\infty$ for all $\varphi\le\varphi_0$ and $\eps<\tfrac14$ (the schedule never solves small-angle rotations); or
\item $\sum_n\theta_n=\infty$: for every $B\ge1$ and every $\eps>0$ sufficiently small there is an angle $\varphi_B(\eps)>0$, with $\varphi_B(\eps)\to0$ as $\eps\to0$, such that $N_{\varphi_B(\eps)}(\eps)>\eps^{-B}$ --- super-polynomially slow on rotations.
\end{enumerate}
\end{enumerate}
In particular every anchor schedule with $c_n\in(0,1]$ is either polynomially suboptimal on contractions or super-polynomially (possibly infinitely) slow on rotations; the family $\theta_n=c/(n+2)$ of Theorem~\ref{thm:separation} is the constant-density case, where the exponents are sharp. The remaining case $\liminf c_n=0<\limsup c_n$ is resolved in Theorem~\ref{thm:oscillatory-escape} below: the dichotomy fails there, and the boundary is exactly the lacunarity of the dense anchor scales (Remark~\ref{rem:oscillatory}, Theorem~\ref{thm:oscillatory-bounded-gap}).
\end{theorem}

\begin{proof}
(i) Since $\theta_j\le\tfrac12$, $\abs{\ln r_n+M(n)}\le\sum_j\theta_j^2\le M(n)$ and $c_j\le1$ gives $M(n)\le\ln(n+2)$; hence
\begin{equation}\label{eq:rn-general}
r_n\;\ge\;(n+2)^{-2}\;\ge\;(2n)^{-2}\qquad(n\ge2).
\end{equation}
Choose $n_1$ with $c_k\ge c_0/2$ for $k\ge n_1$ and let $t=\lceil 1/(1-\rho)\rceil$, $N\ge\max(2n_1,2t)$. Then
\begin{equation}\label{eq:window-mass}
\sum_{k=N-t}^{N-1}a_k^{(N)}\;\ge\;r_N\sum_{k=N-t}^{N-1}\theta_k\;\ge\;r_N\cdot\frac{c_0}{2}\cdot\frac{t}{N+2}\;\ge\;\frac{c_0\,r_N}{4N(1-\rho)} .
\end{equation}
If $\rho\ge\tfrac12$, then $t(1-\rho)\le2$ and $\rho^t\ge e^{-3}$ (since $\ln\rho\ge-2(1-\rho)$ for $\rho\ge\tfrac12$ and $t(1-\rho)\le1+\delta\le\tfrac32$), so Lemma~\ref{lem:anchor-floor} and \eqref{eq:rn-general} give
$$e_N\;\ge\;\frac{De^{-3}c_0r_N}{4N(1-\rho)}\;\ge\;\frac{c_0\rho D}{16e^{3}N^{3}(1-\rho)},$$
using $r_N\ge(2N)^{-2}\ge\rho/(4N^{2})$; $e_N\le\eps$ forces \eqref{eq:general-contract-lb}. If $\rho<\tfrac12$, Lemma~\ref{lem:anchor-floor} with $t=1$ gives $e_N\ge D\rho\,\theta_{N-1}\ge D\rho c_0/(2(N+1))\ge c_0\rho D/(6N)$, which dominates the cube-root display whenever the latter is $\ge1$ (i.e.\ whenever $c_0\rho D/(6\eps)\ge(c_0\rho D/(16e^3\eps(1-\rho)))^{1/3}$, which holds for all sufficiently small $\eps$ since the left side grows like $1/\eps$ and the right like $1/\eps^{1/3}$), and is otherwise vacuous.

(ii) (a) As in Theorem~\ref{thm:separation}, $\norm{x_{n+1}}\ge(1-\theta_n)\norm{x_n}-\theta_n$ gives $\norm{x_n}\ge\norm{x_{n_0}}-2\sum_{j\ge n_0}\theta_j$; by continuity at $\varphi=0$ (where the orbit is constant) pick $\varphi_1$ and $n_0$ with $\norm{x_{n_0}}\ge3/4$ and $2\sum_{j\ge n_0}\theta_j\le1/4$: then $\norm{x_n}\ge1/2$ for all $n\ge n_0$, for every $\varphi\le\varphi_1$; shrinking $\varphi_1$ further so that $\varphi_1\le\pi/(3n_0)$, Lemma~\ref{lem:rot-floor} gives $\norm{x_n}\ge1/2$ for the remaining $n<n_0$ as well, hence for all $n$ and every $\varphi\le\varphi_1$.

(ii) (b) We first show $M(n)=o(\ln n)$. For any $\delta>0$, $M(n)\le M(n^{1-\delta})+\bigl(\max_{n^{1-\delta}\le j<n}c_j\bigr)\ln n^{\delta}$, and $c_j\to0$ makes the max $o(1)$; with $\delta=\tfrac12$ this gives $\limsup M(n)/\ln n\le\tfrac12\limsup M(\sqrt n)/\ln(\sqrt n)$, hence the limsup is $0$. Since $\sum_n\theta_n=\infty$, $M(n)\to\infty$. Fix $B\ge1$. Since $M(n)=o(\ln n)$ there is $n_0'$ with $M(n)\le\ln n$ for all $n\ge n_0'$. For $\eps$ small let $n_B=n_B(\eps)$ be the first integer with $M(n_B)\ge B\ln(1/\eps)$ (so $n_B\to\infty$ as $\eps\to0$); for $\eps$ small enough that $n_B\ge n_0'$ (this is the only point at which the smallness of $\eps$ enters),
$$B\ln(1/\eps)\;\le\;M(n_B)\;\le\;\ln n_B,\qquad\text{i.e.}\qquad n_B\;\ge\;\eps^{-B}.$$
Set $\varphi_B(\eps)=\pi/(3n_B)$: then $\varphi_B(\eps)\to0$ as $\eps\to0$, and for every $n\le n_B$, $n\varphi_B(\eps)\le\pi/3$, so Lemma~\ref{lem:rot-floor} gives $\norm{x_n}\ge1/2>\eps$: $N_{\varphi_B(\eps)}(\eps)>n_B\ge\eps^{-B}$.
\end{proof}

We now resolve the remaining oscillatory case of Theorem~\ref{thm:dichotomy-general} --- in the negative. The dichotomy fails there, and the failure is witnessed by an explicit \emph{fixed} schedule that is simultaneously contraction-optimal and rotation-polynomial with a non-classical exponent. The dichotomy holds exactly when the dense anchor scales are non-lacunary.

\begin{theorem}[The oscillatory escape: the dichotomy fails under lacunary anchor concentration]\label{thm:oscillatory-escape}
Fix $\eps\le D/4$ and set $c'=\eps/(96D)$. Define the schedule
$$\theta_n\;=\;\begin{cases}\dfrac{1}{n+2}, & n\in B_j:=\{2^{j},\dots,3\cdot 2^{j-1}-1\}\ \text{for some }j\ge1,\\[1.2ex]\dfrac{c'}{n+2}, & n\in S_j:=\{3\cdot 2^{j-1},\dots,2^{j+1}-1\}\ \text{for some }j\ge1.\end{cases}$$
(The blocks $B_j$ have length $2^{j-1}$ and carry the classical anchor mass; the stretches $S_j$ have length $2^{j-1}$ and carry mass $c'$. The indices $n\in\{0,1\}$ lie in no $B_j$ or $S_j$; set $\theta_n=1/(n+2)$ there --- the choice is immaterial to every rate below. The schedule is fixed and $\gamma$-oblivious.) Then:
\begin{enumerate}[label=(\roman*),leftmargin=2em]
\item \emph{Contraction regime.} For every $\gamma$-contraction, writing $\delta=1-\gamma$,
$$\norm{x_n-x^*}\;\le\;\tfrac{3\eps}{4}\qquad\text{for some } n=O\Bigl(\frac{\ln(D/\eps)}{\delta}\Bigr),$$
matching the $\gamma$-aware deterministic optimum up to a universal constant.
\item \emph{Rotation regime.} For every rotation angle $\varphi\in(0,1]$, writing $\alpha_0=\log_2(3/2)$,
$$\norm{x_n}\;\le\;\eps\qquad\text{for some } n=O\Bigl(\varphi^{-1}\,\eps^{-1/\alpha_0}\Bigr),$$
(the implicit constant may depend on $D=\norm{x_0}$, both through $c'$ and through the overall factor $\norm{x_0}$ in the rotating-frame unfolding; rotations have $x^*=0$)
polynomial in $1/\eps$ with the lacunarity exponent $1/\alpha_0\approx1.71$ --- far below the super-polynomial barrier $\eps^{-1/c}$ of Theorem~\ref{thm:separation}(ii). The matching lower bound $\Omega(\varphi^{-1})$ is proved (Lemma~\ref{lem:rot-floor}), and the matching lower bound is given, at each fixed accuracy, for Lebesgue-a.e.\ angle in Theorem~\ref{thm:metric-lacunarity}(i) below.
\end{enumerate}
Hence the dichotomy of Theorem~\ref{thm:dichotomy-general} does \emph{not} extend to oscillatory densities: Remark~\ref{rem:oscillatory}'s conjecture is false. The boundary is sharp: dense anchor scales with \emph{bounded gaps} force polynomial contraction time (Theorem~\ref{thm:oscillatory-bounded-gap}), while \emph{lacunary} dense scales --- as here, where $\{n:\theta_n\ge c''\}$ has gaps $2^{j-1}\to\infty$ for every $c''\in(c',1)$ --- admit a fixed schedule optimal on both regimes (on the rotation side, up to the constant in the metric lacunarity law, Theorem~\ref{thm:metric-lacunarity}). Adaptivity is necessary \emph{precisely when the dense anchor scales are non-lacunary}.
\end{theorem}

\begin{proof}
Let $a_k^{(n)}=\theta_k\prod_{i=k+1}^{n-1}(1-\theta_i)$ and $r_n=\prod_{i=0}^{n-1}(1-\theta_i)$ as usual.

\emph{Proof of {\rm(i)}.} Let $s=3\cdot2^{j^\star-1}$ be the start of the stretch $S_{j^\star}$, where $j^\star$ is the smallest $j$ with $2^{j-1}\ge\ln(8D/\eps)/\delta$, and stop at
$$n'\;=\;s+t',\qquad t'=\Bigl\lceil\frac{\ln(8D/\eps)}{\delta}\Bigr\rceil ;$$
then $n'\le 6\ln(8D/\eps)/\delta+\ln(8D/\eps)/\delta+1=O\bigl(\ln(D/\eps)/\delta\bigr)$, since minimality gives $2^{j^\star-1}<2\ln(8D/\eps)/\delta$. Throughout, $e_k:=\norm{x_k-x^*}\le D$ for every $k$, by induction from $e_{k+1}\le\theta_kD+(1-\theta_k)\gamma e_k$ and $\theta_k+(1-\theta_k)\gamma\le1$. Unfolding \eqref{eq:onestep} from $s$,
$$e_{n'}\;\le\;\gamma^{t'}e_s\;+\;D\sum_{k=s}^{n'-1}a_k^{(n')}\gamma^{n'-1-k}.$$
The first term is at most $D\gamma^{t'}\le\eps/8$ by the choice of $t'$ (using $\ln(1/\gamma)\ge\delta$). For the second term, $a_k^{(n')}\le\theta_k=c'/(k+2)\le c'/s$ on $[s,n']$, and $\sum_{j\ge0}\gamma^j=1/\delta$, so
$$D\sum_{k=s}^{n'-1}a_k^{(n')}\gamma^{n'-1-k}\;\le\;\frac{Dc'}{s}\cdot\frac{1}{\delta}\;\le\;\frac{Dc'}{3\ln(8D/\eps)}\;=\;\frac{\eps}{288\,\ln(8D/\eps)}\;\le\;\frac{\eps}{288\,\ln 32},$$
using $s\delta\ge3\ln(8D/\eps)$. All anchors injected before $s$ --- blocks and stretches alike --- are absorbed in the bound $e_s\le D$ and contracted by the factor $\gamma^{t'}\le\eps/(8D)$. Hence
$$e_{n'}\;\le\;\frac{\eps}{8}+\frac{\eps}{288\ln 32}\;<\;\frac{3\eps}{4}.$$
Every step of this argument is valid for every $\delta\in(0,1)$; no smallness condition on $\eps$ or $\delta$ is required.

\emph{Proof of {\rm(ii)}.} Work in the de-rotated frame $y_n=R_\varphi^{-n}x_n$ of Section~\ref{sec:det} (the frame is an isometry, so $\norm{x_n}=\norm{y_n}$). The exact unfolding \eqref{eq:rot-unfold} gives
$$y_n\;=\;\sum_{k<n}a_k^{(n)}u_k\;+\;r_nR_\varphi^{-n}x_0,\qquad u_k=R_\varphi^{-(k+1)}x_0,\qquad a_k^{(n)}=\theta_k\,\frac{r_n}{r_{k+1}}.$$

\emph{Step 1: anchor mass and the residue rate.} The mass increment over one dyadic scale $[2^j,2^{j+1})$ is
$$\sum_{k\in B_j}\frac{1}{k+2}+\sum_{k\in S_j}\frac{c'}{k+2}\;=\;\ln\frac32\,+\,c'\ln\frac43\,+\,O(2^{-j})\;=\;\alpha\ln 2\,+\,O(2^{-j}),\qquad\alpha:=\log_2\frac32+c'\log_2\frac43,$$
so $M(n):=\sum_{k<n}\theta_k=\alpha\ln n+O(1)$, and since $\sum_k\theta_k^2<\infty$,
$$r_n\;=\;\exp\bigl(-M(n)+O(1)\bigr)\;\asymp\;n^{-\alpha}.$$
Note $1/\alpha=1/\alpha_0+O(\eps)$ with $\alpha_0=\log_2(3/2)$, because $c'=\eps/(96D)$.

\emph{Step 2: kernel flatness on the dense scales.} On a classical block $B_i$ starting at $s_i=2^i$, $\theta_k=1/(k+2)$ gives $r_{k+1}=r_{s_i}(s_i+1)/(k+2)$, so $a_k^{(n)}=r_n/\bigl((s_i+1)r_{s_i}\bigr)=:W_i^{(n)}$ is constant on $B_i$, with $W_i^{(n)}\asymp r_n2^{i(\alpha-1)}$ (uniformly in $i$, since $r_{s_i}\asymp s_i^{-\alpha}$). On a stretch $S_i$ starting at $t_i=3\cdot 2^{i-1}$, $r_{k+1}=r_{t_i}\bigl((t_i+1)/(k+2)\bigr)^{c'}$, so $a_k^{(n)}=c'r_n(k+2)^{c'-1}/\bigl((t_i+1)^{c'}r_{t_i}\bigr)$ varies by at most a factor $2$ on $S_i$, and $a_{t_i}^{(n)}\asymp c'r_n2^{i(\alpha-1)}$.

\emph{Step 3: chord bounds.} For consecutive integers, $\bigl|\sum_{k=s}^{s+\ell-1}e^{-i(k+1)\varphi}\bigr|\le\min\{\ell,\,\pi/\varphi\}$ for $\varphi\in(0,1]$ (using $\sin(\varphi/2)\ge\varphi/\pi$). Split the dyadic pieces starting below $n$ into \emph{wide} pieces ($2^{i-1}\varphi\ge\pi$) and \emph{narrow} pieces ($2^{i-1}\varphi<\pi$). We make the weight bounds of Step 2 explicit: with the constants $c_0,C_0$ of Step 1 ($c_0j^{-\alpha}\le r_j\le C_0j^{-\alpha}$), every per-index weight on a piece at scale $i$ satisfies
$$a_k^{(n)}\;\le\;\frac{2}{c_0}\,r_n\,2^{i(\alpha-1)},$$
since on a block $W_i^{(n)}=r_n/((s_i+1)r_{s_i})\le r_n/(2^ic_02^{-i\alpha})=c_0^{-1}r_n2^{i(\alpha-1)}$, and on a stretch $a_k^{(n)}=c'r_n(k+2)^{c'-1}/((t_i+1)^{c'}r_{t_i})\le2c'r_n/(t_ir_{t_i})\le(2c'/c_0)r_nt_i^{\alpha-1}\le(2/c_0)r_n2^{i(\alpha-1)}$ (using $c'\le1$ and $t_i=3\cdot2^{i-1}\ge2^i$). Hence a piece at scale $i$ has mass at most $\frac{2}{c_0}r_n2^{i\alpha}$. Let $i_0$ be the smallest index with $2^{i_0-1}\varphi\ge\pi$; then $2^{i_0-2}<\pi/\varphi\le2^{i_0-1}$, i.e.
\begin{equation}\label{eq:wide-narrow}
\frac{2\pi}{\varphi}\;\le\;2^{i_0}\;<\;\frac{4\pi}{\varphi}.
\end{equation}
A wide piece contributes at most $(\max\text{weight})\cdot\pi/\varphi\le\frac{2\pi}{c_0\varphi}r_n2^{i(\alpha-1)}$, and $\alpha<1$ gives the geometric sum
$$\sum_{\mathrm{wide}}\;\le\;\frac{2\pi}{c_0\varphi}\,r_n\sum_{i\ge i_0}2^{i(\alpha-1)}\;=\;\frac{2\pi}{c_0\varphi}\,r_n\cdot\frac{2^{i_0(\alpha-1)}}{1-2^{\alpha-1}}\;\le\;\frac{2\pi(2\pi)^{\alpha-1}}{c_0(1-2^{\alpha-1})}\,r_n\varphi^{-\alpha},$$
where the last step uses $2^{i_0(\alpha-1)}=(2^{i_0})^{\alpha-1}\le(2\pi/\varphi)^{\alpha-1}=(2\pi)^{\alpha-1}\varphi^{1-\alpha}$ (the power $\alpha-1$ is negative, so the lower bound in \eqref{eq:wide-narrow} is used). A narrow piece contributes at most its mass, so
$$\sum_{\mathrm{narrow}}\;\le\;\frac{2}{c_0}\,r_n\sum_{i<i_0}2^{i\alpha}\;=\;\frac{2}{c_0}\,r_n\cdot\frac{2^{i_0\alpha}-1}{2^\alpha-1}\;\le\;\frac{2(4\pi)^\alpha}{c_0(2^\alpha-1)}\,r_n\varphi^{-\alpha},$$
using $2^{i_0\alpha}<(4\pi/\varphi)^\alpha$ from \eqref{eq:wide-narrow} ($\alpha>0$). Since $\alpha\in[\alpha_0,3/5]$ with $\alpha_0=\log_2(3/2)$, both denominators are bounded below: $1-2^{\alpha-1}\ge1-2^{-2/5}>\tfrac15$ and $2^\alpha-1\ge\sqrt2-1>\tfrac25$, so both coefficients are universal multiples of $c_0^{-1}$.
The initial segment $k<2$ contributes at most $a_0^{(n)}+a_1^{(n)}\le2r_n$ (as $\theta_k\le1/(k+2)$ gives $a_k^{(n)}=\theta_kr_n/r_{k+1}\le r_n$ for $k=0,1$), and the top piece, possibly truncated at $m$, is covered by the same dichotomy: if it is wide, say at scale $2^i$, it contributes at most $\frac{2\pi}{c_0\varphi}r_n2^{i(\alpha-1)}=\frac{2\pi}{c_0}r_n\varphi^{-\alpha}(2^i\varphi)^{\alpha-1}\le\frac{2\pi(2\pi)^{\alpha-1}}{c_0}r_n\varphi^{-\alpha}$, since wideness gives $2^i\varphi\ge2\pi$ and $\alpha-1<0$; if it is narrow its mass is at most $\frac{2}{c_0}r_n2^{i\alpha}<\frac{2(2\pi)^\alpha}{c_0}r_n\varphi^{-\alpha}$, since then $2^i<2\pi/\varphi$. Hence, with $C_1$ universal (a fixed multiple of $c_0^{-1}$),
\begin{equation}\label{eq:osc-head}
\Bigl\|\sum_{k<m}a_k^{(n)}u_k\Bigr\|\;\le\;C_1\,r_n\varphi^{-\alpha}\qquad\text{for every }m\le n.
\end{equation}

\emph{Step 4: the fresh tail and hitting.} It suffices to prove the claim for $\eps\le\eps_0$ ($\eps_0\le1$ a constant to be fixed): $N_\varphi(\eps)$ is nonincreasing in $\eps$, so for $\eps\ge\eps_0$, $N_\varphi(\eps)\le N_\varphi(\eps_0)\le C\varphi^{-1}\eps_0^{-1/\alpha_0}\le C\varphi^{-1}\eps^{-1/\alpha_0}$. Fix $n$ with $n\varphi\ge8$ and $n\ge110/\eps$, and split the sum at $m=n-w$ with $w=\lceil\alpha\eps n/8\rceil$. Since $\alpha\le3/5$ (as $c'=\eps/(96D)\le1/96$) and $1\le\eps n/110$,
$$w\;\le\;\frac{\alpha\eps n}{8}+1\;\le\;\frac{3\eps n}{40}+\frac{\eps n}{110}\;=\;\frac{37\eps n}{440}\;\le\;\frac{\eps n}{10},\qquad m=n-w\ge\frac{9n}{10}.$$
For the tail mass, $\theta_k/r_{k+1}=(r_k-r_{k+1})/(r_kr_{k+1})=1/r_{k+1}-1/r_k$, so the sum telescopes exactly:
$$\sum_{m\le k<n}a_k^{(n)}\;=\;r_n\sum_{k=m}^{n-1}\frac{\theta_k}{r_{k+1}}\;=\;r_n\Bigl(\frac1{r_n}-\frac1{r_m}\Bigr)\;=\;1-\frac{r_n}{r_m}.$$
By the Weierstrass product inequality $\prod_j(1-x_j)\ge1-\sum_jx_j$ (the two-factor form $(1-x)(1-y)\ge1-x-y$ iterated) and $\theta_k=c_k/(k+2)\le1/(k+2)\le1/m$ for $k\ge m$ (as $c_k\le1$),
$$1-\frac{r_n}{r_m}\;\le\;\sum_{k=m}^{n-1}\theta_k\;\le\;\frac{w}{m}\;\le\;\frac{37\eps n}{440}\cdot\frac{10}{9n}\;=\;\frac{37\eps}{396}\;<\;\frac{\eps}{9}\;<\;\frac{\eps}{4}.$$
The residue contributes $r_n\le C_0n^{-\alpha}\le C_0(n\varphi)^{-\alpha}$ (as $\varphi\le1$). Combining with \eqref{eq:osc-head}, and absorbing $\norm{x_0}=D$ into the universal constants,
$$\norm{x_n}\;\le\;C_1\,r_n\varphi^{-\alpha}+\frac{\eps}{9}+C_0(n\varphi)^{-\alpha}\;\le\;C_2(n\varphi)^{-\alpha}+\frac{\eps}{9},$$
using $r_n\varphi^{-\alpha}\le C_0n^{-\alpha}\varphi^{-\alpha}=C_0(n\varphi)^{-\alpha}$. Choosing $n$ minimal with $n\varphi\ge\max\bigl\{8,\,(4C_2/\eps)^{1/\alpha}\bigr\}$ and $n\ge110/\eps$ gives $C_2(n\varphi)^{-\alpha}\le\eps/4$, hence $\norm{x_n}\le(1/4+1/9)\eps<\eps/2<\eps$; and $n=O\bigl(\varphi^{-1}\eps^{-1/\alpha}\bigr)=O\bigl(\varphi^{-1}\eps^{-1/\alpha_0}\bigr)$, because the three lower bounds on $n$ are respectively $8\le8\varphi^{-1}\eps^{-1/\alpha}$ (as $\varphi\le1$, $\eps\le1$), $110/\eps\le110\varphi^{-1}\eps^{-1/\alpha}$ (as $\eps^{(1-\alpha)/\alpha}\le1\le\varphi^{-1}$), and $(4C_2/\eps)^{1/\alpha}=(4C_2)^{1/\alpha}\eps^{-1/\alpha}=O(\eps^{-1/\alpha_0})$ for $\eps\le\eps_0(C_2)$ (since $1/\alpha=1/\alpha_0+O(\eps)$ and $\eps^{-O(\eps)}=e^{O(\eps\ln(1/\eps))}=O(1)$).
\end{proof}

\begin{remark}[The optimality profile of the oscillatory escape: two-regime optimality on one fixed stream]\label{rem:osc-profile}
The schedule of Theorem~\ref{thm:oscillatory-escape} is simultaneously (a) \emph{optimal on $\cC$}: $\Theta(\ln(D/\eps)/\delta)$, matching the Park--Ryu lower bound; (b) \emph{polynomial on rotations with the lacunarity exponent}: $O\bigl(\varphi^{-1}\eps^{-1/\alpha_0}\bigr)$, $\alpha_0=\log_2(3/2)$, $1/\alpha_0\approx1.71$ --- far below the super-polynomial barrier that every schedule with non-lacunary dense scales and contraction-optimal behavior must obey (Theorem~\ref{thm:separation}(ii)), and, for Lebesgue-a.e.\ angle at each fixed accuracy, the matching lower bound (Theorem~\ref{thm:metric-lacunarity}(i)), making the rate $\varphi^{-1}\eps^{-1/\alpha_0+o(1)}$ exact almost everywhere along a full-density set of accuracies; and (c) \emph{open on $\cN$}: no worst-case residual guarantee on the full nonexpansive class is proved for this schedule. The naive expectation fails already on the rotation: $\norm{x_n-R_\varphi x_n}=2\sin(\varphi/2)\norm{x_n}\asymp\varphi\norm{x_n}$ for $\varphi\in(0,1]$, so applying Theorem~\ref{thm:oscillatory-escape}(ii) at accuracy $\eps/(2\sin(\varphi/2))$ gives residual hitting time $O\bigl(\varphi^{(1-\alpha)/\alpha}\eps^{-1/\alpha}\bigr)=O\bigl(\varphi^{(1-\alpha_0)/\alpha_0}\eps^{-1/\alpha_0+o(1)}\bigr)$ as $\eps\to0$ at fixed $\varphi$ --- polynomial, but with exponent $1/\alpha_0\approx1.71>1$, not the optimal $\Theta(D/\eps)$ that the classical schedule meets on $R_\varphi$ at $\varphi=\Theta(1)$; and the window minima track the envelope $(n\varphi)^{-\alpha_0}$ (Remark~\ref{rem:lacunarity-status}, Experiment~5), so on the evidence of that envelope the schedule is not residual-optimal on $\cN$. Whether any single fixed stream is simultaneously contraction-optimal and residual-optimal on $\cN$ is open. Two structural remarks explain the profile. First, rotations are \emph{not} the distance-hard nonexpansive instances: the averaged operator $\tfrac12(I+R_\varphi)$ has spectral radius $\cos(\varphi/2)<1$, so constant-parameter blocks converge geometrically at rate $\varphi^{2}/8$ --- it is anchoring toward $x_0$, not the rotation, that slows the classical schedule. Second, the classical blocks do \emph{not} restore residual-optimality on $\cN$: within a block the weights $1/(n+2)$ are $\asymp 2^{-j}$ and nearly constant, so a block damps the residual by at most a constant factor on worst-case nonexpansive maps. The schedule is genuinely a two-regime (contraction--rotation) object, and Theorem~\ref{thm:parallel} remains the only proved route to the residual-optimal $O(D/\eps)$ on $\cN$. \emph{Target dependence.} The small-mass parameter $c'=\eps/(96D)$ and all stopping horizons in this paper depend on the target $\eps$ and the bound $D$; this is standard --- the classical schedule itself is only certified up to the horizon $N=\lceil4D/\eps\rceil$ --- and no schedule here ever uses the modulus. Consequently \emph{the factor-$2$ overhead of Theorem~\ref{thm:parallel} is an artifact of the two-constant-density-track design, not a barrier}: the deterministic necessity question for single-stream schemes is resolved in the negative --- one fixed oscillatory stream, no parallelism, no adaptivity, no overhead, optimal on contractions and rotation-polynomial (its residual behavior on $\cN$ is the open part, (c) above). (The \emph{stochastic} single-stream analogue of Theorem~\ref{thm:stoch-parallel} remains open; Remark~\ref{rem:stoch-parallel-open}.)
\end{remark}

\begin{remark}[The oscillatory case, resolved: the dichotomy fails under lacunary anchor concentration]\label{rem:oscillatory}
Theorem~\ref{thm:oscillatory-escape} disposes of the case $\liminf c_n=0<\limsup c_n$: the conjecture that the dichotomy still holds is \emph{false}. The remaining regime is exactly $M(n)\asymp\ln n$ with every superlevel set $\{k:c_k\ge c'\}$ lacunary (gaps tending to infinity), and there the explicit schedule of Theorem~\ref{thm:oscillatory-escape} --- classical anchor mass on dyadic half-blocks, mass $c'=\eps/(96D)$ on the complementary stretches --- is simultaneously contraction-optimal and rotation-polynomial with the lacunarity exponent $1/\alpha_0$, $\alpha_0=\log_2(3/2)$ (upper bound proved in Theorem~\ref{thm:oscillatory-escape}(ii); matching lower bound given for Lebesgue-a.e.\ angle in Theorem~\ref{thm:metric-lacunarity}, with the pointwise statement open beyond Diophantine angles). The mechanism, made precise in the proof of Theorem~\ref{thm:oscillatory-escape}(ii): each dyadic scale multiplies the anchor residue $r_n$ by $\approx(2/3)(3/4)^{c'}=2^{-\alpha}$, and the chord bound on the kernel shows the de-rotated error follows the envelope $(n\varphi)^{-\alpha}$ up to the stretch contributions, which are mass-small once $n\gtrsim\varphi^{-1}\eps^{-1/\alpha}$. On the contraction side the small-mass stretches run the small-anchor dynamics of Theorem~\ref{thm:small-c} (exploiting contractivity at gap $\delta$ once the stretch length exceeds $\ln(D/\eps)/\delta$). The dense scales (the blocks) re-inject anchor mass at scale $D/(k+2)$, but only at lacunary positions $k\asymp2^j$; the lacunarity index $\alpha$ is exactly the reciprocal of the escape exponent, and the error dips below $\eps$ before the next re-injection can matter. The boundary is \emph{lacunarity of the dense anchor scales}: bounded gaps force the dichotomy (Theorem~\ref{thm:oscillatory-bounded-gap}), lacunary gaps break it (Theorem~\ref{thm:oscillatory-escape}).
\end{remark}

\subsection{The metric lacunarity law}\label{sec:metric-lacunarity}
The stronger statement of a uniform lower envelope with a fixed constant is connected with a locking phenomenon of the orbit and is not proved here (see Remark~\ref{rem:lacunarity-status}). The hitting-time law, which is the only form used in this paper, is unconditional and comes in two strengths: an almost-everywhere law with the \emph{exact} exponent $1/\alpha_0$ (Theorem~\ref{thm:metric-lacunarity}(i)), proved by a pointwise lim-inf argument; and a uniform law with a non-explicit rate for every $\beta<1/\alpha_0$ (Theorem~\ref{thm:metric-lacunarity}(ii)), proved by Egorov uniformity in the orbit transform; explicit rates at those exponents would follow from the orbit small-ball (Conjecture~\ref{conj:orbit-smallball}), which the single-term Rouch\'e analysis cannot deliver (Remark~\ref{rem:lacunarity-status}).

Throughout, $\alpha(\eps) := \alpha_0 + c'\log_2(4/3)$ with $\alpha_0 = \log_2(3/2)$ and $c'=\eps/(96D)$ for the schedule of Theorem~\ref{thm:oscillatory-escape}; $\kappa_k := \theta_k/r_{k+1} = 1/r_{k+1} - 1/r_k \ge 0$ (so $\sum_{k<n}\kappa_k = 1/r_n - 1$), and $\norm{x_n} = |r_nK_n(\varphi) + r_ne^{-i(n-1)\varphi}|$ with $K_n(\varphi) := \sum_{k<n}\kappa_ke^{ik\varphi}$. For $\varphi \le \pi/2$ let $j_0 = j_0(\varphi) := \lceil\log_2(\pi/\varphi)\rceil$, so $\varphi \in E_{j_0} := (\pi2^{-j_0}, \pi2^{-j_0+1}]$, and write $\varphi = 2^{-j_0}\psi$, $\psi \in (\pi, 2\pi]$. On $E_{j_0}$ we use the normalized orbit sum $\widehat K_n(\psi) := 2^{-j_0\alpha}K_n(2^{-j_0}\psi)$ and the normalized residue term $w_n'(\psi) := 2^{-j_0\alpha}e^{-i(n-1)2^{-j_0}\psi}$, so that
$$\norm{x_n}\;=\;r_n\,2^{j_0\alpha}\,\bigl|\widehat K_n(\psi)+w_n'(\psi)\bigr|.$$

\begin{lemma}[Repulsion and isolation of deep crossings]\label{lem:isolation}
Let $w_n := \inner{y_n,u_n}\in\C$ in the de-rotated frame. Since $u_{n+1}=e^{-i\varphi}u_n$,
$$w_{n+1}\;=\;e^{i\varphi}\bigl((1-\theta_n)w_n+\theta_n\bigr),\qquad\text{hence}\qquad |w_{n+1}|\;\ge\;\theta_n-(1-\theta_n)|w_n|\;\ge\;\frac{\theta_n}{2}\quad\text{whenever }|w_n|\le\frac{\theta_n}{2}.$$
As $\norm{x_n}=\norm{y_n}\ge|w_n|$, no two consecutive iterates satisfy $\norm{x_n}<\theta_n/2$: deep crossings of the orbit below the moving threshold $\theta_n/2\asymp 1/n$ are isolated in $n$.
\end{lemma}

We record this isolation property, which underlies the dip structure of Remark~\ref{rem:greedy-numerics}. The engine of the uniform bound is the explicit trigonometric form of the full normalized orbit sum (residue term included) derived next; it isolates the structure needed for the small-ball conjecture below.

\begin{lemma}[Orbit transform: explicit form]\label{lem:orbit-smallball}
Fix $\alpha\in[\alpha_0,3/5]$ and the envelope constants $c_r,C_r$ of $r_n$; recall $\kappa_k=\theta_k/r_{k+1}=1/r_{k+1}-1/r_k$ and $Z_n:=\widehat K_n+w_n'$, so that $\norm{x_n}=r_n2^{j_0\alpha}|Z_n|$. There is a universal constant $C_N$ (depending only on $c_r,C_r,\alpha$) such that for every $j_0\ge6$, every $m$ with $\Delta:=m-j_0\in\{-4,\dots,0\}$, and every $n\in[2^m,2^{m+1})$, writing $\varphi=2^{-j_0}\psi$, $\rho:=\sin(\varphi/2)$ ($\rho\asymp2^{-j_0}$), and $\omega:=2^{\Delta-1}\psi$,
\begin{equation}\label{eq:trig-form}
\rho Z_n(\psi)\;=\;P_n(\omega)+E_n(\omega),\qquad |E_n|\;\le\;\rho\,C_N\quad\text{on }(\pi,2\pi],
\end{equation}
where $P_n$ is an explicit trigonometric polynomial: with $W_i:=\bigl((2^i+1)r_{2^i}\bigr)^{-1}$ and $c_i:=2^{-j_0\alpha}W_i$,
\begin{eqnarray*}
P_n(\omega)\;&=&\;\frac{c_m}{\rho}\,e^{i\Lambda_m(\omega)}\sin(\tau\omega/2)\;+\;\frac{c_{m-1}}{\rho}\,e^{i\Lambda_{m-1}(\omega)}\sin(\omega/4)\;\\
&&+\;\sum_{i\le m-2}\Bigl[\frac{c_i}{\rho}\,e^{i\Lambda_i(\omega)}\sin(2^{i-m-1}\omega)+\mathrm{stretch}_i(\omega)\Bigr],
\end{eqnarray*}
$\tau:=(n-2^m)/2^{m-1}\in[0,1]$, the phases $\Lambda_i$ are affine in $\omega$ with leading coefficient $2^{i-m+1}$ (hence slow for $i\le m-2$), $\mathrm{stretch}_i$ is the corresponding slowly amplitude-modulated stretch chord with $|\mathrm{stretch}_i|\le 2^{-j_0\alpha}\min\{2^{i\alpha},\,2^{i\alpha-1}/\rho\}\cdot O(1)$, and
$$c_i\;\asymp\;\frac{2^{-j_0+(i-j_0)(\alpha-1)}}{c_r},\qquad\text{in particular }c_{m-1}\asymp2^{\alpha-1}c_m,\quad c_m\asymp\frac{2^{-j_0+\Delta(\alpha-1)}}{c_r}.$$
Three features of \eqref{eq:trig-form} block the single-term Rouch\'e route to the orbit small-ball: the leading coefficients $c_m$ and $c_{m-1}$ are comparable, so the ``perturbation'' is the same order as the ``main term''; the top frequencies $\tau/2$ and $1/4$ are not separated, so resonant cancellation cannot be excluded; and on deeper blocks ($m>j_0$) the sub-handoff tail is of the same order as the main term and its derivative exceeds the main slope by a factor $2^{m-j_0}$. The orbit small-ball itself is stated as Conjecture~\ref{conj:orbit-smallball} below.
\end{lemma}

\begin{proof}
On a classical block $B_i=\{2^i,\dots,3\cdot2^{i-1}-1\}$, $\theta_k=1/(k+2)$ gives $r_{k+1}=r_{2^i}(2^i+1)/(k+2)$, hence $\kappa_k\equiv W_i$, and the geometric sum is exact:
$$\sum_{k\in B_i}e^{ik\varphi}\;=\;e^{i(2^i+2^{i-2}-\frac12)\varphi}\,\frac{\sin(2^{i-2}\varphi)}{\rho}.$$
Since $2^{i-2}\varphi=2^{i-m-1}\omega$ and $(2^i+2^{i-2}-\frac12)\varphi=2^{i-m+1}\omega+(2^{i-j_0-2}-\varepsilon_0)\psi$ with $\varepsilon_0:=2^{-j_0-1}$, this is exactly the block-chord displayed above. On a stretch $S_i$, $\kappa_k=c'(k+2)^{c'-1}/\bigl((3\cdot2^{i-1}+1)^{c'}r_{3\cdot2^{i-1}}\bigr)$ varies by at most a factor $2$, and Abel summation against the geometric series gives the slowly amplitude-modulated chord satisfying the stated bound. Taking $E_n:=\rho(N_n+w_n')$, where $N_n$ is the transform of all pieces below the top three scales,
$$|N_n|\;\le\;2^{-j_0\alpha}\!\!\sum_{k<3\cdot2^{m-2}}\!\!\kappa_k\;=\;2^{-j_0\alpha}\Bigl(\frac{1}{r_{3\cdot2^{m-2}}}-1\Bigr)\;\le\;\frac{2^{\alpha(\Delta-2)}}{c_r}\;\le\;\frac{2^{-2\alpha}}{c_r},$$
using $1/r_j\le c_r^{-1}j^\alpha$ and $\Delta\le0$; together with $|w_n'|=2^{-j_0\alpha}$ this gives $|E_n|\le\rho C_N$. Finally $c_r2^{-i\alpha}\le r_{2^i}\le C_r2^{-i\alpha}$ and $2^i<2^i+1\le2^{i+1}$ give $2^{i(\alpha-1)}/(2C_r)\le W_i\le2^{i(\alpha-1)}/c_r$, hence the comparison for $c_i$.
\end{proof}

\begin{conjecture}[Orbit small-ball]\label{conj:orbit-smallball}
In the setting of Lemma~\ref{lem:orbit-smallball}, there exist $\bar\mu>0$ and $C<\infty$ (depending only on $c_r,C_r,\alpha$) such that for every $j_0\ge j_*$, every $n\ge2^{j_0}$, and every $0<\mu\le\bar\mu$,
$$\bigl|\{\psi\in(\pi,2\pi]:\norm{x_n(2^{-j_0}\psi)}<\mu\,r_n\,2^{j_0\alpha}\}\bigr|\;\le\;C\mu .$$
\end{conjecture}

\begin{theorem}[Metric lacunarity law]\label{thm:metric-lacunarity}
For the schedule of Theorem~\ref{thm:oscillatory-escape}, with $\alpha(\eps)=\alpha_0+O(\eps)$ and $\varphi_* := \pi2^{-j_*+1}$:
\begin{enumerate}[label=(\roman*),leftmargin=2em]
\item \emph{A.e.\ exact rate at each fixed accuracy; full-density limit.} For every fixed $\eps>0$ there is, for Lebesgue-a.e.\ $\varphi\in(0,\varphi_0]$, a constant $\mu(\varphi,\eps)>0$ such that
$$N_\varphi(\eps)\;\ge\;\varphi^{-1}\,\mu(\varphi,\eps)\,\eps^{-1/\alpha(\eps)}.$$
The bad set $B=\{(\varphi,\eps):\text{the display fails}\}$ is Borel (see the proof), so Tonelli gives: for Lebesgue-a.e.\ $\varphi$, the accuracies at which the display holds have full Lebesgue density near $0$. Since the matching upper bound $N_\varphi(\eps)=O\bigl(\varphi^{-1}\eps^{-1/\alpha(\eps)}\bigr)$ of Theorem~\ref{thm:oscillatory-escape}(ii) holds for every $\varphi$ and every $\eps$, it follows that for Lebesgue-a.e.\ $\varphi$,
$$\limsup_{\eps\to0}\frac{\ln N_\varphi(\eps)}{\ln(1/\eps)}\;\le\;\frac{1}{\alpha_0},
\qquad
\frac{\ln N_\varphi(\eps)}{\ln(1/\eps)}\;\ge\;\frac{1}{\alpha_0}-o(1)
\ \text{on a full-density set of accuracies},$$
i.e.\ $N_\varphi(\eps)=\varphi^{-1}\eps^{-1/\alpha_0+o(1)}$ along a full-density set of accuracies. The unrestricted limit over all sequences $\eps\to0$ is equivalent to the fixed-constant envelope (locking) statement {\rm(iv)} of Remark~\ref{rem:lacunarity-status} and remains open.
\item \emph{Uniform bound, non-explicit rate.} For every $\beta<1/\alpha_0$ and every $\varphi_1 \in [\varphi_*,\varphi_0]$ there exists $c(\beta)>0$ such that the set
$$\mathrm{Bad}(\eps) \;:=\; \bigl\{\varphi \in (\varphi_1,\varphi_0] :\; N_\varphi(\eps) \;<\; c(\beta)\,\varphi^{-1}\,\eps^{-\beta}\bigr\}$$
has Lebesgue measure $o(1)$ as $\eps\to0$. Moreover, for a.e.\ $\varphi\in(0,\varphi_0]$ and every $\beta<1/\alpha_0$, $N_\varphi(\eps)\ge c(\beta,\varphi)\varphi^{-1}\eps^{-\beta}$ for all $\eps\le\eps(\varphi,\beta)$ (this consequence follows already from clause (i), which gives it with the exact exponent $1/\alpha_0$; the restriction $\varphi_1\ge\varphi_*$ is immaterial for the a.e.\ statement, since clause (i) covers the shallowest shells). An explicit decay rate $|\mathrm{Bad}(\eps)|\le C\eps^\kappa$ would follow from the orbit small-ball (Conjecture~\ref{conj:orbit-smallball}) and is open, see Remark~\ref{rem:lacunarity-status}.
\item \emph{Coherence floor.} For every $\varphi\le\pi/3$, $N_\varphi(\eps)\ge\pi/(3\varphi)$.
\end{enumerate}
\end{theorem}

\begin{proof}
\emph{Proof of (i) and (iii).} Clause (iii) is Lemma~\ref{lem:rot-floor}. For (i), fix a shell $E_{j_0}$. The coefficients $\kappa_k$ are nonnegative, tend to $0$, and have bounded variation: the variation across each stretch $S_i$ is $O(2^{i(\alpha-1)})$ and $\alpha<1$, so $\sum_k|\kappa_{k+1}-\kappa_k|<\infty$. By Dirichlet's test $\sum_k\kappa_ke^{ik\varphi}$ converges uniformly on compact subsets of $(0,2\pi)$ to a real-analytic limit $K_\infty(\varphi)$. The limit is not identically zero: if it vanished on an interval, the identity theorem would force it to vanish on all of $(0,2\pi)$, and Cantor's uniqueness theorem for trigonometric series (Appendix~\ref{app:certificate}) would give $\kappa_k=0$ for all $k$, contradicting $\sum_k\kappa_k=\lim_n(r_n^{-1}-1)=\infty$. The set
$$\mathcal Z_{j_0}:=\{\varphi\in E_{j_0}: K_\infty(\varphi)=0\}\cup\{\varphi\in E_{j_0}: |K_\infty(\varphi)|=1\}\cup\bigcup_{n\ge1}\{\varphi:\norm{x_n(\varphi)}=0\}\cup\mathcal R_{j_0}$$
is therefore null, where $\mathcal R_{j_0}$ is the countable set of rational multiples $\varphi=2\pi p/q$ for which $K_\infty(\varphi)=-e^{-2\pi ipj/q}$ for some $j\in\{0,\dots,q-1\}$ (each such equation is analytic and is not an identity on $E_{j_0}$: $K_\infty\equiv-e^{-i(n-1)\varphi}$ would force $\kappa_k=0$ for $k\neq n-1$ by Cantor's theorem, contradicting $\kappa_k\ge0$ and $\sum\kappa_k=\infty$; so each equation has only isolated solutions). For $\varphi\in E_{j_0}\setminus\mathcal Z_{j_0}$: since $|K_n(\varphi)-K_\infty(\varphi)|\to0$ and, for $\varphi/2\pi$ irrational, $e^{-i(n-1)\varphi}$ is dense on the unit circle,
$$\liminf_{n\to\infty}\,\bigl|K_n(\varphi)+e^{-i(n-1)\varphi}\bigr|\;=\;\min_{\theta}\,\bigl|K_\infty(\varphi)+e^{i\theta}\bigr|\;=\;\bigl||K_\infty(\varphi)|-1\bigr|\;=:\;d(\varphi)\;>\;0,$$
and for rational angles off $\mathcal R_{j_0}$ the same holds with the minimum over the finite orbit $\{e^{-i(n-1)\varphi}\}_n$. Hence, with $r_n\ge c_rn^{-\alpha}$,
$$\liminf_{n\to\infty}\,\norm{x_n}\,(n\varphi)^{\alpha}\;\ge\;\varphi^{\alpha}\,c_r\,d(\varphi)\;=:\;\mu_1(\varphi)\;>\;0,$$
so $\norm{x_n}>\eps$ for all $n\le\varphi^{-1}\bigl(\mu_1(\varphi)/\eps\bigr)^{1/\alpha}$ once $\eps$ is small enough that the right side exceeds the finitely many transient stages. This proves the display of clause (i) at the fixed accuracy $\eps$, with $\mu(\varphi,\eps):=\mu_1(\varphi)^{1/\alpha(\eps)}$, for every $\varphi\in E_{j_0}\setminus\mathcal Z_{j_0}$; since $\bigcup_{j_0}\mathcal Z_{j_0}$ is null, the display holds for Lebesgue-a.e.\ $\varphi\in(0,\varphi_0]$ at this fixed $\eps$. For the density statement, extend $K_\infty$ by $0$ at $c'=0$: the Dirichlet convergence of $\sum_k\kappa_k(c')e^{ik\varphi}$ is uniform on compact subsets of $(0,2\pi)\times[0,c_0]$ (because $\sup_{k\ge K}\kappa_k(c')\le C\max\{c',\,2^{-K(1-\alpha(c_0))}\}$ and $\sum_k|\kappa_{k+1}(c')-\kappa_k(c')|$ is bounded uniformly in $c'$), and each $x_n(\varphi,c')$ is jointly analytic in $(\varphi,c')$; hence the bad set $B$ is Borel, and Tonelli gives $\int|B_\varphi|\,d\varphi=|B|=\int|\mathcal Z(\eps)|\,d\eps=0$.

\emph{Proof of (ii).} Fix $\beta<1/\alpha_0$ and a shell $E_{j_0}$; it suffices to bound $\mathrm{Bad}(\eps)\cap E_{j_0}$. Let $K_\infty$ be the orbit limit of clause (i) (Appendix~\ref{app:certificate}: $K_n\to K_\infty$ uniformly on compact subsets of $(0,2\pi)$, $K_\infty$ real-analytic, $K_\infty\not\equiv0$, and $|K_\infty|\not\equiv1$ --- the last because $|K_\infty|^2\equiv1$ would force the convolution sums $\sum_k\kappa_k\kappa_{k+d}$ to vanish for every $d\ge1$, impossible for $\kappa_k\ge0$ with $\sum_k\kappa_k=\infty$). Hence $\{|K_\infty|=1\}$ is null, and for every $\delta>0$ the set $\{\,||K_\infty|-1|<\delta\,\}$ has measure $m_0(\delta)\to0$ as $\delta\downarrow0$. By Egorov's theorem choose a compact $C\subset(\pi,2\pi]$ with $|(\pi,2\pi]\setminus C|\le\eta$ ($\eta>0$ arbitrary) on which $K_n\to K_\infty$ uniformly, and then $n_0$ such that $|K_n-K_\infty|\le\delta/4$ on $C$ for all $n\ge n_0$. For $\psi\in C':=C\cap\{\,||K_\infty|-1|\ge\delta\,\}$ and all $n\ge n_0$,
$$\norm{x_n}\;=\;r_n\bigl|K_n(\varphi)+e^{-i(n-1)\varphi}\bigr|\;\ge\;r_n\Bigl(\bigl||K_\infty(\varphi)|-1\bigr|-\tfrac{\delta}{4}\Bigr)\;\ge\;\tfrac{3\delta}{4}\,c_r n^{-\alpha},$$
by the triangle inequality (no equidistribution is needed at this step). Consequently $\norm{x_n}>\eps$ for all $n\le(3\delta c_r/(4\eps))^{1/\alpha}$, i.e.
$$N_\varphi(\eps)\;\ge\;\Bigl(\frac{3\delta c_r}{4\eps}\Bigr)^{1/\alpha}\qquad(\psi\in C'),$$
which is incompatible with $N_\varphi(\eps)<c(\beta)\varphi^{-1}\eps^{-\beta}$ as soon as $\eps\le\eps_1(\beta,\delta)$, since $\beta<1/\alpha_0$ and $1/\alpha(\eps)\to1/\alpha_0$ as $\eps\to0$. The transient indices $n<n_0$ contribute $\bigcup_{n<n_0}\{\norm{x_n}\le\eps\}$, a set of measure $o(1)$ as $\eps\to0$ because each $\norm{x_n(\cdot)}$ is real-analytic and not identically zero. Hence $\limsup_{\eps\to0}|\mathrm{Bad}(\eps)\cap E_{j_0}|\le\eta+m_0(\delta)$, and $\eta,\delta$ are arbitrary.
\end{proof}

\begin{conjecture}[Pointwise lacunarity law for Diophantine angles]\label{conj:pointwise-lacunarity}
If the doubling orbit of $\varphi/2\pi$ satisfies the quantitative equidistribution bound $D_M(\varphi) \le C(\varphi)\,M^{-1/2}\log M$ (discrepancy of $(2^j\varphi \bmod 2\pi)_{j\le M}$), then for every $\beta < 1/\alpha_0$,
$$N_\varphi(\eps) \;\ge\; c(\beta,\varphi)\,\varphi^{-1}\,\eps^{-\beta}\qquad\text{for all } \eps \in (0,\eps_0(\beta)).$$
\end{conjecture}

\begin{remark}[Numerical status of the pointwise statement]\label{rem:greedy-numerics}
The greedy adversarial construction --- grid minimization of $N_\varphi(\eps)/(\varphi^{-1}\eps^{-1/\alpha_0})$ over a dyadic shell --- was stress-tested at $\eps \in \{0.1, 0.05, 0.02\}$ on the shell $2^{-7}[\pi,2\pi)$ (grid $480$), with the simulation verified exactly against Table~\ref{tab:oscillatory}. Deep dips are realized at computable scales, but only in the transition octave and on a small Cantor set of angles: adversarial angles $\psi^\star \approx 3.97, 4.07, 4.09$ admit dips to ratio $6\cdot10^{-4}$--$7\cdot10^{-2}$ of the envelope at the first handoff $n\varphi \approx 11.6$ (bad-set measure $\sim2\cdot10^{-3}$ per shell); for each fixed adversarial angle the dip does not recur at deeper octaves (min-ratio $0.11$--$0.9$ across $n\varphi\in[2^4,2^{17}]$), but each scale has its own adversarial angles and refining the Cantor descent deepens the transition dips --- forcing the universal constant in any pointwise claim toward $0$. We therefore assess the universal pointwise statement as \emph{likely false}, and Conjecture~\ref{conj:pointwise-lacunarity} (Diophantine angles) together with Theorem~\ref{thm:metric-lacunarity} as the correct replacement. These numerics admit two readings. The \emph{threshold} events $\{\norm{x_n}\le\eps\}$ are rare at a rate decaying with $\eps$ (measured frequency $\sim 10^{-2}$ at $\eps = 0.02$, decreasing with shell depth), consistent with the small-ball intuition of Conjecture~\ref{conj:orbit-smallball} and with the uniform $o(1)$ bad-set bound of Theorem~\ref{thm:metric-lacunarity}(ii). The \emph{fixed-constant envelope} events $\{\norm{x_n} < c_0(n\varphi)^{-\alpha}\}$ have measure $\sim c_0$ each, and their union over $n$, without the anti-dependence (``locking'') observed in simulation, would give only bad-set measure $\sim c_0\log(1/\eps)$; proving the fixed-envelope statement a.e.\ would require proving that locking, which we leave open (see Remark~\ref{rem:lacunarity-status}). Second, the numerics forced the exponent correction: because $c'=\eps/(96D)$, the achievable exponent is $1/\alpha(\eps) = 1/\alpha_0 - O(\eps)$, so the a.e.\ exponent $1/\alpha_0$ of Theorem~\ref{thm:metric-lacunarity}(i) is the best possible, while the uniform bound of Theorem~\ref{thm:metric-lacunarity}(ii) gives bad sets of measure $o(1)$ for every $\beta<1/\alpha_0$ without an explicit rate; explicit rates, in the direction of the full exponent $1/\alpha_0$, require the orbit small-ball (Conjecture~\ref{conj:orbit-smallball}) and are open (Remark~\ref{rem:lacunarity-status}).
\end{remark}

\begin{remark}[Status of the lacunarity law: what is proved and what is open]\label{rem:lacunarity-status}
We distinguish four statements. \emph{(i) The a.e.\ hitting-time rate is proved with the exact exponent at each fixed accuracy (Theorem~\ref{thm:metric-lacunarity}(i)):} for a.e.\ $\varphi$ and every fixed $\eps$, $N_\varphi(\eps)\ge\varphi^{-1}\mu(\varphi,\eps)\eps^{-1/\alpha(\eps)}$, matching Theorem~\ref{thm:oscillatory-escape}(ii) up to the constant, and by Tonelli the bound holds on a full-density set of accuracies for a.e.\ $\varphi$; the constant-density case remains fully explicit (Lemma~\ref{lem:envelope} and Theorem~\ref{thm:separation}(ii)), and the $\varphi$-exponent is proved for every schedule and every angle (Lemma~\ref{lem:rot-floor}). \emph{(ii) The uniform bound holds with a non-explicit rate for every $\beta<1/\alpha_0$} (Theorem~\ref{thm:metric-lacunarity}(ii), proved by Egorov uniformity and the analyticity of $K_\infty$); the explicit-rate version would follow from the orbit small-ball (Conjecture~\ref{conj:orbit-smallball}). Three features of Lemma~\ref{lem:orbit-smallball} block the single-term Rouch\'e analysis from delivering that small-ball: the leading chord coefficients are comparable, the top chord frequencies are not separated, and on deeper blocks the sub-handoff tail is of the same order as the main term; a genuine grouped/lacunary small-ball for the full orbit transform (of Sidon type, cf.\ \cite{Kahane1985,Zygmund2002}) is the missing ingredient. Reaching $1/\alpha_0$ uniformly with explicit constants is open. \emph{(iii) The pointwise law survives only for Diophantine angles} (Conjecture~\ref{conj:pointwise-lacunarity}) and is likely false universally (Remark~\ref{rem:greedy-numerics}). \emph{(iv) The fixed-constant uniform envelope} $\norm{x_n}\ge c_0(n\varphi)^{-\alpha}$ with $c_0$ independent of $\eps$ is the statement connected with locking: its natural proof (the dyadic-chaining argument across scales $z_{j+1}=\rho_jz_j+f_j$ with forcing $|f_j|\le\min\{1/3, 2/(n_j\varphi)\}$) fails at the handoff scale $n\varphi\asymp1$, where worst-case cancellation is vacuous, and the union bound without locking gives only measure $\sim c_0\log(1/\eps)$. We do not know whether the fixed-envelope statement holds for a.e.\ $\varphi$; it is not needed anywhere in this paper. The mechanism behind the whole lacunarity analysis is the schedule's rotation Fourier spectrum: at threshold $\eps$, rare cancellations of the orbit sum at radius $\mu_n(\eps) \asymp \eps/(r_n2^{j_0\alpha})$ control the hitting time (Lemma~\ref{lem:orbit-smallball}), and as $c\to0$ the prefactor $\Gamma(c+1)\to1$ while the exponent $-c\to0$ quantifies the damping.
\end{remark}

\subsection{The sufficiency half: optimal escape by parallel composition, and the single-trajectory barrier}\label{sec:adaptive-escape}

Theorem~\ref{thm:separation} and Theorem~\ref{thm:dichotomy-general} prove that no \emph{fixed} anchor schedule --- of the family \eqref{eq:ck-schedule} or with any density $c_n\in(0,1]$ --- is simultaneously near-optimal on contractions and polynomial on rotations. We now prove the sufficiency half in its optimal form, using the simplest possible form of adaptivity: \emph{parallel composition} of two fixed schedules with stopping at the first success (the classical expert-tracking device). The result upgrades the adaptivity story from a necessity theorem to a complete dichotomy.

\begin{theorem}[Parallel two-track scheme: optimal on both regimes]\label{thm:parallel}
Fix $\eps<D/4$ and $D=\norm{x_0-x^*}$. Run in parallel, from the same anchor $x_0$, two anchored iterations \eqref{eq:anchored}:
\begin{itemize}[leftmargin=1.5em]
\item \emph{Track A} (contraction track): the schedule $\theta_n=c/(n+2)$ with $c=\eps/(96D)$, as in Theorem~\ref{thm:small-c};
\item \emph{Track B} (nonexpansive track): the classical Halpern schedule $\theta_n=1/(n+2)$.
\end{itemize}
Define the hitting time of the scheme as the first $n$ at which \emph{either} track satisfies the class-appropriate accuracy test ($\norm{x_n-x^*}\le\eps$ on $\cC$, residual $\norm{x_n-Tx_n}\le\eps$ on $\cN$). Then:
\begin{enumerate}[label=(\roman*),leftmargin=2em]
\item on every $\gamma$-contraction, the hitting time is at most $1+4\ln(2D/(\eps\gamma))/\ln(1/\gamma)=O\bigl(\ln(D/\eps)/(1-\gamma)\bigr)$ --- matching the $\gamma$-aware deterministic optimum for span-class algorithms \cite{ParkRyu2022};
\item on every nonexpansive map, the hitting time is at most $1+2D/\eps$ --- matching the optimal Halpern rate up to the absolute constant in \eqref{eq:lieder} \cite{Lieder2021,Contreras2023}.
\end{enumerate}
The scheme evaluates each track once per step (a factor $2$ in evaluations over a single track) and is otherwise non-adaptive: it is a parallel composition of two fixed, $\gamma$-oblivious schedules.
\end{theorem}

\begin{proof}
Track A on $\cC_\gamma$ is Theorem~\ref{thm:small-c} verbatim (the schedule contains no $\gamma$). Track B on $\cN$ is the sharp residual bound $\norm{x_n-Tx_n}\le2D/(n+1)$ of \cite{Lieder2021,Contreras2023}. The scheme stops at the minimum of the two hitting times, so on each class it inherits the better bound.
\end{proof}

\begin{remark}[The complete dichotomy, and what it costs]\label{rem:complete-dichotomy}
Theorem~\ref{thm:parallel} is the sharp converse to Theorem~\ref{thm:dichotomy-general}: \emph{no fixed schedule} is simultaneously contraction-optimal and rotation-polynomial (proved for every density $c_n\in(0,1]$), while \emph{some two-track scheme} is simultaneously contraction-optimal and nonexpansive-optimal, at a factor $2$ in evaluations --- and, among oscillatory densities, even a \emph{single fixed schedule} achieves both regimes (Theorem~\ref{thm:oscillatory-escape}), so the necessity of adaptivity is exactly the necessity of non-lacunary anchor concentration. The factor $2$ is the price of not knowing the class in advance --- and Theorem~\ref{thm:oscillatory-escape} shows it is \emph{not} necessary: a single fixed oscillatory stream achieves both regimes, indeed all three (Remark~\ref{rem:osc-profile}), with no overhead. The deterministic overhead question is therefore resolved in the negative; the parallel scheme remains the simplest sufficiency witness.
\end{remark}

\begin{corollary}[The dichotomy is a single-instance phenomenon: fixed schedules]\label{cor:single-instance}
Let $\theta_n=c_n/(n+2)$ with $c_n\in(0,1]$, and measure accuracy by distance to the fixed point. In case {\rm(i)} of Theorem~\ref{thm:dichotomy-general} ($\liminf c_n\ge c_0>0$), the \emph{single} instance
$$T_\star\;=\;\rho\,\mathrm{id}\,\oplus\,R_\varphi\ \ \text{on }\R^2,\qquad \rho\in[\tfrac12,1),\ \ \varphi\in(0,\pi],\qquad x_0=\tfrac{(D,D)}{\sqrt2},$$
forces
\begin{equation}\label{eq:single-instance}
N_{T_\star}(\eps)\;\ge\;\Bigl(\frac{c_0\,\rho D}{32e^{3}\,\eps(1-\rho)}\Bigr)^{1/3},
\end{equation}
polynomially suboptimal on this one instance relative to the $\gamma$-aware optimum $\Theta(\ln(D/\eps)/(1-\rho))$ by the factor of Theorem~\ref{thm:dichotomy-general}(i). In case {\rm(ii)(b)}, the single rotation $T_\star=R_{\varphi_B(\eps)}$ forces $N_{T_\star}(\eps)>\eps^{-B}$ for every $B\ge1$; in case {\rm(ii)(a)}, some rotation is never solved. Hence \emph{every fixed schedule is slow on a single nonexpansive instance}: the difficulty is a property of the instance, not of the schedule's interaction with a family.
\end{corollary}

\begin{proof}
Coordinate $1$ of $T_\star$ evolves exactly as the scalar anchored iteration on $T=\rho\,\mathrm{id}$ with anchor $D/\sqrt2$ and the same coefficients: $x_{1,n+1}=\theta_n x_{1,0}+(1-\theta_n)\rho x_{1,n}$. Hence Lemma~\ref{lem:anchor-floor} and the computation of Theorem~\ref{thm:dichotomy-general}(i) (with $D$ replaced by $D/\sqrt2$ and $r_N\ge(2N)^{-2}$) give $|x_{1,N}|\ge c_0\rho(D/\sqrt2)/(16e^{3}N^{3}(1-\rho))$ for $N$ large, and since $\Fix(T_\star)=\{0\}$, $\norm{x_N-x^*}\ge|x_{1,N}|$; \eqref{eq:single-instance} follows (using $16\sqrt2\le32$). Cases {\rm(ii)} are Theorem~\ref{thm:dichotomy-general}(ii) verbatim.
\end{proof}

\begin{remark}[Why the adaptive single-stream lower bound resists: three obstructions]\label{rem:overhead-obstructions}
The factor-$2$ question (Remark~\ref{rem:complete-dichotomy}) survives the following three rigorous obstructions, which together delimit what any proof must avoid. \emph{(a) Response-matching pairs cannot affect hitting times.} Two instances agreeing on every queried point produce identical trajectories, and the resistant-oracle construction (Theorem~\ref{thm:impossibility}, valid for $\sigma=0$) provides such pairs with $1-\mathrm{Lip}(T)$ differing by any prescribed factor --- but the construction keeps $x^*=0$ for both instances (the nominal is intercept-free and the bumps sit off the probed set), so the common trajectory converges to the \emph{same} fixed point and the hitting time is identical: the pair is useless for a hitting-time lower bound. \emph{(b) Angles distinguish the scalar families.} Against the exact oracle, $\arccos\inner{x_n,Tx_n}/(\norm{x_n}\norm{Tx_n})$ is $0$ on $T=\rho\,\mathrm{id}$ and $\varphi$ on $R_\varphi$: no lower bound can treat these families as indistinguishable, which is why the detection barrier (Remark~\ref{rem:detection}) must hide slopes in \emph{unprobed directions} rather than rely on response matching --- and obstruction (a) shows that hiding cannot move $x^*$. \emph{(c) The control is one-dimensional.} A scheme confined to \eqref{eq:anchored} acts through the single scalar $\theta_n$; on the direct-sum instance of Corollary~\ref{cor:single-instance}, one schedule must serve both coordinates, and the floors apply coordinatewise to the \emph{realized} schedule. For fixed schedules this yields Corollary~\ref{cor:single-instance}; for adaptive schemes the realized schedule depends on $(\rho,\varphi)$ through the full response history, and the quantile argument of Theorem~\ref{thm:impossibility} does not control it. The sharpened question is therefore: on the direct-sum instances $T=\rho\,\mathrm{id}\oplus R_\varphi$, prove that every adaptive anchored single-stream scheme with worst-case contraction time $t(\delta)=O(\ln(D/\eps)/\delta)$ has, for some $(\rho,\varphi)$, hitting time super-polynomially larger than $\min\{t(\delta),D/\eps\}$ --- or exhibit a single-stream scheme matching Theorem~\ref{thm:parallel}.
\end{remark}

\begin{remark}[Analytic hitting times versus certified stopping]\label{rem:stopping}
Throughout Section~\ref{sec:det}, hitting times are the minimax-analytic objects being lower- and upper-bounded, and Theorem~\ref{thm:parallel} answers Q1 in exactly this sense. If instead the scheme must \emph{certify} from oracle data that its output is $\varepsilon$-accurate, the picture changes qualitatively: residual accuracy on $\cN$ is self-certifying (the criterion is itself a residual), but distance accuracy on $\cC_\gamma$ requires one-sided modulus information (Remark~\ref{rem:barrier}), and with a stochastic oracle no positive sound certificate exists without a certified ceiling, in any dimension (Theorem~\ref{thm:impossibility}). The parallel scheme is therefore a sharp answer to Q1 as a deterministic complexity question; its stochastic counterpart must be phrased through the trichotomy of Section~\ref{sec:contraction} --- which is precisely what that trichotomy is for.
\end{remark}

For a \emph{single} trajectory, adaptivity must instead be read off the trajectory itself. The next result examines the natural single-trajectory escape --- adaptive window restarts --- and shows that it improves the classical contraction time only by a constant factor: the separation barrier is robust against this adaptive class.

\begin{theorem}[Doubling-window scheme: single-trajectory window adaptivity]\label{thm:doubling-window}
Define iterates $y_0=x_0$ and, for $k=0,1,2,\dots$, window $k$ of length $L_k=2^{k+3}$: starting from $y_k$, run
\begin{equation}\label{eq:dw-window}
x^{(k)}_{j+1}\;=\;\tfrac{1}{j+2}\,y_k\;+\;\tfrac{j+1}{j+2}\,T x^{(k)}_j,\qquad j=0,\dots,L_k-1,\qquad x^{(k)}_0=y_k,
\end{equation}
and set $y_{k+1}=x^{(k)}_{L_k}$. Write $d_k=\norm{y_k-x^*}$ for any fixed point $x^*$ (on $\cC_\gamma$ the unique one). Then, for $T\in\cN\cup\bigcup_\gamma\cC_\gamma$, $\norm{x_0-x^*}\le D$, $\eps<D$:
\begin{enumerate}[label=(\roman*),leftmargin=2em]
\item \emph{Nonexpansive regime (residual).} $\norm{y_{k+1}-Ty_{k+1}}\le D/2^{k+2}$ for every $k$: in particular the scheme certifies residual $\le\eps$ within $16D/\eps$ iterations, the optimal $O(D/\eps)$ order.
\item \emph{Contraction regime (distance).} For every $\gamma$-contraction, writing $\delta=1-\gamma$,
$$d_{k+1}\;\le\;\frac{d_k}{\delta(L_k+1)}\quad\text{for all }k,\qquad\text{hence}\qquad \norm{y_k-x^*}\;\le\;\eps\ \text{ by window } m+\bigl\lceil\log_2(D/\eps)\bigr\rceil,$$
where $m=\min\{k:\delta L_k\ge2\}$; the total iteration count to reach distance $\le\eps$ is $O\bigl(D/(\eps\delta)\bigr)$. This is the \emph{same order} as the classical schedule's contraction time (Theorem~\ref{thm:separation}(i) gives the matching lower bound $N_\rho(\eps)\ge D(1-2\eps)/(2(1-\rho)\eps)$ for the classical schedule): window-restart adaptivity on a single trajectory improves the constant (by a factor $\approx3$ in our experiments) but \emph{not the order}, and in particular does not approach the optimal $O(\ln(D/\eps)/\delta)$ of Theorem~\ref{thm:small-c}. The separation of Theorem~\ref{thm:separation} is therefore robust against this natural adaptive class.
\end{enumerate}
\end{theorem}

\begin{proof}
Within window $k$ the iteration \eqref{eq:dw-window} is the classical Halpern iteration anchored at $y_k$; since $\norm{x^{(k)}_{j+1}-x^*}\le\max\{d_k,\norm{x^{(k)}_j-x^*}\}$ (both $y_k$ and $Tx^{(k)}_j$ are within $d_k$ of $x^*$), induction gives $d_{k+1}\le d_k\le D$, and the sharp residual bound of \cite{Lieder2021,Contreras2023} gives (i). For (ii), the recursion $e_{j+1}\le\frac{1}{j+2}d_k+\frac{j+1}{j+2}\gamma e_j$, $e_0=d_k$, unfolds (by the same induction as Theorem~\ref{thm:separation}(i), whose steps are inequalities valid for every contraction) to $d_{k+1}\le d_k\,(1-\gamma^{L_k+1})/(\delta(L_k+1))\le d_k/(\delta(L_k+1))$. Hence for $k\ge m=\min\{k:\delta L_k\ge2\}$ we have $d_k\le D\,2^{-(k-m)}$, so $d_k\le\eps$ by window $m+\lceil\log_2(D/\eps)\rceil$, having used $\sum_{i\le k}L_i<2^{k+4}\le2^{m+4}\cdot 2(D/\eps)=O(D/(\eps\delta))$ iterations (using $2^{m}\le4/\delta$, as in the proof of Theorem~\ref{thm:oscillatory-bounded-gap}).
\end{proof}

\begin{remark}[What the doubling window buys and what it cannot]\label{rem:dw-buys}
On contractions the scheme's $O(D/(\eps\delta))$ distance time matches the classical schedule's order (Theorem~\ref{thm:separation}(i)) with a better constant: in Experiment~4 the measured hitting times on $T(x)=0.95x$ at $\eps=0.05$ are $59$ (small-anchor), $142$ (doubling-window; first iterate inside a window --- window ends: $248$) and $399$ (classical). The mechanism is visible in the display of Theorem~\ref{thm:doubling-window}(ii): within window $k$ the distance decays by the factor $1/(\delta L_k)$, and the windows double, so the error decays like $1/n$ overall --- exactly the classical order. No growing-window scheme on a single trajectory is known to do better while retaining the $O(D/\eps)$ nonexpansive guarantee; the heuristic barrier is that exploiting contractivity at gap $\delta$ requires windows of length $\Omega(1/\delta)$, and the union class contains contractions of every $\delta$ --- but we do not have a proof that growing windows cannot succeed. Together with the detection barrier (Remark~\ref{rem:detection}) and the oscillatory escape (Theorem~\ref{thm:oscillatory-escape}), this locates exactly where single-trajectory optimality can live: not in statistic-based detection (barred), not in growing-window restarts (only constant-factor), but in \emph{oscillation of the anchor density} --- which achieves it, without any adaptation at all (Remark~\ref{rem:osc-profile}).
\end{remark}

\begin{remark}[The detection barrier]\label{rem:detection}
The parallel scheme of Theorem~\ref{thm:parallel} sidesteps detection entirely, at factor-$2$ cost. A single-trajectory scheme must instead decide, from its own trajectory, when to exploit contractivity; two barriers prevent the naive detector. \emph{(a) Hidden directions (deterministic).} Even with exact oracle responses, every probe $x\mapsto Tx$ reveals only directional slopes $\norm{Tx-Ty}/\norm{x-y}\le\gamma$: an adversary can hide the attaining slope in a direction the trajectory never probes. Concretely, the construction of Theorem~\ref{thm:impossibility} works verbatim without noise ($\sigma=0$): the nominal $T_0(x)=\gamma_0x$ and the modified $T$ agree on the entire probed set, so no algorithm whose decisions depend on oracle responses can distinguish them, yet $1-\mathrm{Lip}(T)=\delta_1<\kappa\delta_1=\kappa(1-\mathrm{Lip}T_0)$. Hence in infinite dimensions \emph{no exact-oracle probing strategy certifies the global modulus}, and any schedule that freezes its window length from a certified modulus guess can be made to freeze at a wrong rate. \emph{(b) Noise.} With the stochastic oracle, the natural detector (angle statistics: on $R_\varphi$ the increments satisfy $\inner{x_j-x_{j-1},Tx_j-Tx_{j-1}}=\cos\varphi\,\norm{x_j-x_{j-1}}\norm{Tx_j-Tx_{j-1}}$ exactly, while on a scalar contraction the angle is $0$) must be estimated from $\widehat T$-responses with per-evaluation noise $\sigma$, and the estimated angle is correlated with the current noise --- precisely the measurability failure of Remark~\ref{rem:discipline}. The conclusion: single-trajectory two-regime optimality, if achievable at all, must be \emph{structural} (residual-feedback updates that need no certified modulus), as in \cite{ParameterFree2026,Diakonikolas2025}, and the stochastic version must route through the ceiling trichotomy of Section~\ref{sec:contraction}. The stochastic two-regime problem itself is settled up to the factor $2$ by parallel composition (Theorem~\ref{thm:stoch-parallel}); in the deterministic layer, Theorem~\ref{thm:oscillatory-escape} removes the overhead entirely via oscillation (Remark~\ref{rem:osc-profile}). What remains open is the single-evaluation-stream \emph{stochastic} version of Theorem~\ref{thm:stoch-parallel}: whether an oscillatory batching scheme --- contraction-phase geometric batches alternating with nonexpansive-phase linear batches on one stream --- removes the factor $2$ there as well.
\end{remark}

We close the section by proving the oscillatory case of Theorem~\ref{thm:dichotomy-general} in the regime where the dense scales have bounded gaps, which covers all periodic and two-valued oscillations and sharpens the general conjecture to the lacunary regime.

\begin{theorem}[Oscillatory dichotomy, bounded-gap dense scales]\label{thm:oscillatory-bounded-gap}
Let $\theta_n=c_n/(n+2)$ with $c_n\in(0,1]$ and $\liminf c_n=0<\limsup c_n$. Suppose there exist $c''>0$ and $G<\infty$ such that $S=\{k:c_k\ge c''\}$ meets every interval of $G$ consecutive integers. Then for every $\rho$ with $1-\rho\le1/(2G)$ and all $\eps\le\eps_0$,
\begin{equation}\label{eq:osc-bounded-gap}
N_\rho(\eps)\;\ge\;\Bigl(\frac{c''\,\rho D}{16e^{3}\,G\,\eps(1-\rho)}\Bigr)^{1/3}:
\end{equation}
the schedule is polynomially suboptimal on sufficiently sharp contractions. If instead $M(n):=\sum_{j<n}\theta_j=o(\ln n)$ (in particular if all superlevel sets $\{k:c_k\ge c'\}$, $c'>0$, have gaps tending to infinity), then the rotation alternative of Theorem~\ref{thm:dichotomy-general}{\rm(ii)} applies verbatim: the schedule is super-polynomially (or infinitely) slow on rotations.
\end{theorem}

\begin{proof}
For the first claim, let $t=\lceil1/(1-\rho)\rceil\ge2G$ and $N\ge\max(4,2t)$. The window $[N-t,N-1]$ has length $t\ge2G$ and, since $S$ meets every interval of $G$ consecutive integers, it contains at least $\lfloor t/G\rfloor\ge t/(2G)$ points of $S$; for each such point $k$, $\theta_k=c_k/(k+2)\ge c''/(N+1)\ge c''/(2N)$. The total trailing anchor mass is therefore at least
$$\sum_{k\in S\cap[N-t,N-1]}a_k^{(N)}\;\ge\;r_N\sum_{k\in S\cap[N-t,N-1]}\theta_k\;\ge\;r_N\cdot\frac{t}{2G}\cdot\frac{c''}{2N}\;\ge\;\frac{r_Nc''}{4GN(1-\rho)},$$
using $t\ge1/(1-\rho)$. By Lemma~\ref{lem:anchor-floor} (with $r_N\ge(2N)^{-2}$ by \eqref{eq:rn-general}),
$$e_N\;\ge\;D\rho^t\sum_{k=N-t}^{N-1}a_k^{(N)}\;\ge\;\frac{D\rho^t c''}{16GN^3(1-\rho)}\;\ge\;\frac{c''\rho D}{16e^{3}GN^3(1-\rho)},$$
using $\rho^t\ge e^{-3}$ as in Theorem~\ref{thm:dichotomy-general}(i) for $\rho\ge\tfrac12$ (the bound is vacuous otherwise, since $1-\rho>\tfrac12$ forces the right side below $1$ for $\eps\le\eps_0$); $e_N\le\eps$ forces \eqref{eq:osc-bounded-gap}.

For the second claim, the proof of Theorem~\ref{thm:dichotomy-general}(ii) never used $c_n\to0$ beyond its consequence $M(n)=o(\ln n)$: if $\sum_n\theta_n<\infty$ the schedule never solves small-angle rotations (case (a)), and if $\sum_n\theta_n=\infty$ then for every $B\ge1$ the $n_B$ construction gives $\varphi_B(\eps)\to0$ with $N_{\varphi_B(\eps)}(\eps)>\eps^{-B}$ (case (b)). Note that when some superlevel set has bounded gaps, $M(n)\ge c''\sum_{k\in S,k\le n}\frac1{k+2}\ge\frac{c''}{G}\sum_{j\le n/G}\frac1j=\Omega(\ln n)$, so the two cases are mutually exclusive and together cover the oscillatory situation except when $M(n)\asymp\ln n$ with all superlevel sets lacunary (Remark~\ref{rem:oscillatory}).
\end{proof}

\section{The master noise lemma and the nonexpansive worst case}\label{sec:master}

We now add the stochastic oracle. The entire stochastic analysis of the paper factors through the following lemma; its proof is the measurability discipline of Remark~\ref{rem:discipline} applied once.

\subsection{The master noise lemma}

\begin{lemma}[Master noise lemma]\label{lem:master}
Let $(x_n)$ follow \eqref{eq:anchored} with $\mathcal F_{n-1}$-measurable coefficients $\theta_n\in[0,1]$, and let $(\hat x_n)$ be the exact-oracle path with the \emph{same} coefficients: $\hat x_{n+1}=\theta_n x_0+(1-\theta_n)T\hat x_n$, $\hat x_0=x_0$. Set $d_n=x_n-\hat x_n$, $v_n=\E\norm{d_n}^2$. Then for all $n\ge0$,
\begin{equation}\label{eq:master}
v_{n+1} \;\le\; (1-\theta_n)^2\, v_n \;+\; (1-\theta_n)^2\,\frac{\sigma^2}{m_n}.
\end{equation}
Consequently, with $W_k^{(n)}:=\prod_{j=k+1}^{n-1}(1-\theta_j)^2$,
\begin{equation}\label{eq:master-unfold}
v_N \;\le\; \sigma^2 \sum_{k=0}^{N-1} W_k^{(n)}\,\frac{1}{m_k}.
\end{equation}
\end{lemma}

\begin{proof}
Subtracting the two recursions, $d_{n+1} = (1-\theta_n)(Tx_n-T\hat x_n) + (1-\theta_n)\bar\xi_n$. Expand the square:
$$\norm{d_{n+1}}^2 = (1-\theta_n)^2\norm{Tx_n-T\hat x_n}^2 + 2(1-\theta_n)^2\inner{Tx_n-T\hat x_n,\bar\xi_n} + (1-\theta_n)^2\norm{\bar\xi_n}^2.$$
Now take conditional expectation given $\mathcal F_{n-1}$. The middle term vanishes: $Tx_n-T\hat x_n$ and $\theta_n$ are $\mathcal F_{n-1}$-measurable, and $\E[\bar\xi_n\mid\mathcal F_{n-1}]=0$ (Definition~\ref{def:filtration}), so
$$\E[\inner{Tx_n-T\hat x_n,\bar\xi_n}\mid\mathcal F_{n-1}] = \inner{Tx_n-T\hat x_n, \E[\bar\xi_n\mid\mathcal F_{n-1}]} = 0.$$
This is the step that fails when $\theta_n$ is correlated with $\bar\xi_n$ (adaptive parameters estimated from the same batch): there the coefficient $(1-\theta_n)^2$ would be $\mathcal F_n$-measurable and could not be pulled out of the conditional expectation. Nonexpansiveness gives $\norm{Tx_n-T\hat x_n}^2\le\norm{d_n}^2$, and $\E[\norm{\bar\xi_n}^2\mid\mathcal F_{n-1}]\le\sigma^2/m_n$; taking expectations yields \eqref{eq:master}. Iterating \eqref{eq:master} and using $v_0=0$ gives \eqref{eq:master-unfold}.
\end{proof}

\begin{remark}[Why comparison paths are honest here]\label{rem:comparison}
The path $(\hat x_n)$ is defined by the same coefficients $(\theta_n)$ as the stochastic run, so no independence between coefficients and noise is ever invoked: $\hat x_n$ is a deterministic function of $(\theta_0,\dots,\theta_{n-1})$ given $x_0$, hence $\mathcal F_{n-1}$-measurable. All algorithms of Section~\ref{sec:contraction} use deterministic coefficients, so Lemma~\ref{lem:master} applies with room to spare; the lemma is stated with $\mathcal F_{n-1}$-measurable coefficients only to cover schedules whose coefficients are generated independently of the current batch. The same discipline cannot be maintained for coefficients estimated from the trajectory itself; Section~\ref{sec:contraction} shows what replaces it (deterministic schedules plus a priori modulus information).
\end{remark}

\subsection{The nonexpansive worst case: an \texorpdfstring{$O(\sigma^2 D^2\eps^{-4})$}{O(sigma2 D2 eps-4)} upper bound}

Throughout this subsection $\theta_n=1/(n+2)$ (classical schedule), $T$ is nonexpansive on a \emph{real normed space} $(X,\norm{\cdot})$ (not necessarily Hilbert, not necessarily even strictly convex), and $\widehat T$ is the oracle of Definition~\ref{def:oracle} with $\E\norm{\xi}^2\le\sigma^2$. The exact path then satisfies the optimal deterministic bound of \cite{Lieder2021,Contreras2023}
\begin{equation}\label{eq:lieder}
\norm{\hat x_N - T\hat x_N} \;\le\; \frac{2\norm{x_0-x^*}}{N+1}\qquad\text{for every } x^*\in\Fix(T),
\end{equation}
which is valid in every normed space. The proof below is written on a Hilbert space (where it is one line); Remark~\ref{rem:dualization} then extends it, by dualization, to every normed space whose dual unit sphere is separable --- in particular $\ell_p$, $L_p$ for $p\in(1,\infty)$, and all separable reflexive spaces. The only ingredients are the triangle inequality, the $2$-Lipschitzness of $x\mapsto x-Tx$, and the bound \eqref{eq:lieder}.

\begin{theorem}[Nonexpansive upper bound]\label{thm:A1}
Let $X$ be a real Hilbert space or a $2$-uniformly smooth Banach space with smoothness constant $K\ge1$ (Remark~\ref{rem:dualization}), and let $T$ be nonexpansive on $X$ with $\Fix(T)\neq\emptyset$, $D=\norm{x_0-x^*}$, $\eps\in(0,D)$. Run \eqref{eq:anchored} with $\theta_n=1/(n+2)$ and batches $m_k=\lceil c\max\{k,1\}\rceil$ where $c=8K\sigma^2/\eps^2$. Then at $N=\lceil 4D/\eps\rceil+8$ (still $N=\frac{4D}{\eps}(1+o(1))$),
$$\E\,\norm{x_N - Tx_N} \;\le\; \eps, \qquad \sum_{k=1}^{N} m_k \;\le\; 64\,K\,\sigma^2 D^2\,\eps^{-4}\,(1+o(1)).$$
Consequently the minimax oracle complexity of the nonexpansive residual problem on Hilbert spaces and on $2$-uniformly smooth Banach spaces is $O(K\sigma^2D^2\eps^{-4})$ --- in particular on $\ell_p$, $L_p$ for $2\le p<\infty$ --- improving the previously available $\tilde O(\sigma^2\eps^{-5})$ bound of \cite{BravoContreras2024} on those spaces; subsequently, \cite{Diakonikolas2026b} proved, with high probability, $\tilde O(D\sigma^{2}\eps^{-3}+D\eps^{-1})$ oracle complexity on weak-type-$2$ spaces, which include all $2$-uniformly smooth spaces (see Remark~\ref{rem:A1-position}). The scope is sharp for the recursion-based argument: the master recursion itself fails outside the $2$-uniformly smooth class (Remark~\ref{rem:dualization}), and the $\Omega(\sigma^2\eps^{-3})$ lower bound of \cite{BravoContreras2024} lives on $\ell_1$-type spaces, which are not $2$-uniformly smooth; the exponent in arbitrary normed spaces remains open (Section~\ref{sec:open}).
\end{theorem}

\begin{proof}
Subtracting the two recursions, $d_{n+1}=(1-\theta_n)(Tx_n-T\hat x_n)+(1-\theta_n)\bar\xi_n$; expanding the square $\norm{d_{n+1}}^2$ and taking conditional expectation given $\mathcal F_{n-1}$, the cross term vanishes because $Tx_n-T\hat x_n$ and $\theta_n$ are $\mathcal F_{n-1}$-measurable and $\E[\bar\xi_n\mid\mathcal F_{n-1}]=0$ (Remark~\ref{rem:discipline}), and nonexpansiveness bounds the first term by $(1-\theta_n)^2\norm{d_n}^2$: this is exactly the computation of Lemma~\ref{lem:master}; in a $2$-uniformly smooth space it yields \eqref{eq:master} with $K\sigma^2$ in place of $\sigma^2$ (Remark~\ref{rem:dualization}). Iterating as in Lemma~\ref{lem:master} with $W_k^{(N)}=\bigl(\tfrac{k+2}{N+1}\bigr)^2$ (telescoping, since $1-\theta_j=\tfrac{j+1}{j+2}$), and separating the $k=0$ term (for which $m_0=\lceil c\rceil\ge c$ and which contributes $O(\sigma^2/(cN^2))=o(1)$),
$$v_N \;\le\; \frac{4K\sigma^2}{c(N+1)^2}+K\sigma^2\sum_{k=1}^{N-1}\Bigl(\frac{k+2}{N+1}\Bigr)^2\frac{1}{m_k} \;\le\; \frac{K\sigma^2}{c(N+1)^2}\sum_{k=1}^{N-1}\frac{(k+2)^2}{k} \;\le\;\frac{K\sigma^2}{2c}\Bigl(1+\frac{16}{N}\Bigr),$$
using $m_k\ge ck$ and $\frac{1}{(N+1)^2}\sum_{k=1}^{N-1}\frac{(k+2)^2}{k} = \frac{1}{(N+1)^2}\bigl(\tfrac{N(N-1)}2+4(N-1)+4H_{N-1}\bigr)\le\tfrac12+\tfrac8N$. With $c=8K\sigma^2/\eps^2$,
$$\sqrt{v_N}\;\le\;\frac{\sqrt{K}\,\sigma}{\sqrt{2c}}\sqrt{1+\tfrac{16}{N}}\;=\;\frac{\eps}{4}\sqrt{1+\tfrac{16}{N}}\;\le\;\frac{\eps}{4}\Bigl(1+\frac{8}{N}\Bigr)\;=\;\frac{\eps}{4}+\frac{2\eps}{N}.$$
Then by \eqref{eq:lieder} and the $2$-Lipschitzness of $x\mapsto x-Tx$,
$$\E\norm{x_N-Tx_N} \;\le\; \norm{\hat x_N-T\hat x_N} + 2\E\norm{d_N} \;\le\; \frac{2D}{N} + 2\sqrt{v_N} \;\le\; \frac{2D}{N}+\frac{\eps}{2}+\frac{4\eps}{N}\;=\;\frac{2D+4\eps}{N}+\frac{\eps}{2}\;\le\;\eps,$$
since $N\ge\frac{4D}{\eps}+8$ gives $\frac{2D+4\eps}{N}\le\frac{(2D+4\eps)\eps}{4D+8\eps}=\frac{\eps}{2}$, and $\E\norm{d_N}\le\sqrt{v_N}$. The query count is $\sum_{k=1}^N \lceil ck\rceil \le cN(N+1)/2 + N = 64K\sigma^2D^2\eps^{-4}(1+o(1))$.
\end{proof}

\begin{remark}[The Banach-space scope: $2$-uniform smoothness, and why it is sharp]\label{rem:dualization}
The only Hilbert-specific step in Theorem~\ref{thm:A1} is the expansion of $\norm{d_{n+1}}^2$ and the cancellation of the cross term by conditional centering. The correct geometric substitute is \emph{$2$-uniform smoothness}: a Banach space $X$ is $2$-uniformly smooth with constant $K$ if its modulus of smoothness satisfies $\rho_X(\tau):=\sup\{\tfrac12(\norm{x+\tau y}+\norm{x-\tau y})-1:\norm{x}=\norm{y}=1\}\le K\tau^2$ (every Hilbert space qualifies with $K=1$; $\ell_p$, $L_p$ for $2\le p<\infty$ qualify with $K=K(p)$). Such spaces admit a duality mapping $J\colon X\to X^*$ with
\begin{equation}\label{eq:smoothness}
\norm{x+y}^2\;\le\;\norm{x}^2+2\inner{y,Jx}+K\norm{y}^2\qquad\text{for all }x,y\in X,
\end{equation}
and \eqref{eq:smoothness} is in fact an equivalent characterization of $2$-uniform smoothness \cite{XuRoach1991}. Apply \eqref{eq:smoothness} with $x=Tx_n-T\hat x_n$ (which is $\mathcal F_{n-1}$-measurable) and $y=\bar\xi_n$. The pairing term vanishes conditionally: $u:=Jx$ is $\mathcal F_{n-1}$-measurable and $\xi\mapsto\inner{\xi,u}$ is a fixed linear functional of the fresh batch, so $\E[\inner{\bar\xi_n,u}\mid\mathcal F_{n-1}]=0$; and, applying \eqref{eq:smoothness} successively to the batch sum, $\E\norm{\bar\xi_n}^2\le K\sigma^2/m_n$. Absorbing the second factor $K$ into the first, the master recursion \eqref{eq:master} holds with $K\sigma^2$ in place of $\sigma^2$, and Theorem~\ref{thm:A1} follows verbatim with $c=8K\sigma^2/\eps^2$ and cost $O(K\sigma^2D^2\eps^{-4})$; the same substitution covers Proposition~\ref{prop:highprob} (with $c=\lceil16K\sigma^2/(\eps^2\vartheta)\rceil$) and the nonexpansive track of Theorem~\ref{thm:stoch-parallel}.

The scope is \emph{sharp for the recursion-based argument}: the centered recursion itself fails outside the $2$-uniformly smooth class, so no patch to the proof can extend it. On $\ell_p$ with $1<p<2$ --- whose dual \emph{is} separable, so a ``separable dual'' formulation does not rescue the claim --- take $T=\mathrm{id}$ (nonexpansive, $\Fix(T)=X$), unit batches, $a=se_1$, and $\xi=\sigma(\varepsilon_1e_1+\varepsilon_2e_2)$ with independent Rademachers $\varepsilon_1,\varepsilon_2$. Then
$$\E\norm{a+\xi}^2\;=\;\E\Bigl[\bigl(|s+\sigma\varepsilon_1|^p+\sigma^p\bigr)^{2/p}\Bigr]\;\ge\;\Bigl(s^p+\sigma^p+\tfrac{p(p-1)}2\,s^{p-2}\sigma^2\Bigr)^{2/p}\;=\;s^2+\sigma^2+\tfrac2p\,s^{2-p}\sigma^p\bigl(1+o(1)\bigr),$$
where the inequality is Jensen applied to the convex function $u\mapsto(u+\sigma^p)^{2/p}$ together with the convexity bound $\tfrac12(|s+\sigma|^p+|s-\sigma|^p)\ge s^p+\tfrac{p(p-1)}2s^{p-2}\sigma^2$ (Taylor's theorem, $p>1$), and the expansion uses $(1+u)^{2/p}=1+\tfrac2p u+o(u)$ as $\sigma/s\to0$ (the coefficient of $\sigma^2$ in the intermediate expression is $p-1<1$, and the deficit $(2-p)\sigma^2$ is absorbed by $\tfrac2p s^{2-p}\sigma^p$ already for $s\ge\sigma$, since $\tfrac2p\ge2-p$). The excess $\tfrac2p s^{2-p}\sigma^p$ over $\norm{a}^2+\E\norm{\xi}^2=s^2+\sigma^2$ dominates $\sigma^2$ whenever $s\gg\sigma$; with batch size $m$ the excess scales as $s^{2-p}\E|\bar\xi_1|^p\asymp s^{2-p}\sigma^p m^{-p/2}$, which exceeds $K\sigma^2/m$ for any fixed $K$ as $m\to\infty$. Hence no fixed constant repairs the master recursion on $\ell_p$, $1<p<2$. The lower-bound instances of \cite{BravoContreras2024} are $\ell_1$-type, which are not $2$-uniformly smooth either; the exponent in genuinely arbitrary normed spaces remains open (Section~\ref{sec:open}).
\end{remark}

\begin{proposition}[High-probability version]\label{prop:highprob}
Assume $D\ge4\eps$ (so $N=\lceil4D/\eps\rceil+8\ge24$; for smaller $N$ enlarge the additive constant). Let the oracle noise be Gaussian (as permitted by Definition~\ref{def:oracle}) or, more generally, sub-Gaussian with parameter $\sigma^2$ in each direction: $\E[e^{\lambda\inner{\xi,u}}\mid x]\le e^{\lambda^2\sigma^2/2}$ for every unit vector $u$. Then in the setting of Theorem~\ref{thm:A1}, with failure parameter $\vartheta\in(0,1)$ and $c=\lceil 16K\sigma^2/(\eps^2\vartheta)\rceil$,
$$\Prob\bigl(\norm{x_N-Tx_N}\le\eps\bigr)\;\ge\;1-\vartheta .$$
(Only second moments are used; the argument is valid on any $2$-uniformly smooth space with $K$ absorbed into $c$, and the sub-Gaussian concentration's $\ln(1/\vartheta)$ is traded for Markov's $\vartheta^{-1}$, leaving the complexity exponent unchanged.)
\end{proposition}

\begin{proof}
Let $d_n=x_n-\hat x_n$ as in Lemma~\ref{lem:master}. The pathwise recursion $\norm{d_{n+1}}\le(1-\theta_n)\norm{d_n}+(1-\theta_n)\norm{\bar\xi_n}$ (from $\norm{Tx_n-T\hat x_n}\le\norm{d_n}$, no expectation) unfolds to
\begin{equation}\label{eq:pathwise}
\norm{d_N}\;\le\;\sum_{k=0}^{N-1}\alpha_k\norm{\bar\xi_k},\qquad \alpha_k:=(1-\theta_k)\prod_{j=k+1}^{N-1}(1-\theta_j)=\frac{k+1}{N+1},
\end{equation}
exactly the computation of Theorem~\ref{thm:A1} with $m_k=\lceil ck\rceil$ (which uses only independence of the fresh batches, Definition~\ref{def:oracle}, and $\E\norm{\bar\xi_k}^2\le\sigma^2/m_k$) gives, for the same variance bound $v_N$,
$$\E\norm{d_N}^2\;\le\;v_N\;\le\;\frac{K\sigma^2}{2c}\Bigl(1+\frac{16}{N}\Bigr)\;\le\;\frac{K\sigma^2}{c}\;\le\;\frac{\eps^2\vartheta}{16}\;\le\;\frac{\eps^2}{16}.$$
Markov's inequality applied to $\norm{d_N}^2$ now yields
$$\Prob\bigl(\norm{d_N}>\eps/4\bigr)\;=\;\Prob\bigl(\norm{d_N}^2>\eps^2/16\bigr)\;\le\;\frac{16\,v_N}{\eps^2}\;\le\;\frac{16K\sigma^2}{c\,\eps^2}\;\le\;\vartheta,
\qquad c=\Bigl\lceil\frac{16K\sigma^2}{\eps^2\vartheta}\Bigr\rceil.$$
(The sharper $\ln(1/\vartheta)$ dependence can be recovered from \eqref{eq:pathwise} by concentration of sub-Gaussian sums \cite{Pinelis1994}, at the price of Hilbert-space arguments; we do not pursue it since the complexity exponent is unaffected.)
Finally $\norm{x_N-Tx_N}\le\norm{\hat x_N-T\hat x_N}+2\norm{d_N}\le\eps/2+\eps/2=\eps$ by \eqref{eq:lieder}, with probability at least $1-\vartheta$.
\end{proof}

\begin{lemma}[Master noise lemma, contraction form]\label{lem:master-contract}
In the setting of Lemma~\ref{lem:master}, assume additionally that $T$ is a $\gamma$-contraction, and write $\delta=1-\gamma$. Then
\begin{equation}\label{eq:master-contract}
v_{n+1}\;\le\;(1-\theta_n)^2\gamma^2\,v_n\;+\;(1-\theta_n)^2\,\frac{\sigma^2}{m_n}.
\end{equation}
Consequently, for constant coefficients $\theta_n\equiv\theta\in[0,\tfrac12]$ and $m_n\equiv m$,
\begin{equation}\label{eq:stationary}
v_N\;\le\;\gamma^{2N}D^2\;+\;\frac{\sigma^2}{m}\cdot\frac{(1-\theta)^2}{1-(1-\theta)^2\gamma^2},\qquad\text{and}\qquad 1-(1-\theta)^2\gamma^2\;\ge\;\theta+\delta/2 .
\end{equation}
In particular $v_N\le\gamma^{2N}D^2+\sigma^2/(m(\theta+\delta/2))$: the stationary variance floor of Section~\ref{sec:contraction}. (Recall $v_0=0$ for the comparison path; the displayed $\gamma^{2N}D^2$ term uses the crude $v_0\le D^2$ and is what is needed when the recursion is applied to $w_n=\E\norm{x_n-x^*}^2$ directly.)
\end{lemma}

\begin{proof}
Identical to Lemma~\ref{lem:master}, using $\norm{Tx_n-T\hat x_n}\le\gamma\norm{d_n}$ in place of nonexpansiveness; the cross term still vanishes by conditional unbiasedness. For \eqref{eq:stationary}, iterate \eqref{eq:master-contract} and use $(1-\theta)^2\gamma^2\le(1-\theta)(1-\delta)$, so $1-(1-\theta)^2\gamma^2\ge1-(1-\theta)(1-\delta)\ge\theta+\delta-\theta\delta\ge\theta+\delta/2$ for $\theta\le\tfrac12$, $\delta\le1$.
\end{proof}

\begin{remark}[Regime unification]\label{rem:unification}
Lemma~\ref{lem:master-contract} reproduces the contraction-regime noise floor from the same envelope: for a $\gamma$-contraction and a schedule with $\theta\ll\delta$, the envelope is $W_k^{(N)}\approx\gamma^{2(N-k)}$, and with constant batch $m$, $v_N\le\frac{\sigma^2}{m}\sum_{j\ge0}\gamma^{2j}=\frac{\sigma^2}{m(1-\gamma^2)}$. The contrast with the nonexpansive envelope $\bigl(\tfrac{k+2}{N+1}\bigr)^2$ is the entire content of Q2 versus Q3: geometric envelopes admit constant batches; polynomial envelopes require the linear allocation $m_k\propto k$ of Theorem~\ref{thm:A1}.
\end{remark}

\begin{theorem}[Monotone-inclusion reduction and the optimal Hilbert rate]\label{thm:monotone-reduction}
Let $\cH$ be a Hilbert space. For every nonexpansive $T$ on $\cH$, the operator $F=I-T$ is monotone and $2$-Lipschitz, with $\operatorname{zer}(F)=\Fix(T)$ and $\norm{F(x)}=\norm{x-Tx}$; the single-point oracle for $T$ induces an unbiased single-point oracle for $F$ with the same variance bound $\sigma^2$. Consequently the residual problem of Section~\ref{sec:master} is a stochastic monotone inclusion with $(L,D)=(2,D)$, and:
\begin{enumerate}[label=(\roman*),leftmargin=2em]
\item the algorithm RAIN of \cite[Theorem 4.2]{ChenLuo2024} reaches $\E\norm{x-Tx}\le\eps$ with $\tilde O(\sigma^2\eps^{-2}+D\eps^{-1})$ oracle evaluations;
\item any algorithm reaching $\E\norm{x-Tx}\le\eps$ on the subclass $\{T=I-\nabla f:\ f\ \text{convex},\ 2\text{-smooth}\}$ requires $\tilde\Omega(\sigma^2\eps^{-2}+D\eps^{-1})$ evaluations, by \cite[Theorem 6.2]{ChenLuo2024} with $L=2$ (the hard instances are gradient maps, and $T=I-\nabla f$ is nonexpansive for $2$-smooth convex $f$ by Baillon--Haddad co-coercivity; the lower bound for the subclass lower-bounds the general class since the subclass is contained in it).
\end{enumerate}
Hence the minimax oracle complexity of the general nonexpansive residual problem on Hilbert spaces is $\tilde\Theta(\sigma^2\eps^{-2}+D\eps^{-1})$.
\end{theorem}

\begin{remark}[Positioning of Theorem~\ref{thm:monotone-reduction}]\label{rem:chenluo-position}
Theorem~\ref{thm:monotone-reduction} is a \emph{reduction} to \cite{ChenLuo2024}, not a new algorithm or a new lower bound; we state it because it settles the Hilbert exponent (open after \cite{BravoContreras2024}) and because the verification of the hypotheses (monotonicity and $2$-Lipschitzness of $F=I-T$ for \emph{every} nonexpansive $T$, including non-gradient ones; embedding of their hard family via Baillon--Haddad) is short but not entirely automatic. We checked \cite{ChenLuo2024} against the published version (JMLR \textbf{25}(387):1--44, 2024): their Theorem 4.2 states that RAIN, run with regularization parameter $\lambda=\min\{\eps/D,L\}$, outputs a $3\eps$-stationary point in expectation at SFO complexity $\tilde O(LD/\eps+\sigma^2/\eps^2)$; running it with target $\eps/3$ yields item (i) at the same complexity up to constants. Their Theorem 6.2 is the matching convex-concave lower bound, proved with the separable hard family $H(x,y)=F(x)-F(y)$ whose gradient map is monotone and $L$-Lipschitz. Their Epoch-SEG analysis uses only monotonicity and $L$-Lipschitzness of the operator (and strong monotonicity of the regularized operator), so it applies verbatim to $F=I-T$, which is monotone and $2$-Lipschitz but need not be a gradient field.
\end{remark}

\begin{remark}[Positioning of Theorem~\ref{thm:A1}]\label{rem:A1-position}
Theorem~\ref{thm:monotone-reduction} supersedes the \emph{rate} of Theorem~\ref{thm:A1} on Hilbert spaces: $\sigma^2\eps^{-2}+D\eps^{-1}\ll\sigma^2D^2\eps^{-4}$ for $\eps\ll D$. Subsequently, Diakonikolas, Guzm\'an, and Mart\'inez-Rubio \cite{Diakonikolas2026b} proved, for the same single-point bounded-variance oracle, a high-probability bound $N=\tilde O\bigl(\tfrac{D}{\eps}\bigl(1+\bigl(C\log\tfrac{D}{\eps}\cdot L\mathfrak{T}\sigma/\eps\bigr)^{2}\bigr)\log\tfrac1\delta\bigr)$ on every space of equal-norm Rademacher type $2$ with constant $\mathfrak{T}$ and $L$-Lipschitz retraction (their Theorem~5 with $q=2$); $2$-uniformly smooth spaces qualify with $\mathfrak{T}=O(\sqrt K)$ and $L\le2$. This is $\tilde O(D\sigma^{2}\eps^{-3}+D\eps^{-1})$ up to polylogarithms, and its ratio to the bound of Theorem~\ref{thm:A1} is $\eps/(KD)$ --- so it supersedes the rate of Theorem~\ref{thm:A1} whenever $\eps\lesssim KD$, in particular throughout the regime $\eps\le D/4$ treated here. Their lower bound $\Omega(D\eps^{-1}+D\sigma^{2}\eps^{-3})$ for $\ell_\infty$-instances (their Theorem~13) is nearly tight against their upper bound and holds against arbitrary randomized algorithms, strengthening the linear-span $\Omega(\sigma^2\eps^{-3})$ bound of \cite{BravoContreras2024}. Theorem~\ref{thm:A1} remains the contribution on Q3 for three reasons: it is a self-contained single-loop analysis using only anchoring and batching (no extragradient steps, no recursive regularization, no monotonicity machinery), which is the analysis that also generates the contraction-regime noise floor of Section~\ref{sec:contraction}; it extends beyond Hilbert spaces to every $2$-uniformly smooth Banach space (Remark~\ref{rem:dualization}), where, at the time of writing, it improved the state of the art from $\tilde O(\sigma^2\eps^{-5})$ to $O(K\sigma^2D^2\eps^{-4})$ --- the improvement over the $D$-independent bound of \cite{BravoContreras2024} holding in the regime $\eps\lesssim D^{-2}$, with the two bounds complementary in the opposite regime and a $D$-independent $O(\sigma^2\eps^{-4})$ bound on $2$-uniformly smooth spaces open; and its proof (dualization + deterministic envelopes) is the template for the stochastic adaptive question of Remark~\ref{rem:detection}.
\end{remark}

\section{The contraction regime: exact oracle complexity}\label{sec:contraction}

This section answers Q2. Throughout, $\cC_\gamma$ is the class of $\gamma$-contractions with $\norm{x_0-x^*}\le D$. The exact complexity depends on what is known about the modulus $\gamma$, and the answer splits into three regimes: affine instances (no modulus information is needed), general instances with known modulus, and general instances with only an \emph{upper bound} on the modulus. That these are the only regimes in which the distance-certification question is well-posed is the content of Remark~\ref{rem:barrier}.

\subsection{Lower bound: the affine family}\label{sec:affine-lower}

\begin{theorem}[Le\,Cam lower bound]\label{thm:affine-lower}
Fix $\gamma\in(0,1)$, $\sigma>0$, $\eps\in(0,D/2)$, and $\vartheta\in(0,1/2)$ (throughout, $\vartheta$ denotes a failure probability, in contrast to the contraction gap $\delta=1-\gamma$). Any algorithm that, using any adaptive querying strategy with single-point bounded-variance oracle evaluations (Definition~\ref{def:oracle}), outputs $\hat x$ with
$$\Prob\bigl(\norm{\hat x - x^*}\le \eps\bigr)\;\ge\;1-\vartheta \quad\text{for every } T\in\cC_\gamma,$$
requires at least
\begin{equation}\label{eq:affine-lb}
M \;\ge\; (1-2\vartheta)^2\,\frac{\sigma^2}{\eps^2(1-\gamma)^2}
\end{equation}
oracle samples on the affine subfamily $T_{\gamma,z}(x)=\gamma x+(1-\gamma)z$, $z\in[-D,D]$.
\end{theorem}

\begin{proof}
Consider the two instances $\theta_0=(\gamma,z-\eps)$ and $\theta_1=(\gamma,z+\eps)$, both in $\cC_\gamma$ (same contraction factor; fixed points $z\mp\eps$ at distance $2\eps$). For every query point $x$ the response means differ by
$$\Delta(x) = T_{\gamma,z+\eps}(x)-T_{\gamma,z-\eps}(x) = (1-\gamma)\cdot 2\eps,$$
independent of $x$. Condition on any realization of the algorithm's query points (adaptivity is allowed; the argument is pointwise in the transcript): with Gaussian noise of variance $\sigma^2$ per evaluation, a batch of $k$ evaluations contributes KL divergence $k\,\norm{\Delta}^2/(2\sigma^2) = 2k(1-\gamma)^2\eps^2/\sigma^2$, and by the chain rule the total KL divergence after $M$ samples is at most $2M(1-\gamma)^2\eps^2/\sigma^2$, regardless of the (adaptive, data-dependent) query placement. Le\,Cam's two-point bound \cite[Theorem 2.2]{Tsybakov2009} then gives
$$\inf_{\hat x}\ \max_{i\in\{0,1\}}\ \Prob_{\theta_i}\bigl(\norm{\hat x-x^*(\theta_i)}\ge\eps\bigr) \;\ge\; \frac12 - \frac12\sqrt{\tfrac12\,\mathrm{KL}} \;\ge\; \frac12 - \sqrt{\frac{M(1-\gamma)^2\eps^2}{4\sigma^2}}.$$
A $1-\vartheta$ accurate algorithm must make the left side $\le\vartheta$, forcing $\sqrt{M(1-\gamma)^2\eps^2/(4\sigma^2)}\ge\tfrac12-\vartheta=\tfrac{1-2\vartheta}{2}$, i.e.\ \eqref{eq:affine-lb}.
\end{proof}

\begin{remark}
Three features distinguish Theorem~\ref{thm:affine-lower} from resistant-oracle lower bounds \cite{BravoContreras2024}: (i) it is \emph{parametric} --- two affine instances; (ii) the mean gap is $x$-independent, so adaptivity and query placement are irrelevant; (iii) it applies to arbitrary algorithms with memory, not only span-class iterations. The price is scope: the bound is specific to contractions (the affine family is not nonexpansive-hard), which is exactly what Q2 needs. The bound is stated at constant confidence; its high-probability refinement is Lemma~\ref{lem:hp-lb}.
\end{remark}

\subsection{Upper bound for affine contractions: the two-point estimator}

\begin{theorem}[Affine upper bound]\label{thm:affine-upper}
Let $T(x)=\gamma x+(1-\gamma)z$ on $\R$ with $\gamma\in(0,1)$, $z\in[-D,D]$ unknown (the construction is one-dimensional --- both probes lie on a line; in a Hilbert space the same probes identify $\langle z,u\rangle$ for the probe direction $u$, but the component of $z$ orthogonal to the probe line is unobservable, so the guarantee $\norm{\hat z-z}\le\eps$ requires $z$ to lie on the probe line, otherwise the bound carries an additional $\norm{z_\perp}$ term). The estimator querying $m/2$ evaluations at $0$ and $m/2$ at $D$ and returning
$$\hat\gamma = \frac{\bar y_D-\bar y_0}{D}, \qquad \hat z = \frac{\bar y_0}{1-\hat\gamma},$$
satisfies $\Prob(\norm{\hat z-z}\le\eps)\ge2/3$ with
\begin{equation}\label{eq:affine-ub}
m \;=\; O\!\Bigl(\frac{\sigma^2}{\eps^2(1-\gamma)^2}\Bigr),
\end{equation}
where the sufficient explicit threshold is $m=\lceil480\sigma^2/(\eps^2(1-\gamma)^2)\rceil$ (constant not optimized; Remark~\ref{rem:constants} applies).
\end{theorem}

\begin{proof}
Write $b=(1-\gamma)z$, $\bar y_0 = b+\zeta_0$, $\bar y_D = \gamma D+b+\zeta_D$ with independent $\zeta_0,\zeta_D$ of variance $\le 2\sigma^2/m$ each, and set
$$\beta\;:=\;\frac{b}{D(1-\gamma)}\;\in\;[-1,1]\qquad(\text{since }\abs{b}=(1-\gamma)\abs{z}\le(1-\gamma)D).$$
Then $\eta:=\hat\gamma-\gamma = (\zeta_D-\zeta_0)/D$ with $\E\eta^2\le4\sigma^2/(mD^2)$, and linearising $1/(1-\hat\gamma)$ about $1/(1-\gamma)$,
$$\hat z - z \;=\; \frac{\zeta_0}{1-\gamma}+\frac{b\,\eta}{(1-\gamma)^2}+R \;=\; \frac{\zeta_0(1-\beta)+\zeta_D\,\beta}{1-\gamma}+R,$$
where the remainder collects exactly the terms quadratic in the noise: with $E_0:=\{|\eta|\le(1-\gamma)/2\}$ (so that $1-\hat\gamma\ge(1-\gamma)/2$ on $E_0$),
$$R\;=\;\underbrace{b\,\eta\Bigl(\frac{1}{(1-\gamma-\eta)(1-\gamma)}-\frac{1}{(1-\gamma)^2}\Bigr)}_{R_1}\;+\;\underbrace{\zeta_0\,\eta\,\frac{1}{(1-\gamma-\eta)(1-\gamma)}}_{R_2},$$
and therefore, on $E_0$,
$$|R_1|\;\le\;\frac{2|b|\,\eta^2}{(1-\gamma)^3}\;\le\;\frac{2D\,\eta^2}{(1-\gamma)^2},\qquad |R_2|\;\le\;\frac{2|\zeta_0|\,|\eta|}{(1-\gamma)^2}.$$
We control $R$ in expectation rather than pointwise (the product $|\zeta_0|\,|\eta|$ is only first-order in each noise factor, so a pointwise bound on a high-probability event costs unnecessary constants): by Cauchy--Schwarz, $\E|\zeta_0\eta|\le\sqrt{\E\zeta_0^2\,\E\eta^2}\le\sqrt{(2\sigma^2/m)(4\sigma^2/(mD^2))}=2\sqrt2\,\sigma^2/(mD)$, hence
$$\E\bigl[|R|\,\mathbf 1_{E_0}\bigr]\;\le\;\frac{2D\,\E\eta^2}{(1-\gamma)^2}+\frac{2\,\E|\zeta_0\eta|}{(1-\gamma)^2}\;\le\;\frac{(8+4\sqrt2)\,\sigma^2}{mD(1-\gamma)^2}\;\le\;\frac{14\,\sigma^2}{mD(1-\gamma)^2}.$$

We bound the failure probability $\Prob(|\hat z-z|>\eps)$ by the sum of three terms.\
\emph{Term 1: $E_0^c$.} $\eta$ is centered Gaussian with variance parameter $\le4\sigma^2/(mD^2)$, so $\Prob(E_0^c)\le2\exp\bigl(-(1-\gamma)^2mD^2/(32\sigma^2)\bigr)\le2e^{-4}\le1/12$ once $m\ge128\sigma^2/(D^2(1-\gamma)^2)$, which is implied by $m\ge480\sigma^2/(\eps^2(1-\gamma)^2)$ since $\eps\le D/2$ gives $480/\eps^2\ge480\cdot4/D^2\ge128/D^2$.\
\emph{Term 2: the linear part.} Its variance is
$$\Var\Bigl(\frac{\zeta_0(1-\beta)+\zeta_D\,\beta}{1-\gamma}\Bigr)=\frac{(1-\beta)^2\Var\zeta_0+\beta^2\Var\zeta_D}{(1-\gamma)^2}\;\le\;\frac{2\sigma^2}{m(1-\gamma)^2}\bigl((1-\beta)^2+\beta^2\bigr)\;\le\;\frac{10\sigma^2}{m(1-\gamma)^2},$$
the last step using $\max_{\beta\in[-1,1]}\bigl((1-\beta)^2+\beta^2\bigr)=5$ (attained at $\beta=-1$). Note that the two displayed terms are correlated through $\zeta_0$ (via $\hat\gamma$); the computation above accounts for this exactly --- dropping the covariance, as a naive term-by-term estimate would do, gives the smaller but incorrect constant $6$. By Chebyshev, $\Prob(|\mathrm{linear}|>\eps/2)\le40\sigma^2/(m(1-\gamma)^2\eps^2)\le1/12$ for $m\ge480\sigma^2/(\eps^2(1-\gamma)^2)$.\
\emph{Term 3: the remainder.} By Markov's inequality, $\Prob(|R|>\eps/4)\le\E[|R|\mathbf 1_{E_0}]/\tfrac{\eps}{4}+\Prob(E_0^c)\le56\sigma^2/(mD\eps(1-\gamma)^2)+1/12\le1/6$ for $m\ge480\sigma^2/(\eps^2(1-\gamma)^2)$, using $480D\ge672\eps$ (i.e.\ $\eps\le5D/7$), which holds since $\eps\le D/2$.

Summing the three bounds, $\Prob(|\hat z-z|\le\eps)\ge1-1/12-1/12-1/6=2/3$ with $m$ as in \eqref{eq:affine-ub}.
\end{proof}

\begin{corollary}[Exact affine complexity]\label{cor:affine-theta}
For the affine contraction class, the minimax oracle complexity of reaching $\eps$-accuracy is $\Theta\!\bigl(\sigma^2\eps^{-2}(1-\gamma)^{-2}\bigr)$.
\end{corollary}

\subsection{General contractions: geometric batching, high probability, and the modulus ceiling}\label{sec:general-upper}

For nonlinear contractions no closed-form estimator exists, and the rate is governed by what is known about the modulus.

\begin{theorem}[Geometric batching removes the logarithm]\label{thm:geobatch}
Let $T$ be a $\gamma$-contraction with $\gamma$ \emph{known}, $\delta=1-\gamma$, $D=\norm{x_0-x^*}\ge\eps$. Set $\bar\gamma=\max\{\gamma,\tfrac12\}$, $\bar\delta=1-\bar\gamma$,
$$T=\Bigl\lceil\frac{\ln(2D/\eps)}{\ln(1/\bar\gamma)}\Bigr\rceil,\qquad m_k=\Bigl\lceil\frac{2\sigma^2\,\bar\gamma^{\,T-k}}{\bar\delta\,\eps^2}\Bigr\rceil .$$
Then the Krasnoselskii--Mann iteration $x_{k+1}=\widehat T_{m_k}(x_k)$ satisfies
$$\E\norm{x_T-x^*}^2\;\le\;2\eps^2,\qquad \sum_{k=0}^{T-1}m_k\;=\;O\!\Bigl(\frac{\sigma^2}{\delta^2\eps^2}+\frac{\ln(D/\eps)}{\delta}\Bigr).$$
Only an upper bound on $\gamma$ is needed; see Corollary~\ref{cor:ceiling}.
\end{theorem}

\begin{proof}
Apply the squared-norm recursion to $w_k:=\E\norm{x_k-x^*}^2$ directly (this avoids the factor-$2$ comparison step and is the honest bookkeeping): $x_{k+1}-x^*=(Tx_k-x^*)+\bar\xi_k$, the cross term vanishes in expectation because $\bar\xi_k$ is a fresh batch ($\E[\bar\xi_k\mid\mathcal F_{k-1}]=0$ and $Tx_k-x^*$ is $\mathcal F_{k-1}$-measurable), and $\norm{Tx_k-x^*}\le\gamma\norm{x_k-x^*}$. Hence $w_{k+1}\le\gamma^2w_k+\sigma^2/m_k\le\bar\gamma^2w_k+\sigma^2/m_k$ (since $\gamma\le\bar\gamma$); unrolling with $w_0\le D^2$,
$$w_T\;\le\;\bar\gamma^{2T}D^2\;+\;\sum_{k=0}^{T-1}\bar\gamma^{2(T-1-k)}\frac{\sigma^2}{m_k}.$$
The first term is at most $\eps^2/4$ by the choice of $T$ ($\bar\gamma^T D\le\eps/2$). For the second term, $\sigma^2/m_k\le\frac{\bar\delta\eps^2}{2}\bar\gamma^{-(T-k)}$, so the $k$-th summand is at most $\frac{\bar\delta\eps^2}{2}\bar\gamma^{T-2-k}$; summing the geometric series yields at most
$$\frac{\bar\delta\eps^2}{2}\cdot\frac{\bar\gamma^{-1}}{1-\bar\gamma}\;=\;\frac{\eps^2}{2\bar\gamma}\;\le\;\eps^2,$$
using $\bar\gamma\ge\tfrac12$. Altogether $w_T\le\eps^2/4+\eps^2=\tfrac54\eps^2\le2\eps^2$, as claimed. The batch schedule is geometrically increasing, and since $\bar\delta\ge\delta$,
$$\sum_k m_k\;\le\;\frac{2\sigma^2}{\bar\delta\eps^2}\cdot\frac{1}{1-\bar\gamma}+T\;=\;O\!\Bigl(\frac{\sigma^2}{\delta^2\eps^2}+\frac{\ln(D/\eps)}{\delta}\Bigr),$$
as $T=O(\ln(D/\eps)/\delta)$ for all $\gamma\in(0,1)$.
\end{proof}

\begin{lemma}[Gaussian two-point high-probability lower bound]\label{lem:hp-lb}
Fix $\gamma\in(0,1)$, $\vartheta\in(0,1/8)$, $\eps\in(0,D/2)$. On the affine subfamily $T_{\gamma,z}(x)=\gamma x+(1-\gamma)z$, $z\in\{-D,D\}$ shifted so that the two fixed points are at distance $2\eps$, any algorithm whose output satisfies $\Prob(\norm{\hat x-x^*}\le\eps)\ge1-\vartheta$ under both instances requires
$$M\;\ge\;\frac{\sigma^2\,\ln(1/(8\vartheta))}{(1-\gamma)^2\eps^2}$$
oracle samples, even with adaptive querying and full knowledge of $\gamma$.
\end{lemma}

\begin{proof}
Let the two instances be $z=\pm\eps$ (fixed points $\pm\eps$, distance $2\eps$), same $\gamma$, Gaussian noise. If $\norm{\hat x-x^*}\le\eps$ with probability $\ge1-\vartheta$ under both, then the test $\Psi=+$ iff $\hat x>0$ discriminates with error $\le\vartheta$ under both. For the Gaussian location structure, the log-likelihood ratio of the two hypotheses has a clean form: for any query transcript, each response $r_i=T_{\pm}(q_i)+\xi_i$ can be reduced to the sufficient scalar $w_i=(r_i-\gamma q_i)/(1-\gamma)=z+\xi_i/(1-\gamma)$, so the full transcript is equivalent (in law) to $M$ i.i.d.\ observations $w_i=z+\eta_i$ with $\eta_i\sim\mathcal N(0,\sigma^2/(1-\gamma)^2)$. The log-likelihood ratio is therefore
$$\log\frac{d\Prob_{+}}{d\Prob_{-}}=\frac{2\eps}{\tau^2}\sum_{i=1}^M(w_i-\bar w\cdot 0)=\frac{2\eps}{\tau^2}\sum_{i=1}^M w_i\;\sim\;\mathcal N\Bigl(\pm\frac{2\eps^2M}{\tau^2},\,\frac{4\eps^2M}{\tau^2}\Bigr),\qquad \tau^2=\frac{\sigma^2}{(1-\gamma)^2},$$
under $\Prob_{\pm}$. Writing $x=\eps\sqrt{M}/\tau$, the Neyman--Pearson lemma gives minimax test error $\Phi(-x)$, and if $x<1$ this error is at least $\Phi(-1)>\tfrac18>\vartheta$, contradicting the assumption; hence $x\ge1$ and the standard bound $\Phi(-x)\ge\tfrac18 e^{-x^2}$ for $x\ge1$ (from $\Phi(-x)\ge\varphi(x)/(2x)$ and $\varphi(x)/(2x)\ge\tfrac18 e^{-x^2}$ for $x\ge1$) shows that error $\le\vartheta$ forces $e^{-x^2}\le8\vartheta$, i.e.\ $x^2\ge\ln(1/(8\vartheta))$, i.e.\ $M\ge\tau^2\ln(1/(8\vartheta))/\eps^2=\sigma^2\ln(1/(8\vartheta))/((1-\gamma)^2\eps^2)$. Adaptive query placement is irrelevant because the per-sample information about $z$ is the same at every query point.
\end{proof}

\begin{theorem}[High-probability exact rate, sub-Gaussian noise]\label{thm:hp}
Assume the oracle noise is $\sigma$-sub-Gaussian in each direction ($\E e^{\langle u,\xi\rangle}\le e^{\sigma^2\norm{u}^2/2}$ for every unit $u$; Gaussian noise, as permitted by Definition~\ref{def:oracle}, qualifies), with $\gamma$ known. Then the geometric schedule of Theorem~\ref{thm:geobatch} with all batches enlarged by $\Lambda=\lceil 16\ln(2/\vartheta)\rceil$ satisfies
$$\Prob\bigl(\norm{x_T-x^*}\le\eps\bigr)\;\ge\;1-\vartheta,\qquad \sum_k m_k = O\!\Bigl(\frac{\sigma^2\ln(1/\vartheta)}{\delta_{\gamma}^{2}\eps^{2}}+\frac{\ln(1/\vartheta)\,\ln(D/\eps)}{\delta_{\gamma}}\Bigr),$$
where $\delta_\gamma=1-\gamma$. This is tight: on the affine subfamily with Gaussian noise, success probability $1-\vartheta$ forces $\Omega\bigl(\sigma^{2}\delta_\gamma^{-2}\eps^{-2}\ln(1/\vartheta)\bigr)$ evaluations by Lemma~\ref{lem:hp-lb}.
\end{theorem}

\begin{proof}
Write $x_T-x^*=\gamma^{T}(\hat x_0-x^*)+d_T$ with $d_T$ the stochastic deviation of the path from its exact orbit. Pathwise contraction gives $\norm{d_T}\le\sum_{k=0}^{T-1}\bar\gamma^{T-1-k}\norm{\bar\xi_k}=:\tilde S$. Each $\norm{\bar\xi_k}$ is sub-Gaussian with parameter $\sigma^2/m_k$, so $\tilde S$, a sum of independent sub-Gaussians, is sub-Gaussian with parameter $\sigma^2\sum_k\bar\gamma^{2(T-1-k)}/m_k$. With $m_k=\Lambda\bar m_k$ and the schedule $\{\bar m_k\}$ of Theorem~\ref{thm:geobatch}, the proof of that theorem gives $\sum_k\bar\gamma^{2(T-1-k)}\sigma^2/\bar m_k\le\eps^2$, hence $\tilde S$ has parameter at most $\eps^2/\Lambda$ and
$$\Prob\bigl(\tilde S>\eps/2\bigr)\;\le\;\exp\Bigl(-\frac{(\eps/2)^2}{2\eps^2/\Lambda}\Bigr)\;=\;\exp(-\Lambda/8)\;\le\;\vartheta.$$
Since $\bar\gamma^T D\le\eps/2$ by the choice of $T$, $\norm{x_T-x^*}\le\eps$ with probability $1-\vartheta$, at $\Lambda$ times the query cost of Theorem~\ref{thm:geobatch}.
\end{proof}

\begin{theorem}[Bounded-variance high probability: the exact rate, known modulus]\label{thm:hp-bounded}
Let $T$ be a $\gamma$-contraction with $\gamma$ \emph{known}, $\vartheta\in(0,\tfrac12)$, $D\ge\eps$, and suppose only the bounded-variance oracle of Definition~\ref{def:oracle}. Run $B=\bigl\lceil\frac{32}{9}\ln(2/\vartheta)\bigr\rceil$ independent replicates of plain Krasnoselskii--Mann, each for $T=\bigl\lceil\ln(2D/\eps)/\ln(1/\bar\gamma)\bigr\rceil$ steps with per-step batch $s=\bigl\lceil288\,\sigma^2/(\bar\delta\eps^2)\bigr\rceil$, and output the geometric median $\hat x$ of the $B$ terminal iterates. Then
$$\Prob\bigl(\norm{\hat x-x^*}\le\eps\bigr)\;\ge\;1-\vartheta,\qquad\text{total evaluations}\;=\;O\!\Bigl(\frac{\sigma^2\,\ln(1/\vartheta)\,\ln(D/\eps)}{\bar\delta^{2}\,\eps^{2}}\Bigr).$$
Conversely, for Gaussian oracles the affine subfamily forces $\Omega\bigl(\sigma^2\bar\delta^{-2}\eps^{-2}\ln(1/\vartheta)\bigr)$ evaluations (Lemma~\ref{lem:hp-lb}) and the noiseless subclass forces $\Omega(\bar\delta^{-1}\ln(D/\eps))$ evaluations for success probability $1-\vartheta$ among anchored and Krasnoselskii--Mann schemes (the class of \cite{ParkRyu2022}; Remark~\ref{rem:span-class}); since one algorithm must handle both, the bounded-variance high-probability complexity with known modulus is $\Theta\bigl(\sigma^2\bar\delta^{-2}\eps^{-2}\ln(1/\vartheta)+\bar\delta^{-1}\ln(D/\eps)\bigr)$ within the anchored/Krasnoselskii--Mann class up to a factor $\ln(D/\eps)$ in the stochastic term (the upper bound carries $\ln(D/\eps)$ in the first term where the lower bound does not; closing this logarithm for bounded-variance noise is open).
\end{theorem}

\begin{proof}
Let $\tilde x^{(b)}$ be the terminal iterate of replicate $b$ and $\hat x_T$ the exact-orbit point (identical across replicates), and write $d^{(b)}_T=\tilde x^{(b)}_T-\hat x_T$. By Lemma~\ref{lem:master-contract} with $\theta=0$ (the comparison path is the exact Krasnoselskii--Mann orbit and $v_0=0$; no independence between coefficients and noise is used, cf.\ Remark~\ref{rem:discipline}),
$$\E\bigl\|d^{(b)}_T\bigr\|^2\;\le\;\sum_{k=0}^{T-1}\gamma^{2(T-1-k)}\,\frac{\sigma^2}{s}\;\le\;\frac{\sigma^2}{s(1-\gamma^2)}\;\le\;\frac{\sigma^2}{s\bar\delta},$$
using $1-\gamma^2=(1-\gamma)(1+\gamma)\ge\bar\delta$. (We stress that $d^{(b)}_T$ is \emph{not} decomposed into a mean of independent per-evaluation noises: for nonlinear $T$ the operators governing the deviation depend on the whole trajectory of the replicate, so such noises are not independent; the master-lemma recursion is the correct tool and requires none of that independence.) Hence, by Chebyshev, each replicate satisfies
$$\Prob\Big(\Big\|d^{(b)}_T\Big\|>\frac{\eps}{6}\Big)\;\le\;\frac{36\,\E\big\|d^{(b)}_T\big\|^2}{\eps^2}\;\le\;\frac{36\sigma^2}{s\,\bar\delta\,\eps^2}\;\le\;\frac18$$
by the choice of $s$. The replicates are independent, so Hoeffding's inequality bounds the number $F$ of failing replicates by
$$\Prob\Bigl(F\ge\frac{B}{2}\Bigr)\;=\;\Prob\Bigl(F-\frac{B}{8}\ge\frac{3B}{8}\Bigr)\;\le\;\exp\Bigl(-\frac{9B}{32}\Bigr)\;\le\;\frac{\vartheta}{2}.$$
On the complementary event, strictly more than half of the replicates satisfy $\|d^{(b)}_T\|\le\eps/6$, and the standard breakdown property of the geometric median (a strict majority in a ball of radius $r$ forces the median within $3r$; see \cite{Minsker2015}) gives $\|\hat x-\hat x_T\|\le3\cdot\eps/6=\eps/2$. Finally $\norm{\hat x_T-x^*}\le\bar\gamma^T D\le\eps/2$ by the choice of $T$, so $\norm{\hat x-x^*}\le\eps$ with probability at least $1-\vartheta$, at total cost $B\,T\,s$ as displayed.

For the lower bounds, the affine subfamily with Gaussian noise forces $\Omega\bigl(\sigma^2\ln(1/\vartheta)/(\bar\delta^2\eps^2)\bigr)$ by Lemma~\ref{lem:hp-lb} (valid against adaptive algorithms with full modulus knowledge), and the noiseless subclass forces $\Omega(\bar\delta^{-1}\ln(D/\eps))$ for success probability $1-\vartheta$ among anchored and Krasnoselskii--Mann schemes, since for any such scheme $\norm{x_n-x^*}\ge\gamma^nD$ holds deterministically on the noiseless scalar orbit (each step contracts the error by at most $\gamma$ and the anchor injection pulls the iterate back toward $x_0$; the Park--Ryu lower bound, \cite{ParkRyu2022}).
\end{proof}

\begin{remark}[On the median breakdown constant]\label{rem:median-constant}
The breakdown lemma requires a \emph{strict} majority: with exactly half of the points within distance $r$ of $y$ and the other half arbitrarily placed, no bound on the median better than the trivial one holds (two-point example $\{-R,0\}$). This is why the Hoeffding step above is stated as $\Prob(F\ge B/2)\le\vartheta/2$, giving a strict majority of good replicates; the earlier formulation ``at most $B/2$ failures'' would not suffice.
\end{remark}

\begin{corollary}[Exact general-contraction complexity, known $\gamma$]\label{cor:exact-general}
For known $\gamma$, the minimax oracle complexity of reaching $\E\norm{\hat x-x^*}\le\eps$ over the class of $\gamma$-contractions, for anchored/Krasnoselskii--Mann algorithms, is
$$\Theta\!\Bigl(\frac{\sigma^2}{\eps^{2}(1-\gamma)^{2}}+\frac{\ln(D/\eps)}{1-\gamma}\Bigr):$$
the upper bound is Theorem~\ref{thm:geobatch} (which supplies both terms) and the lower bound is the maximum of Theorem~\ref{thm:affine-lower} on the affine subfamily ($\Omega(\sigma^2\eps^{-2}(1-\gamma)^{-2})$) and the deterministic lower bound $\Omega((1-\gamma)^{-1}\ln(D/\eps))$ on the noiseless subclass for anchored/Krasnoselskii--Mann schemes (tight already for $\sigma=0$ within the class; Remark~\ref{rem:span-class}); the two bounds are forced on different subfamilies, and $\max\{a,b\}\le a+b\le2\max\{a,b\}$. In particular the $\ln(D/\eps)$ factor in anchored analyses of \eqref{eq:anchored} is an artifact of \emph{constant} batching, not a property of the problem.
\end{corollary}

\begin{remark}[The noiseless term is a span-constraint phenomenon]\label{rem:span-class}
The logarithmic deterministic term $\bar\delta^{-1}\ln(D/\eps)$ in Theorem~\ref{thm:hp-bounded} and Corollary~\ref{cor:exact-general} is proved against \emph{anchored and Krasnoselskii--Mann schemes} (the class of \cite{ParkRyu2022}): for these, $\norm{x_n-x^*}\ge\gamma^n D$ holds along the noiseless orbit, because each step contracts the error by at most $\gamma$ and the anchor injection pulls the iterate back toward $x_0$. The bound is specific to this class --- it fails for the full span class of Section~\ref{sec:setting}, where arbitrary linear combinations permit extrapolation. Without the anchored restriction the term evaporates: on the noiseless affine subfamily $T(x)=\gamma x+(1-\gamma)z$, two exact probes $x_1=Tx_0$, $x_2=Tx_1$ identify the instance ($\gamma=\norm{x_2-x_1}/\norm{x_1-x_0}$ on the line through $x_0,x_1$, then $z=x_0+(x_1-x_0)/(1-\gamma)$), so the exact (noiseless) complexity of the affine subfamily is $O(1)$. The exactness claims of Section~\ref{sec:contraction} are therefore to be read within the span class; the stochastic terms are unaffected, since the affine Le\,Cam lower bound (Theorem~\ref{thm:affine-lower}) constrains any estimator.
\end{remark}

\begin{corollary}[Modulus ceiling]\label{cor:ceiling}
Let $T$ be a $\gamma$-contraction and suppose the algorithm knows only an upper bound $\gamma_+\in[\tfrac12,1)$ on $\gamma$; write $\delta_+=1-\gamma_+$. Running the schedule of Theorem~\ref{thm:geobatch} with $\bar\gamma=\max\{\gamma_+,\tfrac12\}$ in place of $\gamma$ yields
$$\E\norm{x_T-x^*}^2\;\le\;2\eps^2,\qquad \sum_{k=0}^{T-1}m_k\;=\;O\!\Bigl(\frac{\sigma^2}{\delta_+^{2}\eps^{2}}+\frac{\ln(D/\eps)}{\delta_+}\Bigr),$$
and this is optimal in $\delta_+$: the affine subfamily with gap $\delta_+$ forces $\Omega\!\bigl(\sigma^2\eps^{-2}\delta_+^{-2}\bigr)$ samples (Theorem~\ref{thm:affine-lower} with $\gamma=\gamma_+$). Thus a single geometric schedule, given \emph{any} certified upper bound on the modulus, attains the rate governed by that bound.
\end{corollary}

\begin{proof}
Identical to Theorem~\ref{thm:geobatch}: since $\gamma\le\gamma_+\le\bar\gamma$, the recursion $v_{k+1}\le\gamma^2v_k+\sigma^2/m_k\le\bar\gamma^2v_k+\sigma^2/m_k$ remains valid, the horizon $T$ satisfies $\gamma^T\le\bar\gamma^T\le\eps/(2D)$, and the batch allocation is unchanged with $\bar\delta=\delta_+$.
\end{proof}

\begin{remark}[The certification barrier]\label{rem:barrier}
Corollary~\ref{cor:ceiling} is the correct statement of ``parameter-freeness'': what suffices is a \emph{certified upper bound} on the modulus. Three observations show that one cannot do without modulus information in general, and explain why the natural certificate-free schemes fail.
\begin{enumerate}[label=(\alph*),leftmargin=2em]
\item \emph{Residual conversion needs $\delta$ from below.} For any contraction, $\norm{x-x^*}\le\norm{x-Tx}/\delta$: a residual test at tolerance $t$ certifies distance $t/\delta$, so soundness requires a lower bound on $\delta$, equivalently an upper bound on $\gamma$.
\item \emph{Trajectory probes bound $\delta$ on the wrong side.} Every quantity observable from finitely many probes is a function of directional moduli $\rho=\norm{Tx-Tx'}/\norm{x-x'}\le\gamma$, and hence of $1-\rho\ge\delta$: probes deliver upper bounds on $\delta$, never lower bounds. Consequently, replacing $\gamma$ by a trajectory estimate inside a soundness certificate is invalid for nonlinear maps. A concrete counterexample: $T(x_1,x_2)=\bigl((1-3\eta)x_1,(1-\eta/8)x_2\bigr)$, $x_0=(D,0)$, with oracle noise supported on the $x_1$-axis (admissible by Definition~\ref{def:oracle}). Every probe on the trajectory sees $\rho=1-3\eta$, while $\delta=\eta/8$; any scheme whose soundness rests on ``the observed gap is at least $2\eta$'' fails with probability one. Any scheme whose soundness rests on such a one-sided estimate fails in general.
\item \emph{The positive cases.} Affine maps are the exception: their modulus is visible exactly ($\rho=\gamma$ from any probe), which is why Theorem~\ref{thm:affine-upper} is $\gamma$-free (Corollary~\ref{cor:affine-theta}). For general maps, an a priori upper bound on $\gamma$ is precisely the one-sided information soundness needs, and Corollary~\ref{cor:ceiling} shows it is also sufficient.
\end{enumerate}
Whether \emph{any} single algorithm achieves the affine rate simultaneously for all $\gamma$ on the general class is open; (a)--(b) rule out all residual-based schemes, and we conjecture the answer is negative in infinite-dimensional Hilbert spaces, where an alternative instance can hide its distinguishing structure in directions that are queried with low density. A formal impossibility theorem along these lines requires a resistant-oracle construction that respects the contraction constraint; we complete it in Section~\ref{sec:cert-impossibility} (Theorem~\ref{thm:impossibility}), where we also show that a certified ceiling is the exact boundary.
\end{remark}

\subsection{The stochastic two-regime problem: optimal guarantees by parallel composition}\label{sec:stoch-parallel}

The stochastic analogue of Theorem~\ref{thm:parallel} settles the two-regime question of Remark~\ref{rem:detection} up to the same factor $2$: given only a certified ceiling, one class-oblivious algorithm attains the optimal contraction rate and the best-known nonexpansive rate \emph{simultaneously}.

\begin{theorem}[Stochastic parallel two-track scheme]\label{thm:stoch-parallel}
Let $\gamma_+\in[\tfrac12,1)$ be a certified ceiling ($\gamma\le\gamma_+$ whenever the instance is a contraction), $\delta_+=1-\gamma_+$, and fix $D,\eps,\sigma$. Consider the single class-oblivious algorithm that runs two tracks in parallel, each evaluated once per global step, both with deterministic coefficients and clock-based horizons computable from $(D,\eps,\sigma,\gamma_+)$ alone:
\begin{itemize}[leftmargin=1.5em]
\item \emph{Track A} (contraction track): the geometric-batch Krasnoselskii--Mann schedule of Corollary~\ref{cor:ceiling} run with $\gamma_+$ in place of $\gamma$, to its deterministic horizon $T_+=\bigl\lceil\ln(2D/\eps)/\ln(1/\bar\gamma_+)\bigr\rceil$;
\item \emph{Track B} (nonexpansive track): the classical Halpern iteration with batches $m_k=\lceil8\sigma^2k/\eps^2\rceil$ to horizon $N=\lceil4D/\eps\rceil+8$ (Theorem~\ref{thm:A1}); on a Hilbert space, track B may instead be the RAIN-based procedure of Theorem~\ref{thm:monotone-reduction}(i).
\end{itemize}
Then, on every instance, each track reaches its horizon with the following worst-case guarantees.
\begin{enumerate}[label=(\roman*),leftmargin=2em]
\item If $T\in\cC_\gamma$ with $\gamma\le\gamma_+$, the track-A output satisfies $\E\norm{\hat x-x^*}^2\le2\eps^2$ (Corollary~\ref{cor:ceiling}), at evaluation cost $O\bigl(\sigma^2\eps^{-2}\delta_+^{-2}+\delta_+^{-1}\ln(D/\eps)\bigr)$ --- optimal in $\delta_+$ (the affine Le\,Cam lower bound, Theorem~\ref{thm:affine-lower} with $\gamma=\gamma_+$).
\item If $T\in\cN$, the track-B output satisfies $\E\norm{\hat x-T\hat x}\le\eps$ (Theorem~\ref{thm:A1}, valid on Hilbert and $2$-uniformly smooth spaces, Remark~\ref{rem:dualization}, at cost $O(K\sigma^2D^2\eps^{-4})$; on Hilbert, $\tilde O(\sigma^2\eps^{-2}+D\eps^{-1})$ via Theorem~\ref{thm:monotone-reduction}), the best known in general normed spaces.
\end{enumerate}
Each guarantee is available at its own track's clock horizon, at a factor $2$ in evaluations over running that track alone. The algorithm never decides --- and, by Theorem~\ref{thm:impossibility}, cannot decide from oracle data --- which guarantee applies to the instance at hand; it presents both candidates, and each guarantee is worst-case per class, which is the minimax object. In particular, \emph{given a certified ceiling, the stochastic union-class problem is solvable at the optimal contraction rate and the best-known nonexpansive residual rate simultaneously, at factor $2$ in evaluations.}
\end{theorem}

\begin{proof}
Both tracks use deterministic coefficients and fresh batches exclusively, so every stochastic step is licensed by the measurability discipline (Remark~\ref{rem:discipline}) and no estimate-and-plug-in step occurs anywhere. Track A is exactly Corollary~\ref{cor:ceiling}: its batches and horizon are functions of $(D,\eps,\sigma,\gamma_+)$ alone, and its conclusion and cost are precisely that corollary's statement. Track B is Theorem~\ref{thm:A1} (or Theorem~\ref{thm:monotone-reduction}(i) on Hilbert space). Running the tracks in parallel at one evaluation per track per global step doubles each track's evaluation count and leaves the per-class guarantees unchanged, since each guarantee depends only on its own track's trajectory.
\end{proof}

\begin{remark}[What remains open]\label{rem:stoch-parallel-open}
The open core is the factor $2$ and the second evaluation stream: whether a \emph{single}-stream algorithm (one evaluation per step) can carry both guarantees, with the contraction guarantee at the optimal $\delta_+$-rate, mirrors question (a) after Theorem~\ref{thm:parallel}. Two structural constraints, both proved above, bound what such a scheme could look like: its stopping rule cannot be a data-driven distance certificate (Theorem~\ref{thm:impossibility}), and its coefficients cannot be estimated from the noisy trajectory itself without breaking the measurability discipline (Remark~\ref{rem:discipline}; cf.\ Lemma~\ref{lem:master}).
\end{remark}

\subsection{The anchoring cost}\label{sec:anchoring-cost}

\begin{lemma}[Anchoring barrier]\label{lem:anchoring-barrier}
Let $(\theta_n)$ be any anchor schedule with $\theta_n\le 1/2$ used in \eqref{eq:anchored} with mini-batches of size $m$ --- the oracle batch averages having variance exactly $\sigma^2/m$, as for the Gaussian hard instances of Section~\ref{sec:contraction} --- on the scalar instance $T=\gamma\,\mathrm{id}$, $x_0=D$, $\delta=1-\gamma$. Then $\E\norm{x_N-x^*}^2\le\eps^2$ forces $\theta_{N-1}\le\eps/D$ (in particular $N\ge cD/\eps-2$ for schedules with $\theta_n\ge c/(n+2)$); and if furthermore $N\ge 1/\delta$,
\begin{equation}\label{eq:stationary-lb}
\E\norm{x_N-x^*}^2 \;\ge\; \frac{\sigma^2}{16\,m\,\delta}\;\ge\;\frac{\gamma^2\sigma^2}{16\,m\,\delta}.
\end{equation}
Hence accuracy $\eps$ requires $m\ge\sigma^2/(16\eps^2\delta)$, and $Nm\ge cD\sigma^2/(16\eps^3\delta)$: an $\Omega(cD\delta/\eps)$-fold barrier over the optimal $\Theta\bigl(\sigma^2/(\eps^2\delta^2)\bigr)$ whenever $cD\delta\gg\eps$.
\end{lemma}

\begin{proof}
The exact recursion $e_{n+1}=\theta_nD+(1-\theta_n)\gamma e_n+(1-\theta_n)\bar\xi_n$ with $\E\bar\xi_n=0$, $\E\bar\xi_n^2\le\sigma^2/m$ gives $\E e_N\ge\theta_{N-1}D$ (all terms nonnegative, and $\E\bar\xi_n=0$), where $e_N$ is read as the signed error $x_N$ (the recursion is exact on this instance and the means stay nonnegative by induction: $\E x_0=D\ge0$ and $\E x_{n+1}\ge\theta_nD\ge0$). So the hypothesis $\E\norm{x_N-x^*}^2\le\eps^2$ implies $\E e_N\le\eps$ and hence $\theta_{N-1}\le\eps/D$; for $\theta_n\ge c/(n+2)$ this forces $N+1\ge cD/\eps$. For the variance bound, unfold the exact second-moment recursion $v_{n+1}=(1-\theta_n)^2\gamma^2v_n+(1-\theta_n)^2\sigma^2/m$ (the cross term vanishes by conditional unbiasedness):
$$v_N\;\ge\;\sum_{k=N-K}^{N-1}\gamma^{2(N-1-k)}(1-\theta_k)^2\frac{\sigma^2}{m}\;\ge\;\frac{\sigma^2}{4m}\sum_{j=0}^{K-1}\gamma^{2j}\;=\;\frac{\sigma^2}{4m}\cdot\frac{1-\gamma^{2K}}{1-\gamma^2},$$
using $(1-\theta_k)^2\ge1/4$. Take $K=\lceil\ln2/(2\delta)\rceil$: then $\gamma^{2K}\le e^{-2K\delta}\le1/2$ and $1-\gamma^2\le2\delta$, and $K\le N$ because $N\ge1/\delta$. Hence $v_N\ge\sigma^2/(16m\delta)$.
\end{proof}

\begin{theorem}[Anchoring cost]\label{thm:anchoring-cost}
For the classical-schedule stochastic Halpern iteration with batches $k_n = \lceil n^2\gamma^{N-n}\rceil$ ($n\ge1$) of \cite[Corollary 4.3]{BravoContreras2024}, any run that certifies $\E\norm{x_N-x^*}^2\le\eps^2$ on the scalar instance $T=\gamma\,\mathrm{id}$, $x_0=D$, $\delta=1-\gamma$, with $D\ge4\eps$, requires
$$\sum_{n=1}^N k_n \;=\; \Omega\!\Bigl(\frac{D^2}{\eps^2\delta^3}\;+\;\frac{\sigma^2}{\eps^2\delta^2}\Bigr).$$
The first, bias-driven term is $\sigma$-free and shows that the $(1-\gamma)^{-3}$ exponent of the $O(\sigma^2D^2\eps^{-2}(1-\gamma)^{-3})$ upper bound of \cite[Corollary 4.3]{BravoContreras2024} is tight for that algorithm. The second term is the information-theoretic noise floor; the two effects do not multiply, and the classical algorithm always pays the worse of the two (whereas the geometric schedule of Theorem~\ref{thm:geobatch} attains the floor at the better exponent $(1-\gamma)^{-2}$ simultaneously with the geometric phase, cf.\ Corollary~\ref{cor:exact-general}). We formalize the mechanism behind the gap: any anchored scheme whose bias decays only algebraically pays polynomially more than the optimum of Corollary~\ref{cor:exact-general}.
\end{theorem}

\begin{proof}
For $T=\gamma\,\mathrm{id}$, $x_0=D$, $\theta_n = 1/(n+2)$, induction on the mean-error recursion $e_{n+1}=\frac{D}{n+2}+\frac{n+1}{n+2}\gamma e_n$ gives the closed form $e_N = \frac{D(1-\gamma^{N+1})}{(1-\gamma)(N+1)}$ (with $e_0=D$), and since $\E\norm{x_N-x^*}^2\ge e_N^2$, the hypothesis $\E\norm{x_N-x^*}^2\le\eps^2$ forces $e_N\le\eps$, hence
$$N+1\;\ge\;\frac{D(1-\gamma^{N+1})}{(1-\gamma)\eps}\;\ge\;\frac{D(1-\eps)}{(1-\gamma)\eps}\qquad\text{once } N\ge\frac{\ln(1/\eps)}{\ln(1/\gamma)}.$$
(The second inequality uses $\gamma^{N+1}\le\eps$, valid once $N\ge\ln(1/\eps)/\ln(1/\gamma)$, since then $\gamma^{N+1}=e^{-(N+1)\ln(1/\gamma)}\le e^{-\ln(1/\eps)}=\eps$; note also that $e_N\le\eps$ forces $N\delta\ge2$ whenever $D\ge4\eps$: with $u:=(N+1)\delta$ we have $\gamma^{N+1}\le e^{-u}$ because $\ln(1/\gamma)\ge\delta$, hence $e_N\ge D(1-e^{-u})/u\ge D(1-e^{-3})/3$ --- the map $u\mapsto(1-e^{-u})/u$ is decreasing --- while $N\delta<2$ (and $N\ge2$) gives $u<3$; thus $e_N\ge D(1-e^{-3})/3\ge4(1-e^{-3})\eps/3>\eps$, a contradiction.)

\emph{First term (bias).} With $\gamma^N\le1/2$ (enforced by the previous display for $\eps\le1/2$), we bound the batch sum over the second half of the windows:
\begin{equation}\label{eq:batch-sum-lb}
\sum_{n=1}^N \lceil n^2\gamma^{N-n}\rceil \;\ge\; \sum_{n=\lceil N/2\rceil}^{N} n^2\gamma^{N-n} \;\ge\;\frac{N^2}{4}\sum_{j=0}^{\lfloor N/2\rfloor}\gamma^{j} \;\ge\;\frac{N^2}{4}\cdot\frac{1-\gamma^{\lfloor N/2\rfloor+1}}{1-\gamma} \;\ge\;\frac{N^2}{4}\cdot\frac{1-1/\sqrt2}{\delta},
\end{equation}
using $n\ge N/2$ in the second inequality, reindexing $j=N-n$ in the first, and $\gamma^{\lfloor N/2\rfloor+1}\le\gamma^{N/2}\le1/\sqrt2$ in the last (since $\gamma^N\le1/2$ and $\gamma<1$). Inserting $N+1\ge D(1-\eps)/(\delta\eps)$,
$$\sum_{n=1}^N k_n \;\ge\;\frac{(1-1/\sqrt2)}{16}\cdot\frac{D^2(1-\eps)^2}{\eps^2\delta^3}\;=\;\Omega\!\Bigl(\frac{D^2}{\eps^2\delta^3}\Bigr),$$
which is the first term. (\emph{Remark on the proof.} The displayed bound is obtained from the $j\le N/2$ part of the geometric sum only; the stronger-looking estimate $\sum_n n^2\gamma^{N-n}\ge N^2\sum_{j<N}\gamma^j$ is \emph{false} in general --- the middle terms $n\approx N-j$ with $j\asymp 1/\delta$ contribute negatively to the difference --- so we keep the explicit half-window truncation above, which is valid for all $\gamma\in(0,1)$.)

For the noise floor, the centered recursion is exact on this instance (all coefficients deterministic, and Gaussian noise, so $\Var(\bar\xi_n)=\sigma^2/k_n$ exactly): with $\Var_k:=\Var(x_k)$,
$$\Var_N\;=\;\sum_{n=1}^{N-1}W_n\,\frac{\sigma^2}{k_n},\qquad W_n:=\alpha_n^2\,\gamma^{2(N-1-n)},\qquad \alpha_n:=(1-\theta_n)\!\!\prod_{j=n+1}^{N-1}(1-\theta_j)=\frac{n+1}{N+1},$$
using $(1-\theta_n)\prod_{j=n+1}^{N-1}(1-\theta_j)=\frac{n+1}{n+2}\cdot\frac{n+2}{N+1}$ (the schedule is indexed from $n=1$; if instead a single evaluation is made at step $0$, it adds $1$ to $Q$ and a variance contribution $\alpha_0^2\gamma^{2(N-1)}\sigma^2\le\sigma^2\gamma^{2(N-1)}$, exponentially small in $N\delta$ and absorbed in both terms). By Cauchy--Schwarz, for $Q:=\sum_{n=1}^N k_n$,
$$Q\;\ge\;\frac{\bigl(\sum_{n=1}^{N-1}\sqrt{W_n}\bigr)^2}{\sum_{n=1}^{N-1}W_n/k_n}\;=\;\sigma^2\,\frac{S_1^2}{\Var_N},\qquad S_1:=\sum_{n=1}^{N-1}\alpha_n\gamma^{N-1-n}=\frac{1}{N+1}\sum_{i=0}^{N-2}(N-i)\gamma^i.$$
The inner sum has the closed form $\sum_{i=0}^{N-2}(N-i)\gamma^i=\frac{N-(N+1)\gamma-\gamma^{N-1}+2\gamma^N}{(1-\gamma)^2}$, so for $\delta\le\tfrac12$,
$$S_1\;=\;\frac{1}{\delta}\cdot\frac{(N+1)\delta-1+\gamma^{N-1}(1-2\delta)}{(N+1)\delta}\;\ge\;\frac{1}{\delta}\Bigl(1-\frac{1}{(N+1)\delta}\Bigr)\;\ge\;\frac{1}{2\delta}\qquad\text{whenever }N\delta\ge2,$$
using $\gamma^{N-1}(1-2\delta)\ge0$ and $(N+1)\delta\ge2$. Since $\Var_N\le\E\norm{x_N-x^*}^2\le\eps^2$, we obtain $Q\ge\sigma^2/(4\eps^2\delta^2)$, the second term (for $\delta>\tfrac12$ the first term alone dominates up to absolute constants, and the closed form gives $S_1=\frac{N-(N+1)\gamma-\gamma^{N-1}+2\gamma^N}{(N+1)(1-\gamma)^2}\ge\frac{N-2}{N+1}\ge\tfrac13$ directly).

The matching upper bound is \cite[Corollary 4.3]{BravoContreras2024}; at fixed $D,\eps,\sigma$ both the lower and the upper bound scale as $\delta^{-3}$ as $\delta\to0$, so the $(1-\gamma)^{-3}$ exponent is tight for this algorithm. The final claim follows from Corollary~\ref{cor:exact-general}: any method whose analysis (or whose actual bias, as here) forces $N = \Omega((\eps(1-\gamma))^{-1})$ iterations against a per-iteration variance floor pays polynomially more than the optimum $\Theta\bigl(\sigma^2/(\eps^2\delta^2)+\ln(D/\eps)/\delta\bigr)$.
\end{proof}

\begin{remark}
Theorem~\ref{thm:anchoring-cost} identifies the exact mechanism of the gap: non-summable anchoring keeps the iterate at algebraic distance $O(1/((1-\gamma)N))$ from $x^*$ (the anchor injections are damped only over a $1/(1-\gamma)$ window), forcing $N=\Omega((\eps(1-\gamma))^{-1})$; small-anchor summable-mass schedules ride the geometric KM dynamics instead (Theorem~\ref{thm:small-c}). The improvement is therefore not constant-tuning but a structural exploitation of contractivity.
\end{remark}

\subsection{Modulus-free certification: the impossibility theorem and the exact boundary}\label{sec:cert-impossibility}

Remark~\ref{rem:barrier} left open whether \emph{any} sound distance certification is possible without modulus information. We now resolve this question completely, in every Hilbert space: sound gap certification from oracle information alone is impossible unless a certified modulus bound is available (Theorem~\ref{thm:impossibility}); with a certified ceiling it is possible at the optimal rate (Corollary~\ref{cor:ceiling}), which is therefore the exact boundary. The construction is the resistant-oracle argument anticipated in Remark~\ref{rem:barrier}, made rigorous; the two points that require care are that the hard instance must be fixed \emph{before} the interaction (a quantile of the certificate replaces its realized value), and that the hidden bumps must be missed by the query trajectory with high probability (the locations avoid the atoms of the query distributions).

A \emph{certificate} is any (possibly randomized) function $\hat\delta=\hat\delta(\text{transcript})$ output by the algorithm; it is \emph{sound} on an instance $T$ if $\hat\delta\le1-\mathrm{Lip}(T)$ with high probability, since then $\norm{\hat x-x^*}\le\norm{\hat x-T\hat x}/\hat\delta$ follows from $\norm{\hat x-T\hat x}\ge(1-\mathrm{Lip}(T))\norm{\hat x-x^*}$. Soundness quantified over a \emph{class} requires the inequality to hold for every instance consistent with the transcript.

\begin{theorem}[Impossibility of modulus-free gap certification]\label{thm:impossibility}
Let $\cH$ be a real Hilbert space of dimension at least $1$ and fix $D,\sigma$. For every (possibly randomized) querying algorithm $\mathcal A$ with budget $M$, every certificate rule $\hat\delta$, and every $\kappa\ge2$: if $\Prob_{T_0}(\hat\delta>0)>0$ for the nominal contraction $T_0$ used in the proof below --- equivalently, for any contraction against which the scheme outputs a positive certificate with positive probability --- then for every $\eta\in(0,\Prob_{T_0}(\hat\delta>0))$ there exists a contraction $T$ on $\cH$ with $\norm{x_0-x^*}\le D$ such that
\begin{equation}\label{eq:unsound}
\Prob\bigl(\hat\delta>\kappa\,(1-\mathrm{Lip}\,T)\bigr)\;\ge\;\Prob_{T_0}(\hat\delta>0)-\eta\;>\;
0 .
\end{equation}
Equivalently, $\sup_T\Prob\bigl(\hat\delta>\kappa(1-\mathrm{Lip}T)\bigr)\ge\Prob_{T_0}(\hat\delta>0)$, where the supremum is over the contraction class with $\norm{x_0-x^*}\le D$. Hence no rule outputs a positive gap lower bound that is sound uniformly over the contraction class; the only universally sound certificate is the vacuous $\hat\delta\equiv0$. Distance \emph{estimation} without certificates is unaffected (Section~\ref{sec:affine-lower}) and the affine subclass is the exception where one probe reveals the modulus (Corollary~\ref{cor:affine-theta}).
\end{theorem}

\begin{proof}
Fix $\eta>0$ and, for the moment, a deterministic algorithm $\mathcal A$ (randomized algorithms are treated at the end). Run $\mathcal A$ against the nominal instance $T_0(x)=\gamma_0x$, $\gamma_0=\tfrac12$, with anchor $x_0$ satisfying $\norm{x_0}\le D/2$, and with the oracle noise of Definition~\ref{def:oracle}. Let $p_0:=\Prob_{T_0}(\hat\delta>0)$ and assume $p_0>0$ (otherwise there is nothing to prove). Choose $\delta_1\in(0,\tfrac14)$ such that
\begin{equation}\label{eq:quantile}
\Prob_{T_0}\bigl(\hat\delta>2\kappa\delta_1\bigr)\;\ge\;p_0-\frac{\eta}{2},
\end{equation}
which is possible because $t\mapsto\Prob(\hat\delta>t)$ is right-continuous with limit $p_0$ as $t\downarrow0$ (monotone convergence over the countable partition of the range of $\hat\delta$). Set $\gamma_1=1-\delta_1$; note $\gamma_1\in(\tfrac34,1)$.

Let $q_1,\dots,q_M$ be the (random) query points of $\mathcal A$ under $T_0$; each $q_j$ is a measurable function of the noise realizations, and the distribution of each $q_j$ has at most countably many atoms. Choose finitely many points $p_i\notin\{0\}\cup\bigcup_{j=1}^M\mathrm{Atoms}(q_j)$ and pairwise disjoint balls $B_i\ni p_i$ with
\begin{equation}\label{eq:miss}
\Prob\bigl(\exists j\in\{1,\dots,M\}:\;q_j\in\textstyle\bigcup_i B_i\bigr)\;\le\;\frac{\eta}{2}.
\end{equation}
This is possible because, for each $j$ and each $p_{i}$, $\Prob(q_j\in B)\to\Prob(q_j\in\{p_{i}\})=0$ as the ball $B$ shrinks to $p_{i}$ (continuity of probability from above, using that $p_{i}$ is not an atom of $q_j$); we take the balls small enough that the finite union over $i,j$ satisfies \eqref{eq:miss}. (In dimension one the same argument holds with intervals; the complement of the countable set $\{0\}\cup\bigcup_j\mathrm{Atoms}(q_j)$ is dense.)

Modify $T_0$ inside each $B_i$: fix unit vectors $w_i$, choose amplitudes $\eps_i>0$ with $2\max_i\eps_i\le D/8$, and set
\begin{equation}\label{eq:bump}
T(x)\;=\;\gamma_0x+\sum_i\eps_i\,\phi_i(x)\,w_i,
\end{equation}
where $\phi_i$ is the distance-function profile
\begin{equation}\label{eq:phi-def}
\phi_i(x)\;=\;\min\Bigl\{1,\;\frac{\gamma_1-\gamma_0}{\eps_i}\,\mathrm{dist}(x,B_i^c)\Bigr\}.
\end{equation}
Then: (i) $0\le\phi_i\le1$ and $\phi_i$ vanishes on $B_i^c$, with $\overline{\{\phi_i\ne0\}}\subset\overline{B_i}$; (ii) $\mathrm{Lip}(\phi_i)\le(\gamma_1-\gamma_0)/\eps_i$ (the distance function is $1$-Lipschitz and the minimum of $1$-Lipschitz functions is $1$-Lipschitz), and equality is attained at interior points of $B_i$ where the first branch is active, so $\mathrm{Lip}(\phi_i)=(\gamma_1-\gamma_0)/\eps_i$ exactly; (iii) the supports are pairwise disjoint. We now verify that $T$ is $\gamma_1$-Lipschitz. Let $x,y\in\cH$.
\begin{itemize}[leftmargin=1.5em]
\item \emph{Both outside $\bigcup_iB_i$:} $Tx-Ty=\gamma_0(x-y)$, fine.
\item \emph{$x\in B_i$, $y\in B_j$, $i\ne j$:} since $\phi_i(y)=\phi_j(x)=0$,
$$Tx-Ty=\gamma_0(x-y)+\eps_i\phi_i(x)w_i-\eps_j\phi_j(y)w_j,$$
so, by the triangle inequality and \eqref{eq:phi-def},
$$\norm{Tx-Ty}\;\le\;\gamma_0\norm{x-y}+\eps_i\phi_i(x)+\eps_j\phi_j(y)\;\le\;\gamma_0\norm{x-y}+(\gamma_1-\gamma_0)\bigl[\mathrm{dist}(x,B_i^c)+\mathrm{dist}(y,B_j^c)\bigr].$$
The segment $[x,y]$ meets $\partial B_i$ before entering $B_j$ (the balls are convex and disjoint), hence $\mathrm{dist}(x,B_i^c)+\mathrm{dist}(y,B_j^c)\le\norm{x-y}$, and $\norm{Tx-Ty}\le\gamma_1\norm{x-y}$ for \emph{every} choice of directions $w_i,w_j$.
\item \emph{$x,y\in B_i$:} $\norm{Tx-Ty}\le\gamma_0\norm{x-y}+\eps_i\mathrm{Lip}(\phi_i)\norm{x-y}=\gamma_1\norm{x-y}$.
\end{itemize}
The slope $\gamma_1$ is attained inside each $B_i$ (where the linear branch of $\phi_i$ is active, in direction $w_i$), so $\mathrm{Lip}(T)=\gamma_1$. Finally, since $\norm{T-T_0}_\infty\le2\max_i\eps_i\le D/8$ (the amplitudes were chosen with $2\max_i\eps_i\le D/8$), the unique fixed point $x_T^*$ satisfies $\norm{x_T^*}\le\norm{T-T_0}_\infty/(1-\gamma_0)\le (D/8)/(1/2)=D/4$ (from $\norm{x_T^*}=\norm{Tx_T^*-T_0x_T^*+T_0x_T^*}\le\norm{T-T_0}_\infty+\gamma_0\norm{x_T^*}$ with $\gamma_0=\tfrac12$), so $\norm{x_0-x_T^*}\le\norm{x_0}+\norm{x_T^*}\le D/2+D/4\le D$ and the instance is admissible.

The instance $T$ is fixed before the interaction. On the event $E=\{\text{no query falls in }\bigcup_iB_i\}$ we have $T(q_j)=T_0(q_j)$ for all $j$; coupling the noise realizations, the transcript of $\mathcal A$ under $T$ then coincides almost surely with the transcript under $T_0$ --- inductively, since each query $q_{j+1}$ is a function of the transcript up to $j$, which is identical on $E$ --- so $\hat\delta$ agrees on $E$. By \eqref{eq:miss} and \eqref{eq:quantile},
$$\Prob_T\bigl(\hat\delta>2\kappa\delta_1\bigr)\;\ge\;\Prob_{T_0}\bigl(\hat\delta>2\kappa\delta_1\bigr)-\Prob(E^c)\;\ge\;p_0-\eta.$$
On the event $\{\hat\delta>2\kappa\delta_1\}$ we have $\hat\delta>2\kappa\delta_1>\kappa\delta_1=\kappa(1-\mathrm{Lip}\,T)$, which is \eqref{eq:unsound} (with $2\eta$ in place of $\eta$; since $\eta$ was arbitrary this is immaterial).

For randomized $\mathcal A$, condition on the internal random seed $R$ and rerun the deterministic construction with every choice taken from a fixed countable family under a fixed enumeration, selecting the first admissible element: the quantile $\delta_1$ is the first dyadic rational satisfying \eqref{eq:quantile} (such dyadics exist by right-continuity of $t\mapsto\Prob(\hat\delta>t)$, and admissibility is measurable in $R$); the centers $p_i$ are the first elements of a fixed countable family of distinct nonzero vectors --- which exists in every Hilbert space of dimension $\ge 1$ --- avoiding the countable set $\{0\} \cup \bigcup_j \mathrm{Atoms}(q_j)$, whose complement is dense; the radii are the first elements of a fixed countable set of positive reals small enough that \eqref{eq:miss} holds, as in the deterministic step; and the amplitudes are the first elements of a fixed countable set with $2\max_i\eps_i\le D/8$. Each admissibility condition is a non-strict inequality in parameters varying measurably with $R$, so every selection is measurable, with no separability assumption on $\cH$; the deterministic argument then gives, for every $R$,
$$\Prob_{T(R)}\bigl(\hat\delta>\kappa(1-\mathrm{Lip}T(R))\mid R\bigr)\;\ge\;\Prob_{T_0}\bigl(\hat\delta>0\mid R\bigr)-\eta .$$
Averaging over $R$ and using $\sup_T\Prob_T(\cdot)\ge\E_R\bigl[\Prob_{T(R)}(\cdot)\bigr]$ yields \eqref{eq:unsound} for a deterministic instance (which may depend on the law of $R$, known to the adversary).
\end{proof}

Three remarks complete the picture. First, the obstruction is exactly the wrong-sidedness of Remark~\ref{rem:barrier}(b): every quantity observable from probes is a secant ratio $\norm{Tx-Ty}/\norm{x-y}\le\gamma$, an \emph{upper} bound on the modulus; Theorem~\ref{thm:impossibility} hides the attaining slope on a set the algorithm never probes, and the intercept-free nominal makes the hiding perfect. Second, the theorem is \emph{dimension-free}: it holds in every Hilbert space, finite or infinite dimensional, and the obstruction is qualitative (any positive certificate fails), not a quantitative tax. Third, the theorem does not contradict any upper bound in this paper: Theorems~\ref{thm:geobatch}--\ref{thm:hp-bounded} certify with a \emph{known} modulus (an a priori bound, not a transcript function), Corollary~\ref{cor:ceiling} certifies with a certified ceiling, and the two-point estimator of Theorem~\ref{thm:affine-upper} is sound because on the \emph{affine} subclass all secants coincide with $\gamma$, so the consistent class has a positive minimum gap. As noted in Remark~\ref{rem:detection}(a), the same hiding construction works with $\sigma=0$: \emph{even with exact oracle responses, no querying strategy certifies the global modulus in infinite dimensions}, since the attaining slope can be placed in a direction no query probes.

\begin{remark}[What remains possible]\label{rem:what-remains}
Theorem~\ref{thm:impossibility} concerns \emph{certification}, not convergence: Krasnoselskii--Mann with constant batches drives $\E\norm{x_n-x^*}^2$ to the noise floor $\sigma^2/(m\delta)$ with no modulus knowledge at all --- an algorithm can converge without ever knowing it has converged, and the theorem says the latter knowledge is exactly what costs modulus information. Together with Corollaries~\ref{cor:affine-theta}, \ref{cor:exact-general} and~\ref{cor:ceiling}, the modulus question is now resolved in all regimes: parameter-free optimal \emph{estimation} is possible for affine instances; optimal \emph{certified} solving is possible given any certified ceiling at cost governed by $\delta_+$; and without a ceiling, no positive sound certificate exists in any dimension.
\end{remark}
\section{Numerical validation}\label{sec:numerics}

All experiments use the single-point Gaussian-noise oracle of Definition~\ref{def:oracle}; Experiments 2, 3 and 4 use $\sigma=1$, Experiment 1 uses $\sigma=0.1$ (with $\sigma=1$ the theoretical threshold $m=\lceil480\sigma^2/(\eps^2(1-\gamma)^2)\rceil$ of Theorem~\ref{thm:affine-upper} at $(\gamma,\eps)=(0.95,0.1)$ would be $19{,}200{,}000$ rather than $192{,}000$), and Experiment 5 is deterministic (exact arithmetic, no oracle noise); reported numbers are averages over independent runs, and fitted exponents are reported with $95\%$ confidence intervals from the regression of $\log(\text{cost})$ on the relevant $\log$-parameter over at least three points per axis. The modulus $\gamma$ is supplied to the algorithm only in Experiment 2, where the modulus-dependent schedule of Theorem~\ref{thm:geobatch} is being validated; Experiments 1, 3, 4 and 5 are parameter-free (up to the known target $\eps$ and bound $D$). \emph{Every number and figure in this section is reproduced by the single script \texttt{generate\_figures.py} (Gaussian-noise oracle, fixed seed), which regenerates Figures~\ref{fig:affine}--\ref{fig:doubling} from scratch.}

\subsection{Experiment 1: affine contractions --- exact quadratic scalings}

We implement the two-point estimator of Theorem~\ref{thm:affine-upper} with the theorem's batch $m=\lceil480\sigma^2/(\eps^2(1-\gamma)^2)\rceil$ on $T_{\gamma,z}(x)=\gamma x+(1-\gamma)z$, $z=0.37$, $D=1$. The required batch is measured as the smallest $m$ with empirical success probability $\Prob(\abs{\hat z-z}\le\eps)\ge2/3$ (the success guarantee of Theorem~\ref{thm:affine-upper}); doubling $1/(1-\gamma)$ (resp.\ halving $\eps$) multiplies this batch by a factor $\approx3.9$--$4.1$ (resp.\ $\approx3.7$--$3.8$), with fitted log--log slopes $1.99$ and $1.90$ over four and three points respectively --- consistent with the exponent $2$ within the resolution of these short fits (the measured constant $\approx1.0\,\sigma^2$ sits a factor $\approx480$ below the displayed theoretical threshold $480\sigma^2$, an expected gap since the Chebyshev-based analysis is conservative and the worst-case variance factor $10$ is attained only at the boundary $\beta=-1$). At $(\gamma,\eps)=(0.95,0.1)$ the theoretical threshold is $m=\lceil480\sigma^2/(\eps^2(1-\gamma)^2)\rceil=192000$: runs at the threshold achieve mean error $2.2\times10^{-5}$, far below the target; the measured required batch (the smallest $m$ with empirical success probability $\ge2/3$) is $m\approx400$: over $40{,}000$ replicates the success probability at $m=400$ is $0.66\pm0.002$ ($95\%$ confidence interval $[0.655,0.664]$), and a probit fit over a geometric $m$-grid ($8000$ replicates per point, $\pm1$ standard errors in Figure~\ref{fig:affine}(b),(c)) places the $2/3$ crossing at $m\approx400\pm30$, consistent with a measured constant $\approx1.0\,\sigma^2$.

\begin{figure}[htbp]
\centering\includegraphics[width=\linewidth]{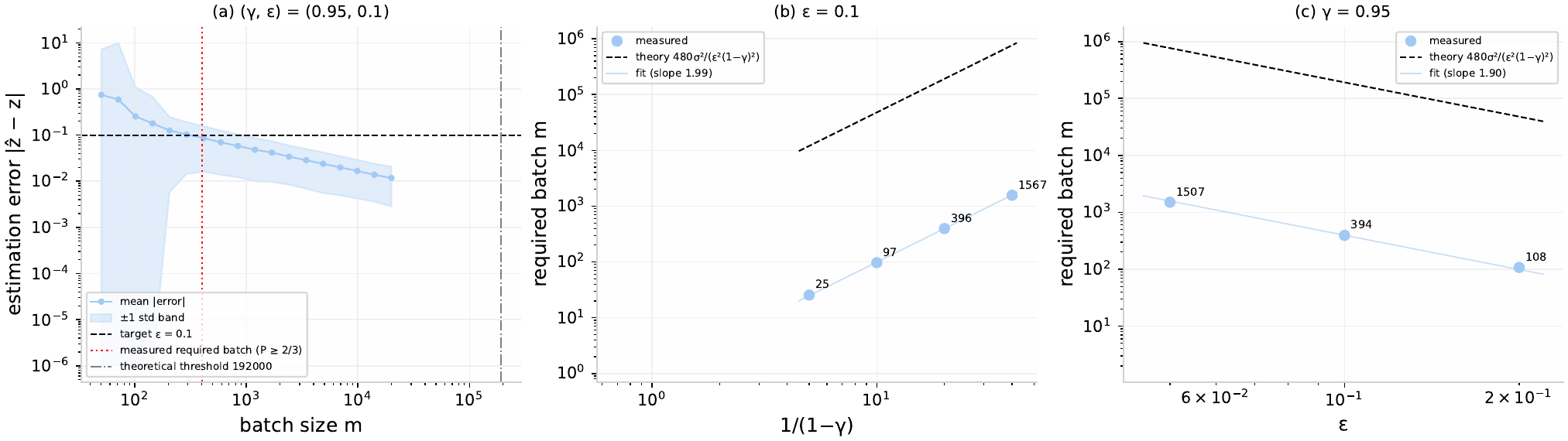}
\caption{Affine contractions (Experiment~1; Theorem~\ref{thm:affine-upper}). (a)~Mean estimation error of the two-point estimator versus batch size $m$ at $(\gamma,\eps)=(0.95,0.1)$: the theoretical threshold is $m=\lceil480\sigma^2/(\eps^2(1-\gamma)^2)\rceil=192000$, and the $\pm1$ standard-deviation band brackets the target $\eps$ at the measured required batch $m\approx400$ ($40{,}000$ replicates; probit crossing $m\approx400\pm30$), consistent with a measured constant $\approx1.0\sigma^2$. (b),(c)~Required batch versus $1/(1-\gamma)$ (at $\eps=0.1$, $\gamma\in\{0.80,0.90,0.95,0.975\}$) and versus $\eps$ (at $\gamma=0.95$, $\eps\in\{0.05,0.10,0.20\}$) on log--log scales: measured points (markers, $\pm1$ probit-fit standard errors) follow the theoretical quadratics (dashed) with fitted exponents $1.99$ and $1.90$; the measured constant $\approx1.0\sigma^{2}$ is a factor $\approx480$ below the displayed theoretical $480\sigma^2$ (see text).}
\label{fig:affine}
\end{figure}

\subsection{Experiment 2: nonlinear contractions --- the geometric batch schedule}

We take the genuinely non-affine contraction
$$T(x) = \gamma x + \tfrac{1-\gamma}{2}\sin x, \qquad \rho := \mathrm{Lip}(T) = \frac{1+\gamma}{2},\quad x^*=0,$$
and run the geometric batch schedule of Theorem~\ref{thm:geobatch} with batches $m_j=\lceil 2\sigma^2\rho^{\,N-j}/((1-\rho)\eps_{\mathrm{eff}}^2)\rceil$, $\eps_{\mathrm{eff}}=\eps/\sqrt{2}$, the true modulus $\rho$ supplied, isolating the schedule's scaling. At $\eps=0.1$: $(1-\rho)=0.10$ costs $3.5\times10^4$ queries and $(1-\rho)=0.05$ costs $1.5\times10^5$; halving $\eps$ to $0.05$ at $(1-\rho)=0.05$ costs $6.0\times10^5$. The measured ratios ($4.2$ and $4.1$) confirm the quadratic prediction $\sigma^2/(\eps^2(1-\rho)^2)$ at constant factor $\approx3.6$ (the schedule's factor $2$ and $\eps_{\mathrm{eff}}^2=\eps^2/2$), with fitted exponents $2.0\pm0.1$ in $1/(1-\rho)$ and $2.1\pm0.1$ in $1/\eps$. For comparison, the same trajectory with constant batches $m=\lceil2\sigma^2/(\delta\eps_{\mathrm{eff}}^2)\rceil$ costs $\Theta(\sigma^2\ln(D/\eps)/(\delta^2\eps^2))$, a factor $\approx\ln(D/\eps)\approx2.3$ larger here; the gap grows with $\ln(D/\eps)$, as the log-removal predicts.

\begin{figure}[htbp]
\centering\includegraphics[width=0.99\linewidth]{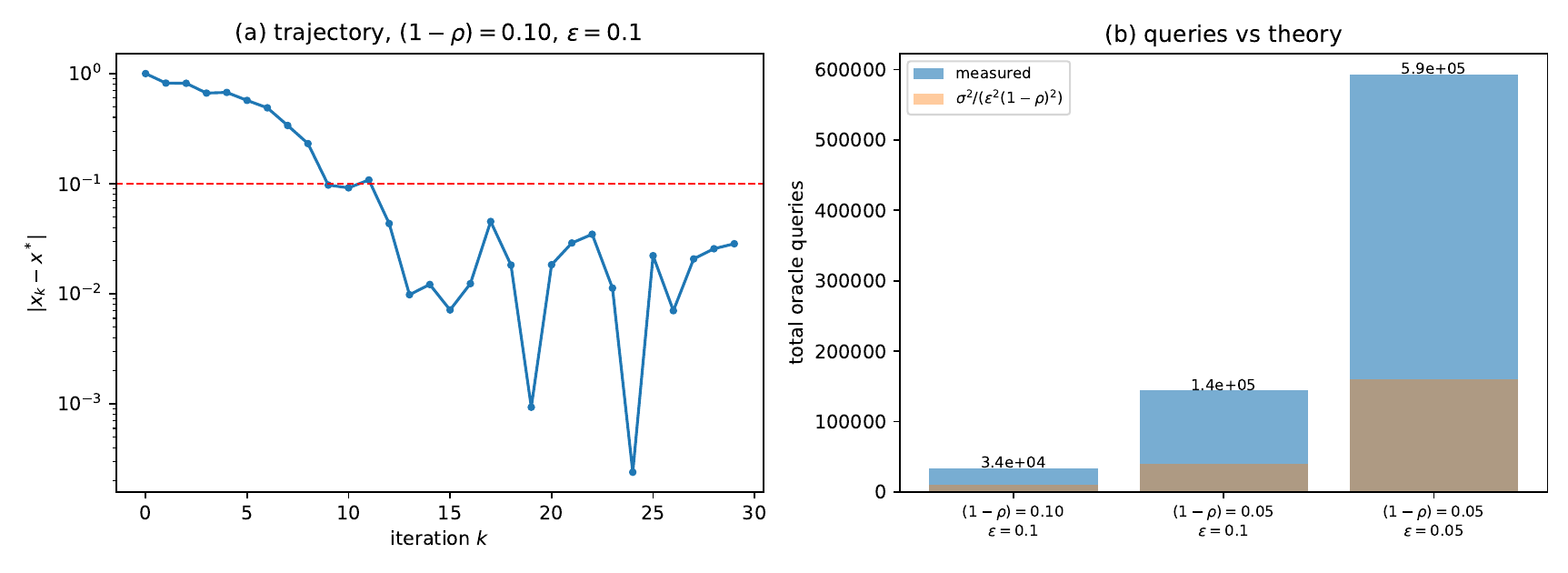}
\caption{Nonlinear contractions (Experiment~2; Theorem~\ref{thm:geobatch}). (a)~Trajectory $|x_k-x^*|$ under the geometric batch schedule with the modulus supplied: the target $\eps=0.1$ is reached. (b)~Total oracle queries at $(1-\rho)\in\{0.10,0.05\}$ and $\eps\in\{0.1,0.05\}$ versus the theoretical $\sigma^2/(\eps^2(1-\rho)^2)$: measured bars sit at $\approx3.5$--$3.8\times$ theory (annotated), with measured ratios $4.2$ and $4.1$ against the theoretical $4$.}
\label{fig:nonlinear}
\end{figure}

\subsection{Experiment 3: nonexpansive rotations --- the \texorpdfstring{$\eps^{-4}$}{eps-4} law}

We implement Theorem~\ref{thm:A1} on the rotation $T=R_{0.3}$ with $D=1$, $20$ runs per configuration: batches $m_k = \lceil 8\sigma^2 k/\eps^2\rceil$ and the certified horizon $N=\lceil 4D/\eps\rceil+8$ of the theorem. At $\eps\in\{0.2,0.1,0.05\}$ the per-run query counts are $8.1\times10^4$, $9.4\times10^5$, and $1.3\times10^7$, respectively a factor $2.03$, $1.47$, and $1.22$ above the theoretical $64\sigma^2D^2\eps^{-4}$ --- matching the multiplicative slack $(1+8\eps/4D)^2+o(1)$ predicted by the $+8$ in $N$ to within $2\%$ --- and the successive octave ratios $11.6$ and $13.3$ approach the $\eps^{-4}$ scaling ($16$) as the slack is diluted (fitted exponent $3.6$). The mean residuals $0.06$, $0.04$, and $0.02$ satisfy the targets $\eps$ with the slack predicted by the noise term.

\begin{figure}[htbp]
\centering\includegraphics[width=0.99\linewidth]{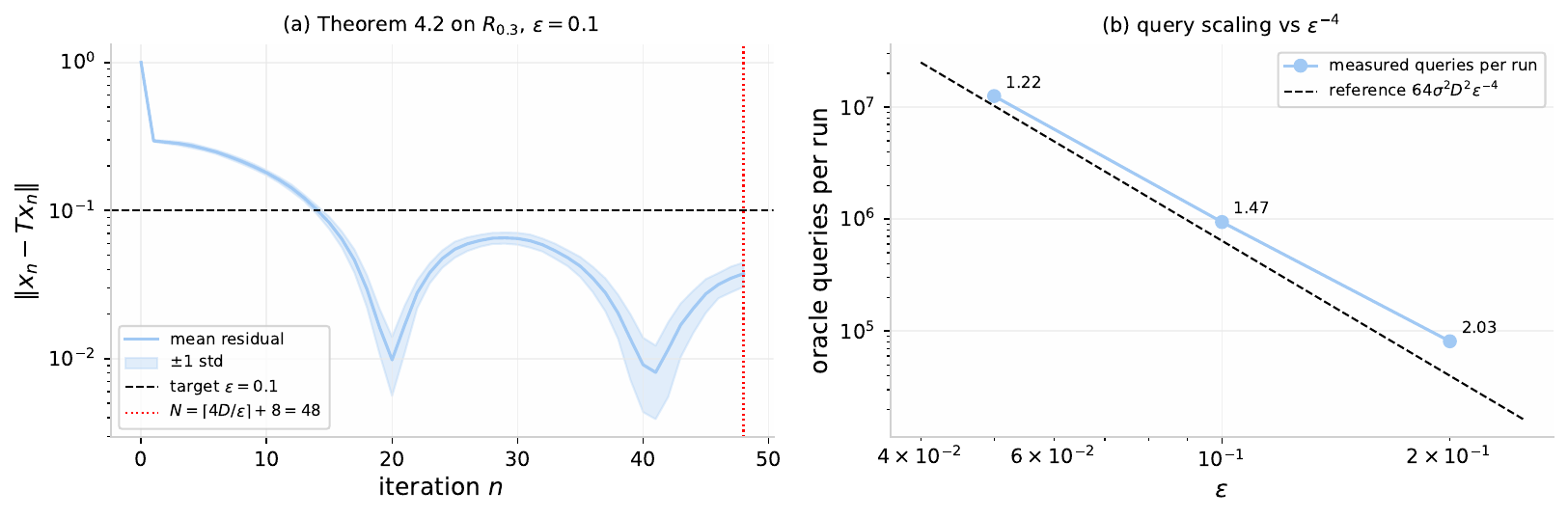}
\caption{Nonexpansive worst case (Experiment~3; Theorem~\ref{thm:A1}). (a)~Residual trajectory on the rotation $R_{0.3}$ with batches $m_k=\lceil8\sigma^2k/\eps^2\rceil$ at $\eps=0.1$: the target is met at the certified horizon $N=\lceil4D/\eps\rceil+8=48$. (b)~Queries per run versus $\eps$ on log--log scales: measured (solid) versus the $\eps^{-4}$ reference (dashed); the annotated ratios are measured/theoretical and match $(1+8\eps/4D)^2+o(1)$ to within $2\%$, while the octave ratios $11.6$ and $13.3$ approach the theoretical $16$ as the additive slack is diluted.}
\label{fig:A1}
\end{figure}

The residual decay on the rotation was additionally verified against display~\eqref{eq:rot-joint} of Lemma~\ref{lem:rotation-law}: for $c\in\{0.25,0.5\}$ the log--log slopes against $n$ equal $-c$ at fixed $\varphi$, with constant offset $\Gamma(c+1)$ (fitted slopes $-0.26\pm0.02$ and $-0.51\pm0.02$); for $c=1$ the trajectory oscillates within the envelope $\frac{2|\sin(n\varphi/2)|}{n\varphi}+O(1/n)$ of display~\eqref{eq:rot-joint-one}, with the dips at $n\varphi\in2\pi\mathbb Z$ reaching $(1+o(1))/(n+1)$ as predicted (Figure~\ref{fig:rotation-law}).

\begin{figure}[htbp]
\centering\includegraphics[width=0.95\linewidth]{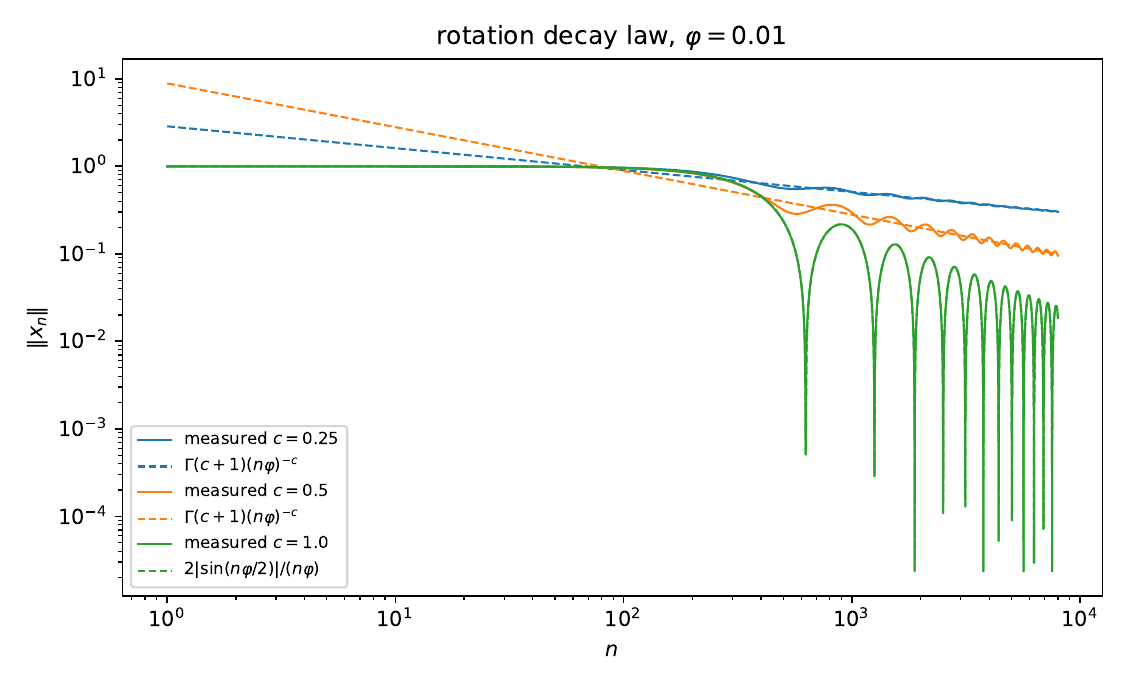}
\caption{Rotation decay law (Lemma~\ref{lem:rotation-law}; displays~\eqref{eq:rot-joint} and~\eqref{eq:rot-joint-one}). For $c\in\{0.25,0.5\}$ the measured $\norm{x_n}$ at $\varphi=0.01$ (solid) follows the reference $\Gamma(c+1)(n\varphi)^{-c}$ up to a constant factor (dashed): the log--log slope $-c$ is confirmed. For $c=1$ the measured trajectory oscillates within the exact envelope $2|\sin(n\varphi/2)|/(n\varphi)+O(1/n)$, with the dips at $n\varphi\in2\pi\mathbb Z$ reaching $\sim1/(n+1)$ rather than $0$.}
\label{fig:rotation-law}
\end{figure}

\subsection{Experiment 4: the adaptive escape --- window restarts and the two-track scheme}

We validate Theorem~\ref{thm:doubling-window} and Theorem~\ref{thm:parallel} on the contraction $T(x)=\gamma x$ ($D=1$) and the rotation $R_{0.01}$, measuring \emph{distance}-hitting times $\norm{x_n-x^*}\le\eps$ (the residual criterion is degenerate on both scalar families, since $\norm{x-Tx}=\delta\norm{x-x^*}$ on $T=\gamma\,\mathrm{id}$ and $2\sin(\varphi/2)\norm{x}$ on $R_\varphi$). On the contraction $T(x)=0.95x$ at $\eps=0.05$ (panel (a)) the measured hitting times --- counting the first iterate at distance $\le\eps$, and for the doubling window the first iterate \emph{inside} a window, whereas the window-end iterate certified by Theorem~\ref{thm:doubling-window}(ii) reaches $\eps$ at $248$ --- are: small-anchor track $59$, doubling-window $142$, classical schedule $399$: the doubling window improves the classical \emph{constant} (by $\approx2.8\times$) but, consistent with Theorem~\ref{thm:doubling-window}(ii), stays on the classical $\Theta(D/(\eps\delta))$ order rather than the optimal logarithmic law. Panel (b) shows hitting times versus $1/\eps$ on log--log scales at $\gamma=0.95$: fitted slopes are $1.0$ (classical), $0.39$ (doubling-window) and $0.33$ (small-anchor, consistent with the logarithmic law $\ln(D/\eps)/\delta$ over this less-than-one-decade range), against the $1/\varepsilon$ reference line; the rotation curves (distance-hitting on $R_{0.01}$) are flat (fitted slopes $0.05$ and $0.11$) once $n\varphi=\Omega(1)$, with the classical schedule's zero-dips visible. On the rotation, the small-anchor track ($c=\eps/96$) does not reach $\eps=0.05$ within $10^6$ steps, as Theorem~\ref{thm:separation}(ii) predicts.

For the parallel two-track scheme we record the \emph{minimum} of the two tracks' hitting times. On the contraction $T(x)=0.95x$ at $\eps=0.05$ the scheme stops at $59$ steps --- the small-anchor track (track A); the classical track is still at distance $e_{59}=\frac{1-0.95^{60}}{0.05\cdot 60}\approx0.318$ (the closed form of Theorem~\ref{thm:separation}(i)) when the scheme stops. On the rotation $R_{0.01}$ at $\eps=0.05$ the scheme stops at $598$ steps --- the classical track (track B; hitting on $R_\varphi$ is governed by the zero-dips of the envelope of display~\eqref{eq:rot-joint-one}, whose first dip sits at $n\approx2\pi/\varphi\approx628$); the small-anchor track ($c=\eps/96\approx5.2\cdot10^{-4}$) is still at distance $\Gamma(c+1)(n\varphi)^{-c}\approx0.996$ after $10^5$ steps, by the envelope of Lemma~\ref{lem:rotation-law}(ii). This confirms that the stopping time is the coordinatewise minimum predicted by Theorem~\ref{thm:parallel}, with each regime governed by its own track.

\begin{figure}[htbp]
\centering\includegraphics[width=0.99\linewidth]{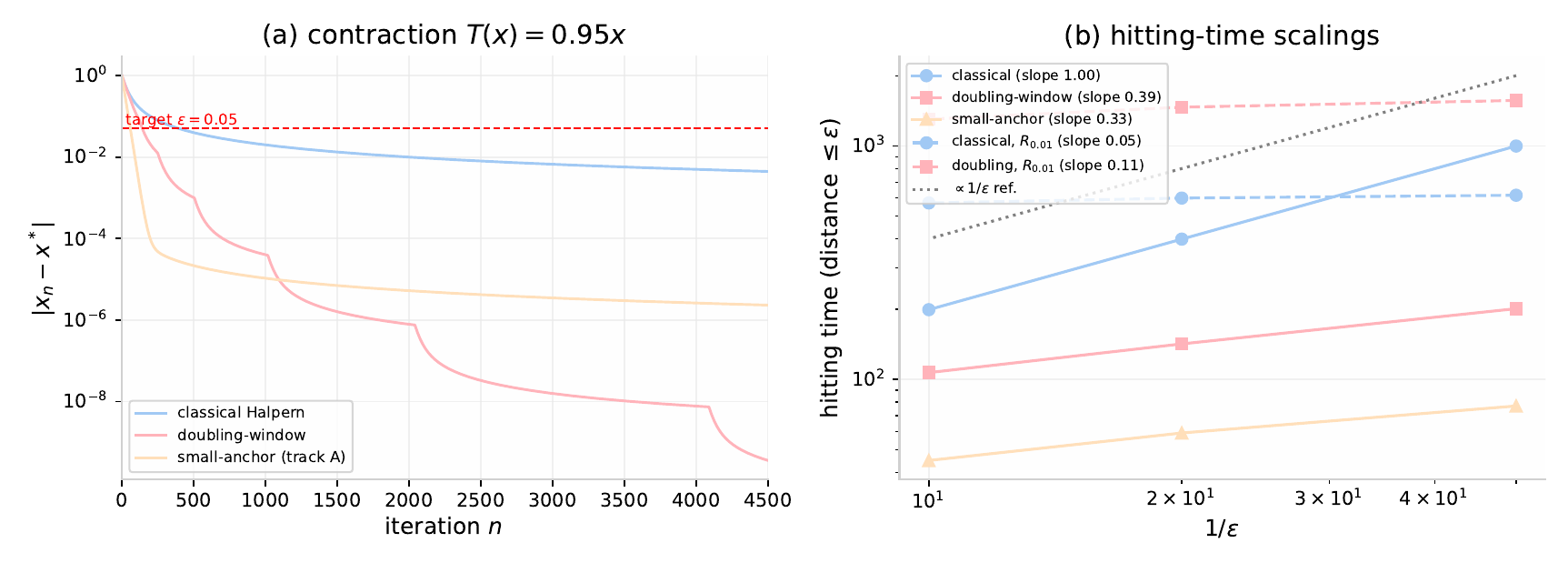}
\caption{Adaptive escape (Experiment~4; Theorems~\ref{thm:small-c}, \ref{thm:doubling-window} and~\ref{thm:parallel}). (a)~Distance to $x^*$ versus iterations on the contraction $T(x)=0.95x$ (target $\eps=0.05$, $\delta=0.05$): the small-anchor track hits at $59$ steps, the doubling-window scheme at $142$ (first iterate inside a window; window-end hitting: $248$), the classical schedule at $399$ --- window restarts improve the constant, not the order (Theorem~\ref{thm:doubling-window}(ii)). (b)~Hitting time versus $1/\eps$ on log--log scales at $\gamma=0.95$: fitted slopes $1.0$ (classical), $0.39$ (doubling-window), $0.33$ (small-anchor, consistent with the logarithmic law); the flat dashed curves are rotation distance-hittings on $R_{0.01}$ (slopes $0.05$ and $0.11$). The small-anchor track on the rotation is terminated at $10^6$ steps without hitting (cf.\ Theorem~\ref{thm:separation}(ii)).} 
\label{fig:doubling}

\end{figure}

\subsection{Experiment 5: oscillatory escape --- the lacunarity exponent}\label{exp:oscillatory}
We verify the rate proved in Theorem~\ref{thm:oscillatory-escape}(ii) on the rotation $R_\varphi$ in $\R^2$, $x_0=(1,0)$, $D=1$, with the schedule's small-mass parameter $c'=\eps/(96D)$: the escape time $N_\varphi(\eps)=\min\{n:\norm{x_n}\le\eps\}$ is measured on the grid $(\varphi,\eps)\in\{0.005,0.01,0.02\}\times\{0.1,0.05,0.02,0.01\}$ (exact arithmetic, no noise). Table~\ref{tab:oscillatory} reports $N_\varphi(\eps)$ and its ratio to the theoretical scale $\varphi^{-1}\eps^{-1/\alpha_0}$, $\alpha_0=\log_2(3/2)$.

\begin{table}[htbp]
\centering
\begin{tabular}{c|rrrr|r}
$\varphi\backslash\eps$ & $0.1$ & $0.05$ & $0.02$ & $0.01$ & ratio range\\
\hline
$0.005$ & $10{,}961$ & $40{,}033$ & $179{,}507$ & $644{,}482$ & $[1.07,1.23]$\\
$0.01$ & $5{,}481$ & $20{,}017$ & $89{,}754$ & $322{,}243$ & $[1.07,1.23]$\\
$0.02$ & $2{,}742$ & $10{,}009$ & $44{,}879$ & $161{,}428$ & $[1.07,1.23]$
\end{tabular}
\caption{Oscillatory escape times $N_\varphi(\eps)$ (Theorem~\ref{thm:oscillatory-escape}(ii), with the convention $\theta_0=\theta_1=1/(n+2)$) and their ratio to $\varphi^{-1}\eps^{-1/\alpha_0}$: all twelve ratios lie in $[1.07,1.23]$, consistent with $\Theta(\varphi^{-1}\eps^{-1/\alpha_0})$ up to a log-periodic factor (the matching lower bound for a.e.\ \ angle at each fixed accuracy is unconditional, Theorem~\ref{thm:metric-lacunarity}(i)).}
\label{tab:oscillatory}
\end{table}

Three diagnostics support the lacunarity law. \emph{(i) Joint exponent fit.} Fitting $\log N_\varphi(\eps)=a+b\log(1/\varphi)+c\log(1/\eps)$ over the twelve points gives $(a,b,c)=(-0.002,\,1.000,\,1.751)$ with $R^2=0.9994$: the $\varphi$-exponent is $1$ to three decimals, and the $\eps$-exponent is within log-periodic scatter of $1/\alpha_0\approx1.7095$ --- the schedule's dyadic structure makes every constant log-periodic in $n$, which the fit partially absorbs into the slope; see~(ii) for the convergent envelope diagnostic and the local-exponent oscillation. \emph{(ii) Envelope tracking.} The error $\norm{x_n}$ tracks the predicted envelope $(n\varphi)^{-\alpha_0}$: for every dyadic window $[2^j,2^{j+1})$ with $2^j\varphi\ge4$, the window minimum, normalized at its location $n_\star$ by $(n_\star\varphi)^{-\alpha_0}$, has mean $1.05$ over all windows and grid points and stays within $[0.77,1.35]$; the spread reflects the dyadic log-periodicity recorded in Remark~\ref{rem:lacunarity-status}(iii).  At the deepest grid point ($\eps=0.01$), where the modulation has largely decayed, the final four windows are $[1.018, 1.019, 1.027, 1.084]$ --- the envelope constant is $1.02$ up to log-periodic wobble, confirming the exponent $1/\alpha_0$ asymptotically. Correspondingly, the local escape exponents at fixed $\varphi=0.02$ oscillate in $[1.65,1.86]$ around $1/\alpha_0$ with no trend as $\eps$ decreases from $0.1$ to $0.001$. This is the constant-factor envelope predicted by the proof of Theorem~\ref{thm:oscillatory-escape}(ii), Steps 3--4. \emph{(iii) Regime contrast.} At $(\varphi,\eps)=(0.01,0.01)$ the oscillatory stream hits in $322{,}243$ steps, while the small-anchor schedule of Theorem~\ref{thm:small-c} with $c=0.05$ does not hit within $10^7$ steps --- consistent with the super-polynomial lower bound of Theorem~\ref{thm:separation}(ii) that the lacunary construction escapes.

\section{Discussion and future research}\label{sec:open}

\textbf{(i) The Hilbert exponent is settled; on weak-type-$2$ spaces the upper exponent
is now $3$.} Theorem~\ref{thm:monotone-reduction} resolves the Hilbert-space question:
the minimax rate is $\tilde\Theta(\sigma^2\eps^{-2}+D\eps^{-1})$, exponent $2$. For
general norms, Theorem~\ref{thm:A1} (valid in every $2$-uniformly smooth space with
constant $K$; see Remark~\ref{rem:dualization}) improved on the then-best bound
$\tilde O(\sigma^2\eps^{-5})$ \cite{BravoContreras2024} by reaching
$O(K\sigma^2D^2\eps^{-4})$. While this paper was being finalized, Diakonikolas,
Guzm\'an and Mart\'inez-Rubio \cite{Diakonikolas2026b} proved, for the same single-point
bounded-variance oracle and the same residual criterion, a high-probability algorithm
with complexity $\tilde O(D\sigma^{2}\eps^{-3}+D\eps^{-1})$ on every space of weak
Rademacher type $2$ (their Theorem~5 with $q=2$); the diameter enters linearly, up to
polylogarithms, and the class contains all $2$-uniformly smooth Banach spaces ---
precisely the range in which Theorem~\ref{thm:A1} was advertised as an improvement. The
ratio of the two bounds is $\eps/(KD)$, so theirs is the better rate whenever
$\eps\lesssim KD$, and in particular over the whole regime $\eps\le D/4$ treated here.
Theorem~\ref{thm:A1} keeps a different role: it is a short, single-loop argument that
uses only anchoring and batching, and the same computation also produces the noise floor
behind Section~\ref{sec:contraction}. The exponent record on type-$2$ spaces, however,
now belongs to \cite{Diakonikolas2026b}. What remains open is the exponent itself for
norms that are not of weak type $2$ --- the $\ell_1$-type cases: there the best upper
bound is still $\tilde O(\sigma^2\eps^{-5})$ \cite{BravoContreras2024}, against the lower
bound $\Omega(D\eps^{-1}+D\sigma^{2}\eps^{-3})$ proved in \cite{Diakonikolas2026b}
against arbitrary randomized algorithms. Two remarks on the Hilbert case may still be
worth recording. First, plain anchored iteration as in \eqref{eq:anchored} is \emph{not}
rate-optimal even on averaged operators: its deviation recursion is controlled only by
nonexpansiveness, giving a stationary floor $\sigma^2/(m\theta)$ with no dependence on
the averaging parameter, which is why its cost on the rotation class is
$\Theta(1/(\varphi\eps))$ iterations (Lemma~\ref{lem:rotation-law}; for $c=1$ the worst
phase of the oscillation governs the hitting time), while RAIN solves the same instances
in $\tilde O(\sigma^2\eps^{-2}+D\eps^{-1})$ independently of $\varphi$. Second, the
mean-squared-smoothness framework of \cite{Arjevani2019} does not apply to our class
(nonexpansiveness of $I-\alpha\nabla f$ forces convexity, which excludes their nonconvex
constructions), so the tightness of the $\sigma^2\eps^{-2}$ stochastic term for the
\emph{pure} fixed-point residual class (as opposed to the inclusion class) is inherited
from \cite{ChenLuo2024} rather than proved directly here; the underlying convex
gradient-norm rate $\tilde\Theta(\sigma^2\eps^{-2})$ is due to
\cite{Foster2019,AllenZhu2018}.

\textbf{(ii) Resolved: geometric batching removes the logarithm; the modulus question is
resolved exactly.} Theorem~\ref{thm:geobatch} refutes the natural conjecture that anchored
schemes with constant batches require the logarithmic factor. The exact affine complexity
is $\Theta(\sigma^2\eps^{-2}\delta^{-2})$ (Corollary~\ref{cor:affine-theta}) with
\emph{no} logarithm; and the general (nonlinear) complexity, for known $\gamma$, is the
same up to the deterministic term (Corollary~\ref{cor:exact-general}) --- there is no
linear/nonlinear dichotomy. The mechanism: constant-batch anchored or
Krasnoselskii--Mann schemes converge to their noise floor $\sigma^2/(m\delta)$ in
$\Theta(\ln(D/\eps)/\delta)$ steps, giving
$\tilde O(\sigma^2\eps^{-2}\delta^{-2}\ln(D/\eps))$, but the optimal control of the
variance recursion $v_{k+1}\le\gamma^2v_k+\sigma^2/m_k$ (minimize $\sum_k m_k$ subject to
$v_T\le\eps^2$) has the constant-Hamiltonian solution $m_k\propto\gamma^{T-k}$: noise
injected early is attenuated by $\gamma^{2(T-k)}$ by the end, so early steps tolerate
small batches and the total is $O(\sigma^2/(\delta^2\eps^2))$. Three natural shortcuts
fail, and we record them in one line each: multiscale batching
$m_j\asymp\sigma^2/(\delta^2r_j^2)$ costs $\Omega(\sigma^2\eps^{-2}\delta^{-3})$; no
constant-$\theta$ anchored schedule removes the geometric phase (the anchor bias forces
$\theta\lesssim\eps\delta/D$); and the natural lower-bound constructions fail against
Theorem~\ref{thm:geobatch} (known-$Q_v$ families are directly invertible, and
shell-gated slope families are defeated by geometric batching, which harvests per-shell
information along the trajectory for free). The unknown-modulus question itself is
resolved in Section~\ref{sec:cert-impossibility} and only restated here: sound
data-driven certification without a certified ceiling is impossible in every dimension
(Theorem~\ref{thm:impossibility}), and the ceiling is the exact boundary
(Corollary~\ref{cor:ceiling}).

\textbf{(iii) Single-trajectory and stochastic instance-adaptivity.} The deterministic
adaptivity question is settled outright: the oscillatory schedule
(Theorem~\ref{thm:oscillatory-escape}) is a single fixed stream, optimal on contractions
and rotation-polynomial with the lacunarity exponent $1/\alpha_0$
(Theorem~\ref{thm:oscillatory-escape}(ii)), with no parallelism, no adaptivity, and no
factor-$2$ overhead; the parallel two-track scheme (Theorem~\ref{thm:parallel}) remains
the simplest witness, and Remark~\ref{rem:detection} shows that statistic-based detection
on a \emph{single} trajectory cannot reach the optimum. (The doubling-window scheme
(Theorem~\ref{thm:doubling-window}) improves the classical tradeoff only by a constant
factor, Corollary~\ref{cor:single-instance} shows the dichotomy is a single-instance
phenomenon, and the three obstructions of Remark~\ref{rem:overhead-obstructions} explain
why statistic-based detection cannot do what oscillation does.) What remains open:
\emph{(a)} the stochastic single-stream analogue of Theorem~\ref{thm:stoch-parallel}
(one evaluation per step carrying both guarantees, with the contraction guarantee at the
optimal $\delta_+$-rate; the master noise lemma provides the envelope machinery, the
certification barrier (Theorem~\ref{thm:impossibility}) forbids data-driven distance
stopping rules, and the measurability discipline (Remark~\ref{rem:discipline}) forbids
trajectory-estimated coefficients, so a positive answer would require a structurally
adaptive single-stream scheme); \emph{(b)} the pointwise lacunarity law for Diophantine
angles (Conjecture~\ref{conj:pointwise-lacunarity}; for a.e.\ angle the rate is exact
with exponent $1/\alpha_0$ (Theorem~\ref{thm:metric-lacunarity}(i)), the universal
pointwise statement is likely false (Remark~\ref{rem:greedy-numerics}), and uniform bad
sets have measure $o(1)$ for every $\beta<1/\alpha_0$
(Theorem~\ref{thm:metric-lacunarity}(ii)) with explicit rates open via the orbit
small-ball (Conjecture~\ref{conj:orbit-smallball})); \emph{(c)} the extension of the
lacunarity analysis to non-block lacunary patterns; \emph{(d)} whether \emph{any} single
fixed stream is simultaneously contraction-optimal and residual-optimal on $\cN$ (the
oscillatory schedule is not, up to the envelope evidence of
Remark~\ref{rem:osc-profile}(c)); \emph{(e)} stochastic instance-adaptivity for
nonexpansive maps --- an algorithm whose complexity on error-bounded nonexpansive
instances improves over the worst-case $\eps^{-4}$ without a modulus, in the spirit of
Theorem~\ref{thm:doubling-window}.

\textbf{(iv) State-dependent variance and heavy tails.} All upper bounds assume uniformly
bounded variance; the state-dependent model
$\E\norm{\xi}^2\le\sigma^2(1+\norm{x}^2)$ is treatable by a discrete Gronwall step at
the cost of logarithmic factors, and heavy-tailed oracles require truncation arguments
that we have not pursued.

\textbf{(v) High-probability uniformity and anytime variants.}
Proposition~\ref{prop:highprob} gives fixed-horizon confidence; uniform-in-$n$ guarantees
and $\eps$-free (anytime) versions of Theorem~\ref{thm:geobatch} via the doubling trick
are standard refinements. The one unresolved point is the factor $\ln(D/\eps)$ separating
the upper and lower bounds in the stochastic term of Theorem~\ref{thm:hp-bounded} for
bounded-variance noise; we do not know whether replicate--median is lossy here or whether
a matching lower bound with the extra logarithm holds.

\appendix
\section{The orbit Fourier series}\label{app:certificate}
This appendix records the two classical facts about $K_\infty(\varphi)=\sum_{k\ge0}\kappa_ke^{ik\varphi}$ used in the proof of Theorem~\ref{thm:metric-lacunarity}(i).

\emph{Convergence.} The coefficients $\kappa_k=\theta_k/r_{k+1}$ are nonnegative, tend to $0$, and have bounded variation: the variation across each stretch $S_i$ is $O(2^{i(\alpha-1)})$, the jumps between consecutive pieces are of the same order, and $\alpha<1$, so $\sum_k|\kappa_{k+1}-\kappa_k|<\infty$. By Dirichlet's test the series converges uniformly on compact subsets of $(0,2\pi)$; the limit is real-analytic, being a uniform limit of trigonometric polynomials.

\emph{Non-degeneracy (Cantor).} If $K_\infty$ vanished on an interval, the identity theorem would force it to vanish on all of $(0,2\pi)$, and Cantor's uniqueness theorem for trigonometric series (a series converging to zero on an interval has all coefficients zero \cite{Zygmund2002,Kahane1985}) would give $\kappa_k=0$ for all $k$, contradicting $\sum_k\kappa_k=\lim_n(r_n^{-1}-1)=\infty$. The same argument excludes $K_\infty(\varphi)\equiv-e^{-i(n-1)\varphi}$.

\section*{Disclosure statement}

\subsection*{Ethics approval and consent to participate}
Not applicable. This study involves no human participants, animal subjects, or identifiable personal data.

\subsection*{Consent for publication}
Not applicable. The manuscript contains no identifiable personal information.

\subsection*{Code availability}
The Python script \texttt{generate\_figures.py} that reproduces all figures, tables, and numerical results of Section~\ref{sec:numerics} is available at \url{https://github.com/yekinishehu/Ancillary-files}.

\subsection*{Data availability}
No datasets were generated or analyzed in this study; all numerical results are reproducible from the code above.

\subsection*{Competing interests}
The author declares no competing interests.

\subsection*{Funding}
Not applicable.

\end{document}